\documentclass[letterpaper,11pt,oneside,reqno]{amsart}
\usepackage[T1]{fontenc}

\usepackage{amsfonts,amsmath, amssymb,amsthm,amscd}
\usepackage{bm}
\usepackage{mathrsfs}
\usepackage{verbatim}
\usepackage{graphicx}
\usepackage[utf8]{inputenc}
\usepackage{upgreek}
\usepackage{xcolor}
\usepackage{mathdots}
\usepackage{mathtools}
\usepackage{dsfont}
\usepackage[left=3cm, right=3cm, top=3cm, bottom=3cm]{geometry}
\usepackage{tikz}
\usepackage{lmodern}
\usepackage{caption}
\usepackage{subcaption}
\usepackage{hyperref} 

\usepackage[backend=biber,style=alphabetic, maxnames=5, maxalphanames=5, sortcites=false, giveninits=true, isbn=false, clearlang=true, doi=false, url=false, sorting=nyt, labelnumber, defernumbers]{biblatex}
\renewbibmacro{in:}{}
\AtEveryBibitem{\clearlist{language}} 
\AtEveryBibitem{\clearfield{eprintclass}}
\AtEveryBibitem{\clearfield{eprinttype}}
\AtEveryBibitem{\clearfield{note}}
\AtEveryBibitem{\clearfield{pages}}
\numberwithin{equation}{section}
\renewcommand{\emph}[1]{\textsf{\textit{#1}}}
\usepackage[width=.9\textwidth]{caption}

\newenvironment{myenumerate}{%
\renewcommand{\theenumi}{(\roman{enumi})}%
\renewcommand{\labelenumi}{\theenumi}%
\begin{list}{\labelenumi}
	{%
	\setlength{\itemsep}{0.4em}%
	\setlength{\topsep}{0.5em}%
	\setlength\leftmargin{2.45em}%
	\setlength\labelwidth{2.05em}%
	\setlength{\labelsep}{0.4em}%
	\usecounter{enumi}%
	}%
	}%
{\end{list}
}
\renewenvironment{enumerate}{
\begin{myenumerate}}%
{\end{myenumerate}}

\makeatletter
\renewcommand\section{\@startsection{section}{1}%
	\z@{1.5\linespacing\@plus.5\linespacing}{1\linespacing}%
	{\large\normalfont\scshape\centering}}
\makeatother

\newcommand{\lb}{\left(}
\newcommand{\rb}{\right)}
\newcommand{\be}{\begin{equation}}
\newcommand{\ee}{\end{equation}}
\newcommand{\EE}{\ensuremath{\mathbb{E}}}
 
\newcommand{\PP}{\ensuremath{\mathbb{P}}}

\newcommand{\C}{\ensuremath{\mathbb{C}}}

\newcommand{\la}{\lambda}
\newcommand{\ga}{\gamma}
\newcommand{\ta}{\theta}

\renewcommand{\rho}{\varrho}

\renewcommand{\leq}{\leqslant}
\renewcommand{\geq}{\geqslant}

\newcommand{\Gammainv}{\mathrm{Gamma}^{-1}}

\newcommand{\pd}{\mathcal{P}_d}
\DeclareMathOperator{\Wis}{Wishart}
\newcommand{\Wisv}{\Wis^{-1}}
\DeclarePairedDelimiter\abs{\lvert}{\rvert} 
\DeclareMathOperator{\tr}{tr}
\renewcommand \d {\mathsf{d}}
\renewcommand{\a}{\alpha}
\renewcommand{\b} {\beta}
\renewcommand{\k} {\kappa}
\newcommand{\bl} {\boldsymbol\lambda}
\newcommand{\bw} {\boldsymbol w}
\newcommand{\id}{\mathrm{id}}
\newcommand{\RR} {\mathbb{R}}
\newcommand{\ZZ}{\mathbb{Z}}
\newcommand{\strip}{\mathbb{S}_N}
\newcommand{\gp} {\mathcal{GP}}
\newcommand{\hp}{\mathcal{P}}
\newcommand{\NN}{\mathbb{N}}

\newcommand{\bB} {\mathbf{B}}
\newcommand{\bta} {\mathbf{\Theta}}
\newcommand{\bR} {\mathbf{R}}

\newcommand{\bS} {\mathbf{S}}
\newcommand{\p}{\mathbf{p}}
\newcommand{\q}{\mathbf{q}}
\renewcommand{\P}{\mathsf{P}}
\renewcommand{\t} {\theta}
\DeclareMathOperator{\wt}{wt}
\newcommand{\GL}{\mathrm{GL}}
\newcommand{\ep} {\varepsilon}
\newcommand{\U} {\mathcal{U}}
\newcommand{ \rmU }{\mathrm{U}}
\newcommand{ \hq}{\mathcal{Q}}
\newcommand{\gq} {\mathcal{GQ}}

\newcommand{\bP}{\mathbf{P}}
\newcommand{\ba} {\boldsymbol\alpha}
\newcommand{\bb} {\boldsymbol\beta}
\newcommand{\bg} {\boldsymbol\gamma}
\usetikzlibrary{patterns}
\usetikzlibrary{shapes.multipart}
\usetikzlibrary{arrows}

\newcommand{\arcleft}{\raisebox{-17pt}{\begin{tikzpicture}
        \draw[thick] (0,0.5) arc (90:270:0.5);
        \node[right] at (-.15,0.65) {\small $x$};
        \node[right] at (-.15,-0.65) {\small $y$};
        \node[left] at (-0.5,0) {\textcolor{red}{\small  $u$}};
\end{tikzpicture}}}

\newcommand{\arcright}{\raisebox{-17pt}{\begin{tikzpicture}
        \draw[thick] (0,-0.5) arc (-90:90:0.5);
        \node[left] at (.15,0.65) {\small $x$};
        \node[left] at (.15,-0.65) {\small $y$};
        \node[right] at (0.5,0) {\textcolor{red}{\small  $v$}};
\end{tikzpicture}}}

\newcommand{\bulkright}{\raisebox{-15pt}{\begin{tikzpicture}
        \draw[thick] (0,0)--(0.6,0.6);
        \node[below left] at (.15,0.05) {\small $y$};
        \node[above right] at (0.45,0.55) {\small $x$};
        \node[above] at (0.3,0.3) {\textcolor{red}{\small $\ga$}};
\end{tikzpicture}}}

\newcommand{\bulkleft}{\raisebox{-15pt}{\begin{tikzpicture}
        \draw[thick] (0.6,0)--(0,0.6);
        \node[below right] at (0.45,.05) {\small $y$};
        \node[above left] at (0.15,0.55) {\small $x$};
        \node[above] at (0.3,0.3) {\textcolor{red}{\small  $\ga$}};
\end{tikzpicture}}}

\newcommand{\bulkrightdotted}{\raisebox{-15pt}{\begin{tikzpicture}
        \draw[dashed] (0,0)--(0.6,0.6);
        \node[below left] at (.15,.05) {\small $y$};
        \node[above right] at (0.4,0.50) {\small $x$};
\end{tikzpicture}}}

\newcommand{\bulkleftdotted}{\raisebox{-15pt}{\begin{tikzpicture}
        \draw[dashed] (0.6,0)--(0,0.6);
        \node[below right] at (0.45,0.05) {\small $y$};
        \node[above left] at (0.2,0.5) {\small $x$};
\end{tikzpicture}}}

\newcommand{ \Ubw}{\U^{\begin{tikzpicture}[scale=0.2] 
\draw[dotted] (0,0) -- (1,0)--(1,1)--(0,1)--(0,0); 		
\draw[thick] (0,1) -- (0,0) -- (1,0); 
\end{tikzpicture}}}

\newcommand{\Ulw}{\U^{\begin{tikzpicture}[scale=0.2]
		\draw[dotted] (0,0) -- (0,1)--(-1,0)--(0,0);
            \draw[thick] (0,0) -- (-1,0);
		\end{tikzpicture}}}
        
\newcommand{\Urw}{\U^{\begin{tikzpicture}[scale=0.2]
		\draw[dotted] (0,0) -- (1,0)--(0,-1)--(0,0);
            \draw[ thick] (0,0) -- (0,-1);
		\end{tikzpicture}}}
        
\newcommand{\rmUbw}{\mathrm{U}^{
\begin{tikzpicture}[scale=0.2] 
\draw[dotted] (0,0) -- (1,0)--(1,1)--(0,1)--(0,0); 		
\draw[thick] (0,1) -- (0,0) -- (1,0); 
\end{tikzpicture}}}

\newcommand{\rmUlw}{\mathrm{U}^{
\begin{tikzpicture}[scale=0.2]
		\draw[dotted] (0,0) -- (0,1)--(-1,0)--(0,0);
            \draw[thick] (0,0) -- (-1,0);
		\end{tikzpicture}}}

\newcommand{\rmUrw}{\mathrm{U}^{
\begin{tikzpicture}[scale=0.2]
		\draw[dotted] (0,0) -- (1,0)--(0,-1)--(0,0);
            \draw[ thick] (0,0) -- (0,-1);
		\end{tikzpicture}}}

\newtheorem{theorem}{Theorem}[section]
\newtheorem{conjecture}[theorem]{Conjecture}
\newtheorem{lemma}[theorem]{Lemma}
\newtheorem{proposition}[theorem]{Proposition}
\newtheorem{corollary}[theorem]{Corollary}
\theoremstyle{definition}
\newtheorem{remark}[theorem]{Remark}
\theoremstyle{definition}
\newtheorem{example}[theorem]{Example}
\theoremstyle{definition}
\newtheorem{definition}[theorem]{Definition}
\theoremstyle{definition}

\title[Stationary inverse-Wishart polymers]{Stationary inverse-Wishart polymers}

\author[G.~Barraquand]{Guillaume Barraquand}
\address{G.~Barraquand, Laboratoire de Physique de l'Ecole Normale Supérieure, Ecole Normale Supérieure, PSL University, CNRS, Sorbonne Université, Université Paris-Cité, 24 rue Lhomond, 75005 PARIS}
\email{guillaume.barraquand@math.cnrs.fr} 

\author[Z.~Ouyang]{Zikun Ouyang}
\address{Z.~Ouyang, CEREMADE, CNRS, Université Paris-Dauphine, Université PSL, 75016 PARIS, FRANCE and Laboratoire de Physique de l'Ecole Normale Supérieure, Ecole Normale Supérieure, PSL University, CNRS, Sorbonne Université, Université Paris-Cité, 24 rue Lhomond, 75005 PARIS}
\email{ouyang@ceremade.dauphine.fr}
\begin{document}

		\begin{abstract} 
A solvable model of directed polymer with matrix-valued disorder is introduced in \cite{AristaBisiOConnell}. The disorder is made of $d\times d$ inverse-Wishart random matrices, so that the model nicely generalizes the well-studied log-gamma polymer, recovered when $d=1$. Much of the features of the log-gamma polymer seem to have analogues for higher $d$, although the integrability needs to be better understood. In this paper, we introduce stationary inverse-Wishart polymer models on a quadrant or a strip of $\mathbb Z^2$. In each setting, we identify stationary measures, characterized explicitly in terms of random walks with inverse-Wishart increments in special cases, or more complicated two-layer Gibbs measures for generic choices of boundary parameters. We also make conjectures about asymptotics of the free energy, and explain important differences between matrix-valued polymer models and 
their  scalar counterpart, due to non-commutativity. 
			\end{abstract}
            

	\maketitle
	\setcounter{tocdepth}{1}
	\tableofcontents

\section{Introduction and main results}
\subsection{Preface}\label{subsec:Preface}
Directed polymers are an important family of models in the Kardar-Parisi-Zhang (KPZ) universality class. In $1+1$ dimensions, the first exactly solvable example of directed polymer was introduced  in \cite{Seppalainen}, identifying a stationary structure and proving an explicit law of large numbers for the free energy. The integrability of the model was further studied  in \cite{CorwinO’ConnellSeppalainenZygouras}. The probability distribution of the partition function of the log-gamma polymer is a marginal of the Whittaker measure which may be seen as a positive temperature counterpart to the celebrated Schur measure on random partitions  \cite{Okounkov,OConnell2,BorodinCorwin}. This eventually led to   Tracy-Widom fluctuations for the free energy, confirming that the model lies in the KPZ universality class \cite{BorodinCorwinRemenik}.

\bigskip
From a physical perspective, it is natural to consider stochastic growth in contact with boundary reservoirs. Well-studied probabilistic models include the open asymmetric simple exclusion process (ASEP) \cite{MacDonaldGibbsPipkin,Liggett} and the KPZ equation on an interval \cite{CorwinShen}. For polymers, the appropriate analogue is when path are confined to  a strip \cite{KrugTang} -- see Figure \ref{fig:model on a strip}. 

\bigskip 

As is common in the study of Markov processes, a fundamental question is the characterization of stationary measures.
For the open ASEP, the stationary measure was characterized using a powerful algebraic method known as the Matrix Product Ansatz \cite{DerridaEvansHakimPasquier}. The stationary measure of the open KPZ equation was first studied in \cite{CorwinKnizel}, building on earlier works on the stationary measures of open ASEP \cite{UchiyamaSasamotoWadati,BrycWesołowski}. The result of \cite{CorwinKnizel} was later reformulated in terms of exponential functionals of two Brownian motions \cite{BrycKuznetsovWangWesołowski, BarraquandLeDoussal}. Such reweighted Brownian motions are very reminiscent of Gibbsian line ensembles arising in the KPZ class  \cite{CorwinHammond1,CorwinHammond2}. This connection motivated  \cite{BarraquandCorwinYang} to study the general two-layer Gibbsian line ensembles, so as to determine the stationary measure of solvable polymer models on a strip, and to recover predictions from \cite{BarraquandLeDoussal}.  

\bigskip
For the log-gamma polymer on the quadrant $\mathbb Z_{\geq 0}^2$, one uncovers a stationary structure by assigning  independent inverse-gamma random variables to the increments of the partition function along the first row and first column of the quadrant \cite{Seppalainen}. By employing a property of the gamma distribution known as  Lukacs' Theorem, it was shown that the stationary measure of the increment process for the partition functions along any downright path  is given by products of independent inverse-gamma variables. For the model on a strip, however, the stationary measure depends non-trivially on boundary parameters and is not, in general,   a product measure. However, it arises as a marginal of a Gibbs measure on an larger state space \cite{BarraquandCorwinYang}.  

\bigskip
From a mathematical perspective, it is natural to consider non-commutative analogues of these models. The paper \cite{AristaBisiOConnell} introduced a matrix-valued directed polymer partition function on the quadrant with inverse-Wishart disorder, generalizing the log-gamma polymer. Authors showed that the probability distribution of the partition function is related to a  matrix generalization of the  Whittaker measure. It is therefore appealing, though challenging in practice, to study this measure to understand the  asymptotics of the polymer model. Other works related to matrix-valued partition functions or exponential functionals of random walks on matrices include \cite{RiderValko,OConnell,AristaBisiO’ConnellMatsumoto–YorandDufresne} in the mathematics literature, and \cite{GrabschTexier,GautiBouchaudLeDoussal,KrajenbrinkLeDoussal} in the physics literature. In particular, \cite{KrajenbrinkLeDoussal} studied the weak noise limit of the  inverse-Wishart polymer from \cite{AristaBisiOConnell} and derived a stationary measure (which agrees with our Theorem \ref{thm:quadrant stationary measure}) using methods from physics. 
A matrix version of the stochastic heat equation, i.e. the exponential of the KPZ equation, is also introduced in \cite{KrajenbrinkLeDoussal}.   

\bigskip 
Inspired by the approach of the log-gamma polymer taken in \cite{Seppalainen}, we introduce in this paper a stationary variant of the inverse-Wishart polymer model defined in \cite{AristaBisiOConnell}. The model is stationary in the sense of Theorem \ref{thm:quadrant stationary measure} below. 
We also consider the inverse-Wishart polymer model on a strip and exhibit a stationary measure in this case as well. 
Remarkably, the method introduced in \cite{BarraquandCorwinYang}, based on so-called two-layer Gibbs measures,  works as well in the matrix setting. The model on a strip is useful in two ways: First, we derive our results for the model on the quadrant by comparison with the model on a large strip in Section \ref{subsec:Application to inhomogeneous quadrant model}. Here, our method differs from \cite{Seppalainen} as it does not seem that one can rely on matrix variants of Lukacs' theorem (see Remark \ref{rmk:matrix Lukacs theorem} for details). Second, we believe that the classification of stationary measures and the proof of  laws of large numbers, which seems to be currently out of reach, will be more tractable for the model on a strip, at least for suitably chosen boundary parameters (which we refer to as the equilibrium regime below). See Conjectures \ref{conj: LLN quadrant} and   \ref{conj:LLN strip}. 
 
\bigskip 
As the KPZ universality class is modeling out of equilibrium phenomena, stationary measures of the height function or the free energy are typically defined modulo the value at certain fixed reference point. In other terms, one defines stationary measures for the spatial height increments. We emphasize that, in contrast to the scalar case, here we study the stationary measure for the process of partition functions themselves -- not ratios of partition functions. The distinction is important  because, in the scalar case, the recurrence of the partition functions induces an autonomous Markovian dynamics for the increment process. This is no longer true for the model with matrix-valued disorder due to the non-commutativity of matrix multiplication. We will explain this issue in more detail throughout the paper (see in particular Remark \ref{rmk:markovian issue quadrant} and Remark \ref{rmk:markovian issue strip}).
Due to the out-of-equilibrium nature of the system, a true stationary probability measure 
is not expected.  
Nevertheless, we will show that there exist $\sigma$-finite measures that are stationary. 
Since the Markov processes considered are transient, we do not investigate the uniqueness of the stationary measures here. 

\bigskip 
The rest of the introduction is structured as follows. In Section \ref{subsec:Inverse-Wishart polymer with boundary conditions}, we introduce a variant of \cite{AristaBisiOConnell}'s model with boundary conditions,  and we  state  stationarity results. In particular, we describe a stationary measure as  multiplicative random walk on the positive definite matrices, which confirms predictions from \cite{KrajenbrinkLeDoussal}. In Section \ref{subsec:Back to model with delta boundary condition}, we propose a conjecture related to the law of large numbers for the model without boundary, and give an upper bound for the conjectured limiting value, based on a martingale argument. 
In Section \ref{subsec:Model on a strip}, we introduce the model on a strip. We focus on  two distinct parameter regimes and study the corresponding invariant measures  in each case.
 
\subsection{Inverse-Wishart polymer on $\mathbb Z_{\geq 0}^2$}\label{subsec:Inverse-Wishart polymer with boundary conditions}
\subsubsection{Notations} 
We fix some preliminary notations in this section and we refer to Section \ref{subsec: preliminary} for a more detailed introduction on the inverse-Wishart distribution and related special functions.
Let $\pd$ denote the set of all $d\times d$ symmetric, strictly positive definite matrices. For $x\in\pd$, we denote its determinant and trace by  $\abs{x}$ and $\tr[x]$ respectively. Together with Borel sets $\mathscr{B}(\pd)$ and the Radon measure defined by 
\be\label{eq:d mu}\d\mu (x)=\abs{x}^{-\frac{d+1}{2}}\prod_{1\leq i\leq j\leq d}dx_{i,j} ,\ee
$\lb\pd, \mathscr{B}(\pd), \d\mu\rb$
is a $\sigma$-finite measure space. Let us define the matrix product $\star$ on $\pd$ by 
$$x\star y:= y^{1/2} x y^{1/2} \qquad \text{ for all } x,y\in \pd,$$ 
where $ y^{1/2}$ refers to the unique solution $z\in\pd$ of the equation $z^2=y\in\pd$. 

A random matrix $X\sim\Wisv(\theta)$ has the (d-variate) inverse-Wishart distribution with parameter $\theta>\frac{d-1}{2}$ if it has density
    \be\label{eq: IW density}\P^{-}_{\ta}(x)=\frac{1}{\Gamma_d(\theta)}\abs{x^{-1}}^\theta e^{-\tr[x^{-1}]}\ee
    with respect to the reference measure $\d\mu$. The normalization constant  $\Gamma_d(\theta)$ is the d-variate gamma function, well defined for $\theta>\frac{d-1}{2}$: 
    \be\label{eq:d-variate gamma function}\Gamma_d(\theta)=\int_{\pd}\abs{x}^\theta e^{-\tr[x]}\d\mu(  x)=\pi^{\frac{d(d-1)}{4}}\prod_{k=1}^d \Gamma\lb \ta-\frac{k-1}{2}\rb.\ee
$X$ is said to have Wishart distribution with parameter $\ta>\frac{d-1}{2}$ if $X^{-1}\sim\Wisv(\theta)$. Its density with respect to $\d\mu$ is given by 
$\P^{+}_{\ta}(x):=\P^{-}_{\ta}(x^{-1})$.

The lack of Markovianity mentioned in  Section \ref{subsec:Preface} is related to  the following observation. Let $x,y\in\pd$,  without further assumptions, the implication 
$$x\star y^{-1}\sim\Wisv(\ta) \implies y\star x^{-1}\sim\Wis(\ta)$$
is true \textbf{if and only if} $d=1$, in which case the inverse-Wishart distribution reduces to the inverse-gamma distribution. We will explain how this is related to the  Markovianity of increments  through more concrete examples in Remark \ref{rmk:markovian issue quadrant} and Remark \ref{rmk:markovian issue strip}.

\subsubsection{The model with boundary conditions}
Let $\ZZ^2_{\geq0}=\{(n,m)\}_{n,m\geq0}$ denote the first quadrant of $\ZZ^2$. Let $\tau_l:\ZZ^2\rightarrow\ZZ^2, x\mapsto x+(l,l), l\geq 0$ be the diagonal translations on $\ZZ^2$, clearly $\ZZ^2_{\geq0}$ is preserved under such translations. We call the subset consisting of points with either  $n=0$ or $m=0$ as the left boundary or bottom boundary of the quadrant, respectively.  We denote the combined set of boundary points by the sequence 
$\hp_{b}=(\mathbf{b}_k)_{ k \in\ZZ}\in\lb\ZZ^2_{\geq0}\rb^{\ZZ}$, 
where  $\mathbf{b}_{-k}=(0,k)$ and $\mathbf{b}_k=(k,0)$ for $k\geq 0$.

Throughout the paper, we write $X \sim \mu$ to indicate that the law of $X$ is $\mu$ even when $\mu$ is a $\sigma$-finite measure, and we continue to refer to $X$ as a “random variable” or “stochastic process” when there is no ambiguity.   

\begin{definition}[Model with boundary conditions]\label{def:quadrant model with bc}Let $d\geq1$. Fix a parameter $\theta \in\RR$ with $2\ta>\frac{d-1}{2}$. Let $\lb W(n,m)\rb_{n,m\geq 1}$ be a family of independent inverse-Wishart random variables with parameter $2\theta$. Let $\bB=\lb B_k\rb_{k\in\ZZ}$ be an independent stochastic process on $\pd^{\ZZ}$. The partition functions of inverse-Wishart polymer $\lb Z^{\ta} (n,m)\rb_{n,m\geq 0}$ is given by the boundary conditions 
$$Z ^{\ta}(\mathbf{b}_k)=B_k, \quad k\in\ZZ, $$ 
and the recurrence: 
$$Z^{\ta} (n,m) =W(n,m)\star \lb Z^{\ta} (n-1,m)+Z^{\ta} (n,m-1)\rb,\quad n,m\geq 1.$$
We say that a $\sigma$-finite measure $\nu$ on $\pd^{\ZZ}$ is a stationary measure for the model if $\bB\sim\nu$ implies $\lb Z ^{\ta}(\tau_l\mathbf{b}_k)\rb_{k\in\ZZ}\sim \nu$ for all $l\geq 0$.
\end{definition}

\begin{remark}
  The partition functions $Z^{\ta} (n,m)$ depend on the boundary condition data $\bB$. Throughout the paper, we will specify explicitly the boundary conditions for different models each time to avoid ambiguity.   The model also depend on the dimension $d\geq 1$, we will often work with the model with some fixed dimension $d$ and omit it from the notation when there is no ambiguity.
 \end{remark}
 \begin{remark}
     Here we define the homogeneous model for simplicity, in the sense that the parameters for the random variables $\lb W(n,m)\rb_{n,m\geq 1}$ are the same.  We will study an inhomogeneous generalization in Section \ref{subsec:Application to inhomogeneous quadrant model} where the disorder depends on different parameters.
 \end{remark}
\subsubsection{The model of \cite{AristaBisiOConnell}}
The inverse-Wishart polymer introduced in \cite{AristaBisiOConnell} can be regarded as the model with the special boundary condition
\be\label{eq:step boundary condition}
B_{1}=\id_d, \qquad B_{k}=0_d \quad \text{for } k\neq 1.
\ee
We refer to this choice as the delta boundary condition, and we denote by $Z_d^{\ta}$ the partition function of the model with such boundary condition.  

It is of great interest to understand the asymptotics of the partition functions. For $d\geq 1$, let us define the normalized free energy
\begin{equation}
	f_d^{\ta}(n):=\frac{\log\abs{Z_{d}^{\ta}(n,n)}}{n}, \qquad n\geq 1.
	\label{eq:deffdn}
\end{equation}
As remarked in \cite{AristaBisiOConnell}, the case $d=1$ reduces to the log-gamma polymer with delta boundary condition, an integrable model which has been extensively studied. In this case, the partition function admits a compact form as a sum over paths:
\be\label{eq:sum over path}
Z_1^{\ta}(n,m)=\sum_{\pi:(1,1)\rightarrow (n,m)} \prod_{(i,j)\in \pi} W(i,j),
\ee
where the summation is over all up-right paths from $(1,1)$ to $(n,m)$. As a consequence, the sequence $\lb \log Z_1^{\ta}(n,n)\rb_{n\geq 1}$ is subadditive and hence the convergence of normalized free energy 
\be\label{eq:cv by free energy d=1}f_1^{\ta}(n)\xrightarrow{n\rightarrow\infty}
f_1^{\ta}  \qquad \text{a.s.}\ee
is assured by Kingman's subadditivity theorem.
Seppäläinen  \cite{Seppalainen} further identified the limit as
\be\label{eq:identify the limit d=1}f_1^{\ta} =-2\sup_{-\ta<u<\ta}\Big(\psi(\ta-u)+\psi(\ta+u) \Big)=-2\psi(\ta)\ee
where $\psi(\theta)$ is the digamma function. This identification was obtained by considering the model with variant boundary conditions, for which we will introduce an analogue for general $d \geq 1$. 

\subsubsection{Stationarity results for model with general reference matrix}
In the rest of the paper (except for the maximal current regime of the model on a strip that we define in \eqref{eq:parameters homogeneous general regime} and \eqref{eq:parameters inhomogeneous general regime}), we mainly consider the model with boundary data $\mathbf{B}$ given by a (possibly random) matrix at certain fixed reference point, called the reference matrix, followed by an independent forward or backward multiplicative random walk with Wishart or inverse-Wishart distribution. Precisely,

\begin{definition}[Stationary boundary conditions] 
\label{def:Stationary boundary conditions}
Let $M \geq 0$, and let $\theta, u \in \mathbb{R}$ be parameters such that $\theta + u, \theta - u > \frac{d-1}{2}$.  The stationary boundary conditions with reference point $-M$ and reference matrix $B_{-M}$ refer to the stochastic process $\mathbf{B} = (B_k)_{k \in \mathbb{Z}}$ on $ \pd^{\ZZ}$ with independent increments distributed as follows
\begin{align*}
B_{k-1}\star B_{k}^{-1} &\sim \Wisv(\t+u), \quad  k<-M;\\
    B_{k}\star B_{k-1}^{-1}&\sim \Wis(\t+u), \quad  0\geq k>-M;\\
    B_{k}\star B_{k-1}^{-1}&\sim \Wisv(\t-u), \quad k> 0,
\end{align*}
where $B_{-M}$ is a (possibly random) element of $\pd$  which is independent of the increments. We emphasize that our definition  requires $\theta>\frac{d-1}{2}$, while earlier, we had only assumed that $2\theta>\frac{d-1}{2}$.
 
Let $\nu$ be the distribution of the reference matrix. We denote the resulting stationary boundary conditions by  $\bB^{\ta,u}_{-M,\nu}$. With a slight abuse of notation, we also use $\bB^{\ta,u}_{-M,\nu}$ to denote the joint distribution of the process. We denote by $Z^{\ta,u}$ the partition function for the model with stationary boundary conditions. See Figure \ref{fig:how to sample boundary conditions} for illustration.
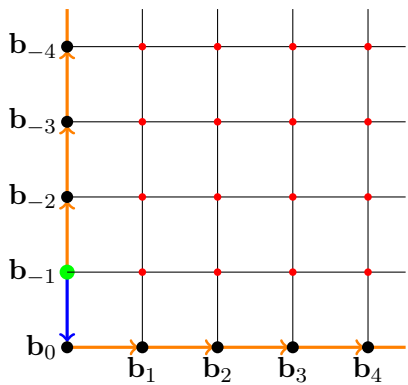
\begin{figure}
\centering
\begin{tikzpicture}[scale=1] 
            \draw[-, very thick, orange] (4,0)-- (4.5,0);
            \draw[-, very thick, orange] (0,4)-- (0,4.5);
            \foreach \x in {1, 2, 3, 4}
                \draw[->, very thick, orange] (\x-0.93,0) -- (\x-0.05,0) ;
            \foreach \y in { 2,3, 4}
                \draw[->,very thick, orange] (0,\y-0.93) -- (0,\y-0.05) ;
            \foreach \y in {1 }
                \draw[->,very thick, blue] (0,\y-0.05) -- (0,\y-0.93) ;
            \fill[green] (0,1) circle(0.1);
            \foreach \y in {0,2, 3, 4}
                \fill[black] (0,\y) circle(0.08);
            \foreach \x in {2, 1, 3, 4}
                \fill[black] (\x,0) circle(0.08);
            \node[left] at (0,0) {$ \mathbf{b}_{0}$};
            \node[left] at (0,1) {$ \mathbf{b}_{-1}$};
            \node[left] at (0,2) {$ \mathbf{b}_{-2}$};
            \node[left] at (0,3) {$ \mathbf{b}_{-3}$};
            \node[left] at (0,4) {$ \mathbf{b}_{-4}$};
\node[below] at (1,0) {$\mathbf{b}_{1}$};
\node[below] at (2,0) {$\mathbf{b}_{2}$};
\node[below] at (3,0) {$\mathbf{b}_{3}$};
\node[below] at (4,0) {$\mathbf{b}_{4}$};
            \foreach \x in {1, 2, 3, 4}
                \draw (\x, 0) -- (\x,4.5) ;
            \foreach \y in {1, 2, 3, 4}
                \draw (0, \y) -- (4.5, \y);
            \foreach \x in {1, 2, 3, 4} 
                \foreach \y in {1, 2, 3, 4}
                \fill[red] (\x,\y) circle(0.05);        
\end{tikzpicture}
\caption{Model with stationary boundary conditions. We sample the partition functions along the boundary recursively. First sample the partition function at the reference point
$(0,M)$  (in green) according to the law of the reference matrix $B_{-M}$ (in our example, $M=1$). An arrow $\p\rightarrow \q$ indicates that, $Z^{\ta,u}(\q)$ is sampled by left-$\star$ multiplying $Z^{\ta,u}(\p)$ with an independent inverse-Wishart (in orange) or Wishart (in blue) random variable. } 
    \label{fig:how to sample boundary conditions}
\end{figure}
\end{definition}

\bigskip 
In order to describe our first stationarity result, we need some further definitions.
A down-right path of length $N$ in the quadrant refers to a sequence $\hp=(\p_k)_{0\leq k \leq N}\in\lb\ZZ^2_{\geq0}\rb^{N+1}$ joining a vertex $\p_0$ on the left boundary   with a vertex $\p_N$ on the bottom boundary, such that the increments $\p_{k}-\p_{k-1}$ are either $(1,0)=``\rightarrow"$ or $(0,-1)=``\downarrow"$. The shape of $\hp$ can thus be recorded as a word 
$\bw(\hp)\in\left\{\rightarrow,\downarrow\right\}^{N}$.  

For a stochastic process $\bS=(S_t)_{t\in T}$ indexed by a countable index set $T$ and some $k\in T$, we denote by $\hat{\bS}^{(k)}=(S_t)_{t\in T, t\neq k}$ the marginal process omitting the coordinate indexed by $k$, and we always identify $\bS$ with the pair $(S_k, \hat{\bS}^{(k)})$.

\begin{definition}[$\Wis^{\pm}$ random walk]\label{def:IW random walk}
Let $N\in \mathbb{N}$ and 
$\bw=(w_1,\dots,w_N)\in\{ \rightarrow, \downarrow\}^{N}$ be a word of length $N$. Let $\bg=(\gamma_1,\dots,\gamma_N)$ be  parameters with $\gamma_k>\frac{d-1}{2}$ for all $ 1\leq k \leq N$.  
 The $\Wis^{\pm}$ random walk starting from $S_0$
refers to the stochastic process 
$\bS=\lb S_k\rb_{0\leq k\leq N} \in\pd^{N+1}$ 
with independent Wishart or inverse-Wishart increments:
$$
S_k\star S_{k-1}^{-1}\sim \left\{  \begin{aligned}
&\Wisv(\gamma_k) \quad & \text{if } w_k=\rightarrow \\
&\Wis(\gamma_k) \quad  &  \text{if } w_k=\ \downarrow 
\end{aligned}
\right., \quad 1\leq k\leq N.$$
where $S_0\in\pd$ is a (possibly random) matrix, called the starting matrix, which is independent of the increments.
Identify  $\bS=(S_0, \hat{\bS}^{(0)})$ where  $\hat{\bS}^{(0)}=\lb S_k\rb_{1\leq k\leq N}$. For every realization of $S_0$, we denote the marginal distribution of $\hat{\bS}^{(0)}$  by 
$ \d\bR^{\bg,\bw}_{S_0} $, that is, the probability measure on $\pd^N$ with density 
$$\bR^{\bg,\bw}_{S_0} (\hat{\bS}^{(0)})= \prod\limits_{\substack{1\leq k\leq N\\ w_k=\rightarrow}}  \Big(\P^{-}_{\ga_k}(S_k\star S_{k-1}^{-1})\Big) 
\prod\limits_{\substack{1\leq k\leq N\\ w_k=\downarrow}}
\Big(\P^{+}_{\ga_k}(S_k\star S_{k-1}^{-1})\Big) $$
with respect to the reference measure $\d\mu(\hat{\bS}^{(0)})=\prod_{1\leq k\leq N} \d\mu(S_k)$.

In particular, when $\bw = (\downarrow)^N$ (resp.\ $\bw = (\rightarrow)^N$), we refer to the process as the ($N$-steps) Wishart (resp.\ inverse-Wishart) random walk starting from $S_0$, and denote the marginal law of the last $N$ coordinates by $\d\bR^{\bg,\downarrow}_{S_0}(\hat{\bS})$ (resp.\ $\d\bR^{\bg,\rightarrow}_{S_0}(\hat{\bS})$).
\end{definition}

For the model with stationary boundary conditions $Z^{\ta,u}(n,m)$, the independence structure of the increments along the boundary propagates into the bulk vertices that are below the reference point. 
\begin{proposition}
\label{prop:ratio homogeneous quadrant} 
Consider the model with stationary boundary conditions $\bB^{\ta,u}_{-M,\nu}$ for arbitrary $\sigma$-finite measure $\nu$ on $\pd$. Then the joint law of partition functions along any down-right path $\hp=(\p_k)_{0\leq k\leq N}$ starting from $\p_0=(0,M)$ is a $\Wis^{\pm}$ random walk with $\nu$-distributed starting matrix. Precisely,  
    $$\Big( Z^{\ta,u}(\p_k)\Big) _{0\leq k\leq N}\sim \d\nu\Big( Z^{\ta,u}(\p_0)\Big) \times 
    \d\bR^{\bg(\hp),\bw(\hp)}_{Z^{\ta,u}(\p_0)} 
    \Big( 
    \big( Z^{\ta,u}(\p_k)\big)_{0\leq k\leq N} \Big),$$
    where $\bg=(\gamma_i)_{1\leq i\leq N}$ is such that
    \be\label{eq:labels homogeneous quadrant model}\gamma_i= \left\{  \begin{aligned}
\ta-u \quad & \text{if } \p_{i}-\p_{i-1}= (1,0) \\
\ta+u \quad & \text{if } \p_{i}-\p_{i-1}=(0,-1) 
\end{aligned}
\right., \quad 1\leq i\leq N.\ee
See Figure \ref{subfig:propagation of boundary conditions} for illustration. As a consequence, when $\nu$ is a probability measure, we have for all $n\in\NN$ and $ m\leq M$:
    \be\label{eq:expectation of partition function}
    \EE\left[\log\frac{\abs{Z^{\ta,u}(n,m)}}{\abs{Z^{\ta,u}(0,0)}}\right]=-n\psi_d(\ta-u)-m\psi_d(\ta+u),\ee
    where $\psi_d(\theta)$ is the d-variate digamma function $\psi_d(\theta):=\frac{\d}{\d\theta}\log \Gamma_d(\theta)=-\EE\big[\log\abs{\Wisv(\theta)}\big]$.
\end{proposition}

Proposition \ref{prop:ratio homogeneous quadrant} is a consequence of the slightly more general  Proposition \ref{prop:finite stationary measure for quadrant with boudnary} below, and is proven using the one-layer Gibbs measure formalism in Section \ref{subsec:Application to inhomogeneous quadrant model}. In the scalar case, the statement of the analogous result  \cite[Theorem 3.3 and Eq. (2.5)]{Seppalainen} is much more appealing.  Unfortunately, in the matrix case,  the constraints $\mathbf{p}_0 = (0, M)$ and $m \leq M$ are essential and cannot be relaxed when $d \geq 2$. Let us illustrate this through a minimal example.  
     
\begin{example}\label{eg:concret example}
Considering the simplest case of Proposition \ref{prop:ratio homogeneous quadrant} where $M=m=n=1$, which we refer to as the `one-step updates'.  Let  $\theta, u\in\RR$ be parameters such that  $\theta+u,\ \theta-u>\frac{d-1}{2}$. Consider independent random variables 
    $$  U  \sim \Wisv(\theta-u), \quad V  \sim \Wis(\theta+u), \quad  W \sim \Wisv(2\theta). $$ 
    Let $S$ be an arbitrary independent random variable on $\pd$. Define 
    \be\label{eq:expression of U'V'}
    U'=\bigg( W\star \Big( U\star(V\star S)+S\Big) \bigg)\star S^{-1}, \qquad 
    V'= \Big( U\star (V\star S)\Big)
    \star 
    \bigg( W\star \Big( U\star(V\star S)+S\Big) \bigg)^{-1}.\ee
Then $(U',V')$ has the same distribution as $(U,V)$.
\begin{figure}
    \centering
    \begin{subfigure}{0.45\textwidth}
    \centering
        \begin{tikzpicture}[scale=1] 
            \draw[-, very thick, orange] (4,0)-- (4.5,0);
            \draw[-, very thick, orange] (0,4)-- (0,4.5);
            \foreach \x in {1, 2, 3, 4}
                \draw[->, very thick, orange] (\x-0.93,0) -- (\x-0.05,0) ;
            \foreach \y in {   4}
                \draw[->,very thick, orange] (0,\y-0.93) -- (0,\y-0.05) ;
            \foreach \y in {1,2,3 }
                \draw[->,very thick, blue] (0,\y-0.05) -- (0,\y-0.93) ;
            \fill[green] (0,3) circle(0.1);
            \node[left,font=\footnotesize] at (0,3) {$S$};
            \foreach \x in {1, 2, 3 }
                \draw[->, very thick, orange!40] (\x-0.93,2) -- (\x-0.05,2) ;
            \foreach \y in {1,2}
                \draw[->,very thick, blue!40] (3,\y-0.05) -- (3,\y-0.93) ;
            \foreach \y in {0,2, 1, 4}
                \fill[black] (0,\y) circle(0.08);
            \foreach \x in {2, 1, 3, 4}
                \fill[black] (\x,0) circle(0.08);
            \foreach \x in {2, 1, 3, 4}
                \node[below,orange, font=\footnotesize] at (\x-0.5,0) {$\ta-u$};
            \foreach \y in {2, 1, 3 }
                \node[left,blue, font=\footnotesize] at (0,\y-0.5) {$\ta+u$};
            \foreach \y in {4 }
                \node[left,orange, font=\footnotesize] at (0,\y-0.5) {$\ta+u$};
            \foreach \x in {2, 1, 3}
                \node[above,orange!40, font=\footnotesize] at (\x-0.5,2) {$\ta-u$};
            \foreach \y in {2, 1  }
                \node[right,blue!40, font=\footnotesize] at (3,\y-0.5) {$\ta+u$};
            \foreach \x in {1, 2,  4}
                \draw (\x, 0) -- (\x,4.5) ;
            \foreach \y in {1,   3, 4}
                \draw (0, \y) -- (4.5, \y);
            \draw (3, 2) -- (4.5,2) ;
            \draw (3, 2) -- (3,4.5) ;
            \foreach \x in {1, 2, 3, 4} 
                \foreach \y in {1, 2, 3, 4}
                \fill[red] (\x,\y) circle(0.05);        
\node[left,red,font=\footnotesize] at (4,3.2) {$2\ta$};
\end{tikzpicture} 
\caption{}
         \label{subfig:propagation of boundary conditions}
    \end{subfigure}
    \hfill
    \begin{subfigure}{0.45\textwidth}
    \centering
        \begin{tikzpicture}[scale=1] 
        \fill[black] (0,0) circle(0.1);
        \fill[green] (0,2) circle(0.1);
        \fill[red] (2,2) circle(0.08);
        \fill[black] (2,0) circle(0.1); 
            \draw[->, very thick, orange] (0.1,0)-- (1.9 ,0);
            \draw[->, very thick, blue] (0,1.9)-- (0,0.1);
             \draw[->, very thick, orange!40] (0.1,2)-- (1.9,2); 
            \draw[->, very thick, blue!40] (2,1.9)-- (2,0.1);
            \draw[->, dotted,very thick, black!40] (0.07,0.07)-- (1.93 ,1.93);
            \node[above left ] at (0,2) {$ S$}; 
            \node[above right,red] at (2,2) {$W$};
            \node[left] at (0,1) {$ V$};
            \node[below] at (1,0) {$ U$};
            \node[above,black!40] at (1,2) {$ U'$};
            \node[right,black!40] at (2,1) {$ V'$};  
\end{tikzpicture}
\caption{}
         \label{subfig:one-step update}
    \end{subfigure}
    \caption{(a) Propagation of boundary conditions, different parameters for (inverse-)Wishart disorders are labeled aside. (b) One step update.}
    \label{fig:placeholder}
\end{figure}
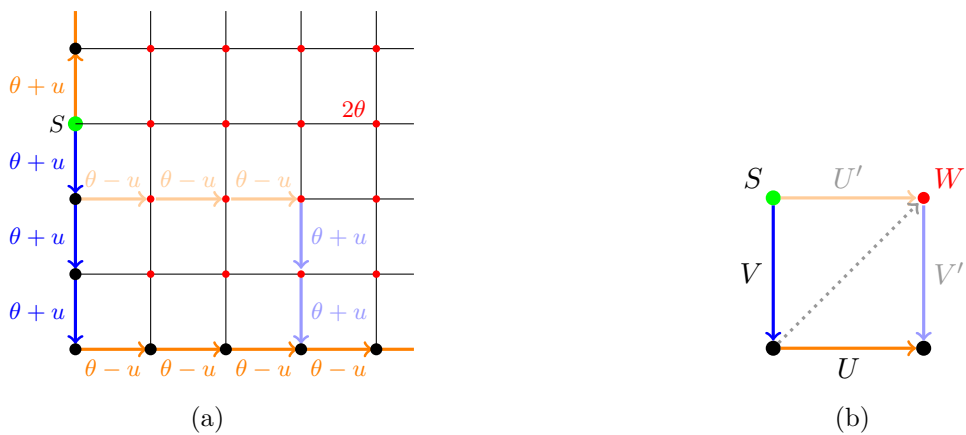
\end{example}
\begin{remark}\label{rmk:matrix Lukacs theorem}
    The property of (inverse-)Wishart distribution described in Example \ref{eg:concret example} is, to our knowledge, not known in the literature. Its scalar analogue is explained in an equivalent form in \cite[Lemma 3.2.]{Seppalainen}, which follows from a classical result  known as Lukacs' theorem \cite{Lukacs}.
     When $d=1$, the product $\star$ reduces to the usual multiplication of positive numbers, and the expression for $U', V'$ simplifies to
    \be\label{eq:expression of U'V' d=1}
    U'=W(UV+1),\qquad V'=\frac{UV}{W(UV+1)}.\ee
    Notably, in this case, the expressions for $U'$ and $V'$   \eqref{eq:expression of U'V' d=1} do not depend on the value of $S$. The identity in distribution $(U',V')\overset{d}{=}(U,V)$ holds even without any hypothesis on the reference matrix (in this case a positive scalar) $S$, in particular without requiring independence  between $S$ and increments $U, V$. 
    For $d\geq 2$, our methods differs from that of \cite{Seppalainen}. Let $U,V,W$ be as in Example \ref{eg:concret example}. The matrix version of Lukacs' theorem \cite{OlkinRubin} implies that if we define
    \be U''= \Big( U\star \big( U^{-1}+V \big)\Big)\star W, \qquad V''= \Big( V\star \big( U^{-1}+V \big)^{-1}\Big)\star W^{-1}.\ee
Then $(U'',V'')$ has the same distribution as $(U,V)$. However, it is not clear how this identity could be used to deduce the stationary structure of the one-step update in the polymer model.    
     \end{remark}
     
\begin{remark}
\label{rmk:markovian issue quadrant}
    When $d=1$,  the increments (i.e. successive ratios of partition functions) themselves form a Markov process under one-step updates. Hence, it is more common in the literature to study the increment process of the log-gamma polymer directly. Its stationary probability measure is given by products of independent inverse-gamma random variables. This  independence  structure propagates into the whole bulk lattice without constraints on the starting point $\p_0$ of the down-right path \cite[Theorem 3.3]{Seppalainen}. The identity \eqref{eq:expectation of partition function} also holds for all $m, n \in \mathbb{N}$; see \cite[Eq. (2.5)]{Seppalainen}. 

    When $d\geq 2$, however, the expressions for $(U',V')$ in  \eqref{eq:expression of U'V'} depend on $S$ and do not simplify due to  non-commutativity. That is, $(U', V')$ is not measurable with respect to the $\sigma$-algebra generated by $(U, V, W)$. Consequently, we have to work with the whole process of partition functions, or equivalently, the increments along with the reference matrix at some reference point, as considered in Definition \ref{def:Stationary boundary conditions}. In Section \ref{subsubsec:Stationary measure for model with mu-distributed reference matrix}, we provide stationary measures for such joint processes.
    	 
    \end{remark}

\subsubsection{Stationary measures for the model with $\mu$-distributed reference matrix}\label{subsubsec:Stationary measure for model with mu-distributed reference matrix} 

Recall that $\mu$ is the $\sigma$-finite measure on $\pd$ given by \eqref{eq:d mu}. The main stationarity result for the model with boundary conditions is the following.
\begin{theorem}\label{thm:quadrant stationary measure}
    For arbitrary reference point $-M$, the stationary boundary conditions
    $\bB^{\ta,u}_{-M,\mu}$ is a stationary measure for the model in the sense given in Definition \ref{def:quadrant model with bc}.
\end{theorem}
Theorem \ref{thm:quadrant stationary measure} is a consequence of the slightly more general Theorem \ref{thm:inhomogeneous quadrant}, and is proven in Section \ref{subsec:Application to inhomogeneous quadrant model}.

We show in Corollary \ref{cor:stationary bc reference point} that, when the reference matrix is $\mu$-distributed, the stationary boundary conditions do not depend on the choice of the reference point. In other words, as $\sigma$-finite measures on $\pd^{\ZZ}$, we have $\bB^{\ta,u}_{-M,\mu}=\bB^{\ta,u}_{0,\mu}$ for all $M\geq 0$. When the reference point is chosen as $M=0$, the stationary measure has a simple characterization given by the following definition.
 
\begin{definition}[Two-sided inverse-Wishart random walk]\label{def:two sided random walk}
Let $\ba=(\a_n)_{n\geq 1}$, $\bb=(\b_n)_{n\geq 1}$ be two families of  parameters such that for all $n\geq 1$, $\a_n, \b_n> \frac{d-1}{2}$. The  two-sided inverse-Wishart random walk starting from $S_0$ 
refers to the stochastic process 
$\bS=\lb S_k\rb_{ k\in\ZZ} \in\pd^{\ZZ}$ 
with independent forward or backward inverse-Wishart increments: for $k\geq 1$,
\begin{align*}
S_k\star S_{k-1}^{-1}  \sim \Wisv(\a_k),\quad
    S_{-k}\star S_{-(k-1)}^{-1}  \sim \Wisv(\b_k), 
\end{align*}
where $S_0\in\pd$ is a (possibly random) matrix which is independent of the increments.  
Identify $\bS=(S_0, \hat{\bS}^{(0)})$. For every realization of $S_0$,  
we denote the marginal distribution of $\hat{\bS}^{(0)}$  by 
$ \d\bR^{\ba,\bb }_{S_0} $, i.e. the probability measure on $\pd^{\mathbb{Z}\setminus\{0\}}$ with density 
$$\bR^{\ba,\bb  }_{S_0} (\hat{\bS}^{(0)})= 
\prod_{i=1}^{\infty}\P^{-}_{\a_i}\Big( S_i\star S_{i-1}^{-1}\Big) \prod_{j=1}^{\infty} \P^{-}_{\b_j}\Big( S_{-j}\star S_{-(j-1)}^{-1}\Big) 
$$
with respect to the reference measure 
$\d\mu(\hat{\bS}^{(0)})=\prod_{i\neq0}\d\mu(S_i)$.

In the special case $\ba=(\a)^{\NN}$, $\bb=(\b)^{\NN}$, we refer to the resulting process as the drifted two-sided inverse-Wishart random walk and denote its distribution by $\d\mathbf{R}^{\alpha,\beta}_{S_0}$. 
Note that under $\d\mathbf{R}^{\alpha,\beta}_{S_0}$, the marginal distribution of $(S_k)_{1 \leq k \leq N}$ is the $N$-step inverse-Wishart random walk $\d\mathbf{R}^{(\alpha)^N,\rightarrow}_{S_0}$ defined in Definition \ref{def:IW random walk}.
\end{definition} 

When the reference point is chosen as $M = 0$, the stationary boundary conditions reduce to a drifted two-sided inverse-Wishart random walk,  
$$\bB^{\ta,u}_{0, \delta_{S_0}}=\d\bR^{\ta-u,\ta+u}_{S_0}.$$
Hence, Theorem \ref{thm:quadrant stationary measure} can be can be restated as the stationarity of the drifted two-sided inverse-Wishart random walk with   $\mu$-distributed reference matrix  for the model with boundary conditions.
This result  was first stated in an equivalent form in \cite{KrajenbrinkLeDoussal}, where it was derived using physics methods. Our work provides a rigorous confirmation of the physics prediction.  

\begin{remark}
    Stationary measures which are not probability measures appear in related contexts. In particular, in the context of continuum directed polymers in dimension $1+1$,  \cite[Theorem 1.2]{FunakiQuastel} showed that the family of $\sigma$-finite measures given by the exponential of Brownian motion with arbitrary drift and Lebesgue-distributed reference point is stationary for the stochastic heat equation on the real line (see also the earlier work \cite{BertiniGiacomin} for an equivalent statement). It is expected in \cite[Remark 1.1]{FunakiQuastel} and then proven in  \cite{JanjigianRassoul-AghaSeppäläinen,DunlapSorensen} that these constitute all extremal invariant measures (up to multiplicative constants). 
 In our case, the classification of stationary measures is an important problem, which we do not address, but is related to the discussion about the value of the constant in law of large numbers in Section  \ref{subsec:Back to model with delta boundary condition}. 
\end{remark}

\subsection{Law of large numbers for the model of \cite{AristaBisiOConnell}}
\label{subsec:Back to model with delta boundary condition}
In this section, we consider the model with delta boundary condition, and we discuss analogues of   \eqref{eq:cv by free energy d=1} and \eqref{eq:identify the limit d=1} for general $d\geq 1$. When $d\geq 2$, due to the non-commutative nature of the model, the partition functions of inverse-Wishart polymer $Z_d^{\ta}(n,m)$ no longer writes as a sum over paths \eqref{eq:sum over path},  and the sequence $\lb\log\abs{Z_d^{\ta}(n,n)}\rb_{n\geq1}$ is no longer subadditive. Hence, the convergence of the normalized free energy is a non-trivial problem.   We propose the following conjecture:
\begin{conjecture}[Law of large numbers]\label{conj: LLN quadrant}
    Fix $d\geq 1$ and a parameter $\theta$ such that $2\ta>\frac{d-1}{2}$, and consider the  model with delta boundary condition (given by \eqref{eq:step boundary condition}), that is the model defined in \cite{AristaBisiOConnell}. Recall the definition of $f_d^{\theta}(n)$ from \eqref{eq:deffdn}. Then there exists a deterministic constant $f_d^{\ta}$ such that 
\be\label{eq: quadrant LLN conj}f_d^{\ta}(n) \xrightarrow{n\rightarrow\infty}
f_d^{\ta} \qquad \text{a.s.}
\ee
\end{conjecture} 

For $d = 1$, the limit is given by \eqref{eq:identify the limit d=1}. For $d \geq 2$, the limit remains unknown.
Note that the model is integrable in the sense that the probability distribution of the  partition function is given by a marginal of the matrix Whittaker measure \cite[Corollary 4.10]{AristaBisiOConnell}. In principle, the limit could therefore be determined via asymptotic analysis of this measure. However, the matrix Whittaker measure is significantly more complicated than its scalar analogue, and its analysis remains a challenge. In any case, there should exist a simpler way to determine the limit. 

Building on \cite{Seppalainen}'s approach for the scalar case, along with the boundary partition function expectation \eqref{eq:expectation of partition function} and the explicit form of $f_1^{\theta}$ in \eqref{eq:identify the limit d=1},  it is tempting to conjecture that for general $d\geq 1$ and all $\ta>\frac{d-1}{2}$,
\begin{equation}
f_d^{\ta}\overset{?}{=}-2\sup_{\frac{d-1}{2}-\ta<u<\ta-\frac{d-1}{2}}\Big(\psi_d(\ta-u)+\psi_d(\ta+u) \Big)=-2\psi_d(\ta).
\label{eq:wrong}
\end{equation}
This turns out to be false. First of all, we observe that \eqref{eq:wrong} would contradict Conjecture \ref{conj: LLN quadrant}: indeed, as $\theta$ goes to $\frac{d-1}{2}$, $\psi_d(\theta)$ diverges, while we expected that $f_d^{\ta}$ is a finite constant for all $\ta>\frac{d-1}{4}$. This is related to the issue, already pointed out in \cite{KrajenbrinkLeDoussal}, that we have no information about stationary measures of the model when $\frac{d-1}{4}<\ta\leq  \frac{d-1}{2}$. 

Furthermore, we can prove that $f_d^{\ta}\neq -2\psi_d(\ta)$, at least for $d\geq 3$, without relying on Conjecture \ref{conj: LLN quadrant}. By introducing a well-chosen martingale, we derive an upper bound for $f^{\ta}_d$ which is strictly smaller than $-2\psi_d(\ta)$. 

For $k\geq 1$, let $\mathcal{F}_k$ be the $\sigma$-algebra generated by the random variables $\lb W(i,j)\rb_{2\leq i+j\leq k}$.
We define the normalized point-to-line partition function 
$$M^{\ta}(k):=\frac{ \sum_{i+j=k} Z^{\ta}(i,j)}{(2c(2\ta))^{k-1}} ,$$
where for $\ga>\frac{d+1}{2}$, $c (\ga)=\frac{1}{\ga-\frac{d+1}{2}}$ is the constant such that $\EE[\Wisv(\ga)]= c(\ga)\id_d$ -- see \cite[Lemma 7.1.1]{Anderson}. Hence, our model is defined as long as $\ta>\frac{d-1}{4}$ but the expectation is finite only when $\ta>\frac{d+1}{4}$. As in the scalar case, the point-to-line partition function is a martingale.
\begin{lemma}
    $\lb M(k)\rb_{k\geq 1}$ is a $\pd$-valued martingale with respect to the filtration $(\mathcal{F}_k)_{k\geq1}$.
\end{lemma}
\begin{proof} For all $k\geq 1$, it is clear that $M(k)$ is $\mathcal{F}_k$-adapted, and  
\begin{equation*}   
\begin{split}\EE\left[M^{\ta}(k+1)|\mathcal{F}_k\right]
&=\sum_{i=0}^{k+1} \EE\left[ W(i,k+1-i)\star \lb Z^{\ta}(i-1,k+1-i)+ Z^{\ta}(i,k-i)\rb \Big|\mathcal{F}_k\right]\\
        &=\sum_{i=0}^{k+1} \EE[ W(i,k+1-i)]\star \lb Z^{\ta}(i-1,k+1-i)+ Z^{\ta}(i,k-i)\rb\\
        &=c(2\ta)\sum_{i=0}^{k+1}  \lb Z^{\ta}(i-1,n+1-i)+ Z^{\ta}(i,n-i)\rb\\
        &=2c(2\ta) M^{\ta}(k).
    \end{split}
\end{equation*} 
\end{proof}
By the matrix supermartingale convergence theorem \cite[Theorem 3.2]{WangRamdas}, 
$(M^{\ta}(k))_{k\geq 1}$
converges to a limit 
$M^{\ta}({\infty})\in\overline{\pd}$ 
almost surely, with 
$\EE[M^{\ta}({\infty})]\leq\EE[M^{\ta}(1)]=\id_d$. 
On the other hand, recall  Minkowski's determinant inequality \cite[Thm 7.8.21]{HornJohnson}:
$$\abs{A}^{\frac{1}{d}}+\abs{B}^{\frac{1}{d}}\leq\abs{A+B}^{\frac{1}{d}} \qquad \text{for all } A,B\in\pd.$$  
One deduces that for all $n\geq 1$, $\abs{Z^{\ta}(n,n)}\leq \abs{\sum_{i+j=2n} Z^{\ta}(i,j)}$.  Therefore, 
\begin{equation*}\frac{\log\abs{Z^{\ta}(n,n)}}{n} \leq \frac{ \log\abs{\sum_{i+j=2n} Z^{\ta}(i,j)}}{n} = \frac{\log\abs{M^{\ta}(2n)}+\log\lb(2c(2\theta))^{2nd}\rb}{n}, 
		\end{equation*}
so that recalling the notation \eqref{eq:deffdn}, 	
\begin{equation}
	f_d^{\theta}(n) \leq\frac{\log \vert M^{\theta}(2n) \vert}{n} -  2d\log \lb\ta-\frac{d+1}{4}\rb.
	 \label{eq:inequalitymatrix} 
\end{equation}
As a consequence, if the convergence \eqref{eq: quadrant LLN conj} holds, using the supermartingale convergence theorem,  we would obtain that 
$$ f_d^{\theta} \leq \limsup_{n\to\infty}  \frac{\log \vert M^{\theta}(2n) \vert}{n}  -2d\log \lb\ta-\frac{d+1}{4}\rb, $$
where the limsup must belong to $[-\infty,0]$, so that for all $\ta>\frac{d+1}{4}$, 
\be\label{eq:violated inequality}f_{d}^{\ta}\leq -2d\log\lb\ta-\frac{d+1}{4}\rb.\ee
In particular for $d\geq 3$, the value of $f_d^{\ta}$ cannot equal $-2\psi_d(\ta)$ since $\psi_d(\theta)> -2d\log\lb\ta-\frac{d+1}{4}\rb$  for all $\ta>\max\big(\frac{d-1}{2},\frac{d+1}{4}\big)=\frac{d-1}{2}$.
Moreover, our numerical simulations suggest that the convergence \eqref{eq: quadrant LLN conj} does hold,  and that $f_d^{\ta}< -2\psi_d(\ta)$ for $d\geq2$. 

\begin{remark}
In 	the argument above, instead of using the matrix supermartingale convergence theorem, we could also have used the concavity of the function $M \mapsto \log\vert M\vert$ so that using \eqref{eq:inequalitymatrix} and Jensen's inequality, for $\ta>\frac{d-1}{4}$, 
$$ \mathbb E\left[ f_d^{\theta}(n) \right] \leq \frac{\log \left\vert \mathbb E\left[ M^{\theta}(2n) \right]  \right\vert}{n} - 2d\log \lb\ta-\frac{d+1}{4}\rb = \frac{\log \left\vert  \tfrac 1 2 \id_d   \right\vert}{n} - 2d\log \lb\ta-\frac{d+1}{4}\rb  $$
and letting $n$ go to infinity, we arrive at the same conclusion.
\end{remark}

\subsection{Model on a strip}\label{subsec:Model on a strip}
The log-gamma polymer on a strip introduced in \cite{BarraquandCorwinYang} extends the classical log-gamma polymer to a new geometric setting. In this section, we consider the inverse-Wishart polymer on a strip. 
\subsubsection{The model}

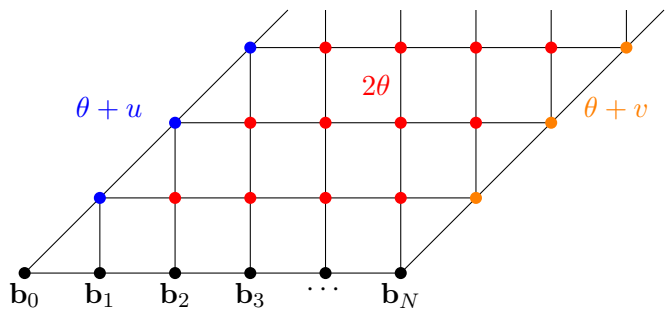
\begin{figure}
    \centering
    \begin{tikzpicture}[scale=1]
\draw (0,0)--(3.5,3.5);
\draw  (5,0)--(8.5,3.5);
\draw  (0,0)--(5,0);
\draw  (1,1)--(6,1);
\draw  (2,2)--(7,2);
\draw  (3,3)--(8,3);   
\draw  (1,0)--(1,1);
\draw  (2,0)--(2,2);
\draw  (3,0)--(3,3);
\draw  (4,0)--(4,3.5);
\draw  (5,0)--(5,3.5);
\draw  (6,1)--(6,3.5);
\draw  (7,2)--(7,3.5);
\draw  (8,3)--(8,3.5); 
\foreach \x in {0,2, 1, 3, 4,5}
                \fill[black] (\x,0) circle(0.08);
\foreach \x in {2, 1, 3}
                \fill[blue] (\x,\x) circle(0.08);
\foreach \x in {2, 1, 3}
                \fill[orange] (\x+5,\x) circle(0.08);
\foreach \x in {1, 2, 3} 
                \foreach \y in {1, 2, 3, 4}
                \fill[red] (\x+\y,\x) circle(0.08);
\node[below] at (0,0) { $\mathbf{b}_0$};
\node[below] at (1,0) {$\mathbf{b}_1$};
\node[below] at (2,0) { $\mathbf{b}_2$};
\node[below] at (3,0) { $\mathbf{b}_3$};
\node[below] at (4,0) { $\cdots$};
\node[below] at (5,0) { $\mathbf{b}_N$};

\node[left ,blue] at (1.7,2.2) {$\ta+u$};
\node[right,orange] at (7.3,2.2) {$\ta+v$};
\node[left,red] at (5,2.5) {$2\ta$};
\end{tikzpicture} 
    \caption{Model on a strip. Different parameters for the inverse-Wishart disorders are labeled aside.}
    \label{fig:model on a strip}
\end{figure}
Let $\strip=\{ (n,m)\in \ZZ^2: m\leq n \leq m+N\}$ denote the strip of width $N$.  Recall that $\tau_l:\ZZ^2\rightarrow\ZZ^2, x\mapsto x+(l,l), l\in\NN_+$ denotes the diagonal translation on $\ZZ^2$, clearly $\strip$ is preserved under such translations. We call the subset consisting of points with either $n=m$ or $n=m+N$ the left boundary and right boundary of the strip, respectively. With slight abuse of notations, the bottom boundary of the strip refers to the sequence of points 
$\hp_{b}=(\mathbf{b}_k)_{ 0\leq k \leq N }\in\lb\strip\rb^{N+1}$, 
where $\mathbf{b}_k=(k,0)$.

\begin{definition}[Inverse-Wishart polymer on a strip]\label{def:homogeneous strip model}
     Let  $N\in\NN $ and let  $\theta, u, v\in \RR$ be parameters  satisfying one of the following two regimes:
\begin{subequations}
\begin{align}
\text{(maximal current regime)} \qquad & u+v,\ 2\theta,\ \theta+u,\ \theta+v>\frac{d-1}{2}; \label{eq:parameters homogeneous general regime}\\
    \text{(equilibrium regime)}\qquad &  u+v=0,\ \theta+u,\ \theta+v>\frac{d-1}{2}.\label{eq:parameters homogeneous equilibrium regime}
\end{align}
\end{subequations}

Consider a family of independent random variables indexed by vertices of the strip, where for $0<m<n<m+N$, 
      $$  W(m,m)\sim \Wisv(\theta+u),\  W(m+N,m)\sim \Wisv(\theta+v),\ W(n,m)\sim \Wisv(2\theta). $$
    
The boundary conditions (or, in the context of model on a strip, initial conditions) refer to a stochastic process $\bB=\lb B_k\rb_{0\leq k\leq N}$ that is independent of the above random variables. The partition functions of inverse-Wishart directed polymer on a strip denoted by $\lb Z^{\ta,u,v}(n,m)\rb_{0\leq m\leq n\leq m+N}$ are defined by the initial conditions $$Z^{\ta,u,v}(\mathbf{b}_{k})=B_k,\quad 0\leq k\leq N;$$ and the recurrence  
    \be\label{eq:recurrence homogeneous strip} Z^{\ta,u,v}(n,m) = W(n,m)\star \begin{dcases*}
Z^{\ta,u,v}(n-1,m)+Z^{\ta,u,v}(n,m-1) & if $0<m<n< m+N$,\\
Z^{\ta,u,v}(n,m-1) & if $0<m=n$,\\
Z^{\ta,u,v}(n-1,m) & if $0<m=n-N$.
\end{dcases*}\ee 
See Figure \ref{fig:model on a strip} for illustration.
We say that a $\sigma$-finite measure $\nu$ on state space $\pd^{N+1}$ is a stationary measure for the model if $\bB\sim\nu$ implies $\lb Z^{\ta,u,v}(n,m)\rb_{m\leq n \leq m+N}\sim\nu$ for all $m\geq 0$. 

\end{definition}
 
\begin{remark}\label{rmk:markovian issue strip}
Again, we emphasize that our definition of a stationary measure differs from that used in the scalar case ($d=1$). Under the polymer recurrence, the increment process   
    $$\left\{ \lb Z^{\ta,u,v}(n,m)\star \big( Z^{\ta,u,v}(m,m)\big)^{-1} \rb_{m+1\leq n\leq m+N}\right\}_{m\geq 0}$$
    forms a well-defined Markov process and admits a stationary probability measure \cite{BarraquandCorwinYang}. However, this property no longer holds when $d \geq 2$ due to the same reason   illustrated in Example \ref{eg:concret example} for the model on the quadrant. The increment process is no longer Markovian, it is therefore only meaningful to consider stationary measures for the full process of partition functions.
\end{remark}

\begin{remark}
Here we define the homogeneous model for simplicity in the sense that the parameters for the disorders $\big( W(n,m)\big)_{0<m<n<m+N}$ are the same, in Section \ref{subsec:Inhomogeneous inverse-Wishart polymer on a strip and two-layer matrix Whittaker process} we will discuss an inhomogeneous extension where the disorder can have different parameters.
\end{remark}

\subsubsection{Stationary measures for the model in the equilibrium regime} 
The following theorem shows that, when the model parameters belong to the equilibrium regime, the inverse-Wishart random walk (Definition \ref{def:IW random walk}) with $\mu$-distributed starting matrix is a stationary measure for the model on a strip in the sense precised in Definition \ref{def:homogeneous strip model}.
\begin{theorem}
\label{thm:homogeneous equilibrium strip}
    Consider the inverse-Wishart polymer on a strip with parameters  
    $\theta, u, v=-u\in \RR$ satisfying the equilibrium regime \eqref{eq:parameters homogeneous equilibrium regime}. Let $\bg=(\ta-u)^N$ . Then, the $\sigma$-finite measure 
    \be\label{eq:stationary measure for the homogneous model on a strip}\d\mu\big( B_0 \big)\times \d\bR^{\bg,\rightarrow}_{B_0} \big(\hat\bB^{(0)}\big)\ee is stationary for the inverse-Wishart polymer on a strip. 
\end{theorem}
Theorem \ref{thm:homogeneous equilibrium strip} is a special case of Theorem \ref{thm: inhomogeneous equilibrium strip}, and is proven in Section \ref{subsec:Equilibrium regime of the strip model}.  

 \bigskip
When the initial condition is fixed or determined by some probability measure, we propose the following conjecture.
\begin{conjecture}[Law of large numbers]\label{conj:LLN strip}
    For $d\geq 1$ and parameters $\ta, u,v$ in either the maximal current regime \eqref{eq:parameters homogeneous general regime} or the equilibrium regime \eqref{eq:parameters homogeneous equilibrium regime}, 
    consider the model on a strip of arbitrary width $N\geq 1$, with initial condition given by an arbitary probability measure on $\pd^{N+1}$. Then there exists a deterministic constant $f_d^{\ta,u,v}$ such that 
    \be\label{eq: strip conj LLN}
    \frac{\log\abs{Z^{\ta,u,v}(n,n)}}{n}\xrightarrow{n\rightarrow\infty}
f_d^{\ta,u,v}\qquad \text{a.s.} \ee
Moreover, in the equilibrium regime, we have
    \be\label{eq: strip conj LLN value} 
f_d^{\ta,u,-u}=-\psi_d(\ta-u)-\psi_d(\ta+u).  \ee  
\end{conjecture}

For $d = 1$, the convergence \eqref{eq: strip conj LLN} can be argued by subadditivity of the partition functions, as is commonly done in the literature for models on the quadrant; see, for example, \cite{Comets}. The expectation (i.e. the annealed free energy) can be computed from the integral formulas of two-layer Whittaker process \cite{Barraquand}. In Appendix \ref{appendix:LLN_d1}, we present an alternative argument for the almost sure convergence,  based on Birkhoff's ergodic theorem, which avoids the use of subadditivity and may be of independent interest.

When $d \geq 2$, proving the convergence \eqref{eq: strip conj LLN} becomes non-trivial for the same reason discussed for the model on the quadrant in Section~\ref{subsec:Back to model with delta boundary condition}. Note that compared with the conjecture for the quadrant case (Conjecture \ref{conj: LLN quadrant}), we have an explicit expression for the limit, which can be confirmed through numerical simulations.
%

\subsubsection{Stationary measures for the model in the maximal current regime}
For the model on a strip in the maximal current regime, we need some notation to describe its stationary measure.
For parameter $\ga>\frac{d-1}{2}$ and arguments $\la,\mu\in \pd^2$, let us consider the function 
    \be\label{eq:Baxter Q operator}\Psi_{\ga}(\la/\mu):=\abs{\mu_1(\la_1)^{-1}}^\ga 
    \abs{\mu_2(\la_2)^{-1}}^\ga  e^{-\tr[\mu_1(\la_1)^{-1}+\mu_2(\la_2)^{-1}+\la_2(\mu_1)^{-1}]}.
    \ee
    This is a matrix generalization of `Baxter Q-type operator' introduced in \cite{OConnell}. Recall that $\d\mu$ is a $\sigma$-finite measure on $\pd$. For $\la=(\la_1,\la_2)\in\pd^2$  we  write $\d\mu(\la)=\d\mu(\la_1)\d\mu(\la_2)$.

\begin{definition}
\label{def: homogeneous two layer matrix whittaker process}
Let $N\in\ZZ_{>0}$. Fix a matrix $\la_1^0\in\pd$ and parameters $\ga, u, v\in\RR$ such that 
$ 2\gamma, \ \gamma+u,\  \ga+v>\frac{d-1}{2}.$
The two-layer matrix Whittaker process starting from $\la_1^0$, denoted by $\d\PP^{\ga,u,v}_{\la^0_1}$, is the probability measure on $\pd
\times(\pd^2)^{N}$ with density 
    \be\label{eq:two-layer density strip} 
    \PP^{\ga,u,v}_{\la^0_1}\lb\la_2^0, (\la^i)_{1\leq i \leq N}\rb=
    \frac{1}{\mathcal{Z}^{\ga,u,v}}\abs{\la_2^0(\la_1^0)^{-1}}^u\abs{\la_2^N(\la_1^N)^{-1}}^v\prod_{i=1}^{N}\Psi_{\ga}(\la^i/\la^{i-1}), \ee
with respect to the reference measure  $ \d\mu(\la^0_2) \prod_{i=1}^{N} \d\mu(\la^i)$. The normalization constant $\mathcal{Z}^{\ga,u,v}$ will be proven finite and not depend on the value of $\la^0_1$ in Proposition \ref{prop:Finiteness of Partition function}. 

We denote the marginal density  of $(\la_1^j)_{1\leq j\leq N}$ with respect to $\prod_{i=1}^{N} \d\mu( \la^i_1)$ by $\bP^{\ga,u,v}_{\la^0_1} $. Here and throughout, we use the bold symbol $\d\mathbf P$ to denote the marginal distribution of the first layer under the probability measure $\d\mathbb P$.  
\end{definition}
The two-layer matrix Whittaker process is a matrix-valued generalization of its scalar analogue defined in \cite{BarraquandCorwinYang,Barraquand}. We provide a probabilistic interpretation of this process in Section \ref{subsubsec:Markovian definition}. The following is the stationarity result for the model on a strip in the maximal current regime.

\begin{theorem}\label{thm: homogeneous strip maximal current} 
    Consider the inverse-Wishart polymer on a strip with parameters  $\theta, u, v\in \RR$ satisfying the maximal current regime \eqref{eq:parameters homogeneous general regime}. Then the marginal distribution of  $\bl_1=(\la_1^{i})_{0\leq i\leq N}$ under the two-layer matrix Whittaker process  with $\mu$-distributed starting matrix, i.e. the $\sigma$-finite measure  
    \be\label{eq:stationary measure for homogeneous strip model in the maximal current regime}  \d\mu\big(\la_1^0\big)\times\d\bP^{\ta,u,v}_{\la^0_1}\lb  (\la^j_1)_{1\leq j \leq N}\rb,\ee
     is a stationary measure for the model on a strip.
\end{theorem}

Theorem \ref{thm: homogeneous strip maximal current} is a special case of Theorem \ref{thm: inhomogeneous maximal current strip}, and is proven in Section \ref{subsec:Proof of inhomogeneous strip max current}. 

\subsubsection{Markovian interpretation}\label{subsubsec:Markovian definition}In this section we give the two-layer matrix Whittaker process $\d\PP^{\ga,u,v}_{\la^0_1}$ \eqref{def: homogeneous two layer matrix whittaker process} a probabilistic interpretation.  

Let us define  the skew matrix Whittaker function through the following ``branching rule'': for parameters $\ga_1,\dots,\ga_k > \frac{d-1}{2}$ and arguments $\la^0=\mu, \la^k=\la\in\pd^{2}$, define
    \be\label{eq:branching rule, skew whittaker function}\Psi_{\ga_1,\dots,\ga_k}(\la/\mu):=\int_{(\pd^2)^{k-1}}\prod_{i=1}^{k}\Psi_{\ga_i}(\la^i/\la^{i-1})\prod_{i=1}^{k-1}\d\mu(\la^i).
    \ee
      
\begin{proposition}\label{prop: homogeneous markovian description}
    Under the probability measure $\d\PP^{\ga,u,v}_{\la^0_1}$, the process $x\mapsto \la^x,\ 0\leq x\leq N$ is an  inhomogeneous Markov chain on $\pd^2$ with initial distribution $\P_0$ and transition kernels $(\P_{x,y})_{0\leq x< y\leq N}$  given respectively by  
    \begin{subequations}\label{eq:markov description}
    \begin{align}
        \P_0(\la)&=\frac{1}{\mathcal{Z}^{\ga,u,v}}\abs{\la_2}^{u}\delta_{\la_1^0}(\la_1)H_{0,N}(\la_2\star \la_1^{-1}), \label{eq:initial distributon}\\
    \P_{x,y}(\la,\mu)&=\frac{H_{N-y}(\mu_2\star\mu_1^{-1})}{H_{N-x}(\la_2\star\la_1^{-1})}
    \Psi_{(\ga)^{y-x}}(\mu/\la),\label{eq:transition kernel of survived random walk}
    \end{align}
\end{subequations} 
where for $k\geq 1$, the function  $H_{k}:\pd\rightarrow\RR_{+}$ is defined by  
$$H_{k}(\mu_2)=\int_{\pd^2} \Psi_{(\ga)^k}
\Big( (\la_1, \la_2)/(\id,\mu_2)\Big) \abs{\la_2(\la_1)^{-1}}^v \d\mu(\la_1)\d\mu(\la_2).$$
\end{proposition}
Proposition \ref{prop: homogeneous markovian description} is the special case $\bg=(\ga)^N$, $\bw=(\rightarrow)^N$ of Proposition \ref{prop: inhomogeneous markovian description}, which is proven in Section \ref{subsec:Probabilistic description of two-layer matrix Whittaker process}.
 \begin{remark}
        The factor $\Psi_{\ga_x}(\la^x/\la^{x-1})$  can be interpreted as the sub-Markovian transition kernels for two independent inverse-Wishart random walks $\bl_1, \bl_2$, which are killed at time $x$ with probability $e^{-\tr[\la_2^{x}(\la_1^{x-1})^{-1}]}$. The Doob-transformed kernels ~\eqref{eq:transition kernel of survived random walk}, which are Markovian due to the branching rule~\eqref{eq:branching rule, skew whittaker function}, then correspond to the transition kernels for the walks conditioned to survive up to time~$N$. Similar interpretations are considered in \cite{BrycKuznetsovWangWesołowski,DasdasSerio,Barraquand}.
\end{remark}

\subsection*{Outline of the rest of the paper}
In the rest of the paper, we study inhomogeneous generalizations of the models introduced in the introduction and prove corresponding stationarity results.  
In Section \ref{sec:Model on a strip}, we introduce the inhomogeneous model on a strip which generalize Definition \ref{def:homogeneous strip model} and prove Theorem \ref{thm: inhomogeneous maximal current strip}  which generalize Theorem \ref{thm: homogeneous strip maximal current} for the model with parameters in the maximal current regime.
In Section \ref{subsec:Equilibrium regime of the strip model}, we prove Theorem \ref{thm: inhomogeneous equilibrium strip} which generalize Theorem \ref{thm:homogeneous equilibrium strip} for the model on a strip with parameters in the equilibrium regime. As an application, we prove Theorem \ref{thm:inhomogeneous quadrant} which generalize Theorem \ref{thm:quadrant stationary measure} for the inhomogeneous model on the quadrant  in Section \ref{subsec:Application to inhomogeneous quadrant model}. 

\subsection*{Acknowledgments}
We are grateful to Neil O'Connell for very useful discussions related to the law of large numbers (Conjecture \ref{conj: LLN quadrant}) and the suggestion that the naive guess \eqref{eq:wrong} might not be correct. We thank Pierre Le  Doussal and Alexandre Krajenbrink for discussions related to directed polymers with matrix valued noise and the predictions from \cite{KrajenbrinkLeDoussal}, in particular Theorem \ref{thm:quadrant stationary measure}. We also thank Elia Bisi for useful comments on a preliminary version of this article. G. B and Z.O. acknowledge support from ANR grants ANR-21-CE40-0019 and ANR-23-ERCB-0007.

\section{Maximal current regime of the model on a strip}
\label{sec:Model on a strip} 
\subsection{Preliminaries}\label{subsec: preliminary}
Let us introduce some preliminary notions that we use throughout this work. For background and details, we refer to \cite{Terras,OConnell}.

The space $\pd$ of all $d\times d$ symmetric, strictly positive definite matrices can be viewed as the homogeneous space $\pd\cong \GL(d,\RR)/ O(d, \RR)$, so that $\mathrm{GL}_d$ acts transitively on $\pd$ by $y^txy$  for $x\in \pd$ and $y\in \mathrm{GL}_d$. The $\sigma$-finite measure $\d\mu$ given by \eqref{eq:d mu} is involution-invariant and $\GL_d$-invariant in the sense that: for all $y\in\GL_d$ and all suitable function $f:\pd\rightarrow\RR$,
\begin{subequations}
\begin{align}
\int_{\pd}f( x^{-1} )\d\mu(  x) &= \int_{\pd}f(x)\d\mu(  x), \label{eq:involution invariance of mu}\\
\int_{\pd}f\lb y^{t} x y\rb\d\mu(  x) &= \int_{\pd}f(x)\d\mu(  x). \label{eq:GL_d invariance of mu}
\end{align}
\end{subequations}
The (second type of) matrix K-Bessel function is defined for $\nu\in\C$ \cite{Herz}:
   \be\label{eq:K-Bessel function}
   K_\nu(V,W)=\int_{Y\in\pd}\abs{Y}^\nu e^{-\tr[VY+WY^{-1}]}\d\mu(Y), \quad V,W\in \pd.
   \ee
One can always reduce one of the arguments to the identity matrix. For example, using the change of variable $Y'=W^{-1/2}YW^{-1/2}$ (or using a Cholesky decomposition as done in \cite[Section 2.6]{OConnell}) together with the $\GL_d$-invariance of $\mu$ \eqref{eq:GL_d invariance of mu},  we obtain
$$K_\nu(V,W)=\abs{W}^{\nu} K_\nu(W^{1/2}VW^{1/2},\id).$$
For this reduced K-Bessel function, using the change of variable $Y'=V^{\frac{1}{2}}Y^{-1}V^{\frac{1}{2}}$, the $\GL_d$-invariance of $\mu$ \eqref{eq:GL_d invariance of mu},  and the involution invariance of $\mu$ \eqref{eq:involution invariance of mu}, we have
$$K_{-\nu}(V,\id)=\abs{V}^{\nu}K_\nu(V,\id).$$
Combining the above two identities, we deduce the following symmetry property of the K-Bessel function:
\be\label{eq:symmetry of Bessel function} \abs{V}^\nu K_\nu(V,W) = \abs{W}^{\nu} K_{-\nu}(V,W).\ee
The derivation of this symmetry identity motivates the following auxiliary lemma, which will be used to prove the Cauchy-type and Littlewood-type identities (Lemma \ref{lem:Cauchy identity} and Lemma \ref{lem:Littlewood identity}, respectively) in Section \ref{subsec:Properties of Gibbs measure}.

\begin{lemma}\label{lem:change of variable}
    For any $W, V, Y\in \pd$,  denote $U=W^{\frac{1}{2}}VW^{\frac{1}{2}}$. The matrix $X$ defined by 
    \be X=\lb W^{-\frac{1}{2}}U^{\frac{1}{2}}W^{-\frac{1}{2}}Y W^{-\frac{1}{2}}U^{\frac{1}{2}}W^{-\frac{1}{2}}\rb^{-1}
    \ee 
    is in $\pd$ and such that $\d\mu(X)=\d\mu(Y)$. 
    Moreover, $X$ satisfies the following equations at the same time:
\be
\left\{
\begin{aligned}
    \tr[WX^{-1}]=\tr[VY]\\
    \tr[VX]=\tr[WY^{-1}]
\end{aligned}\right..
\ee
\end{lemma}
\begin{proof} 
The operations $\star$ and $x\mapsto x^{-1}$ preserve $\pd$, therefore $$X=\Big( \big((Y\star W^{-1})\star (V\star W)\big) \star W^{-1} \Big)^{-1}\in \pd.$$
The equality $\d\mu(X) = \d\mu(Y)$ follows from the involution invariance \eqref{eq:involution invariance of mu} and the $\GL_d$-invariance \eqref{eq:GL_d invariance of mu} of $\d\mu$. For any $A,B\in M_d$, we have $\tr[AB]=\tr[BA]$. Taking $A=W^{-\frac{1}{2}}U^{\frac{1}{2}}W^{-\frac{1}{2}}, B=W^{\frac{1}{2}}U^{\frac{1}{2}}W^{-\frac{1}{2}}Y$, we obtain:
 $$\tr[WX^{-1}]=\tr[W^{\frac{1}{2}}U^{\frac{1}{2}}W^{-\frac{1}{2}}Y W^{-\frac{1}{2}}U^{\frac{1}{2}}W^{-\frac{1}{2}}]=\tr[W^{-\frac{1}{2}}U W^{-\frac{1}{2}} Y] 
 =\tr[VY].$$   
 Taking $A=VW^{\frac{1}{2}}U^{-\frac{1}{2}}W^{\frac{1}{2}}$ and $B=Y^{-1} W^{\frac{1}{2}}U^{-\frac{1}{2}}W^{\frac{1}{2}}$, we obtain:
$$
    \tr[VX]= 
\tr[VW^{\frac{1}{2}}U^{-\frac{1}{2}}W^{\frac{1}{2}}Y^{-1} W^{\frac{1}{2}}U^{-\frac{1}{2}}W^{\frac{1}{2}}]=\tr[Y^{-1} W^{\frac{1}{2}}U^{-\frac{1}{2}}W^{\frac{1}{2}}VW^{\frac{1}{2}}U^{-\frac{1}{2}}W^{\frac{1}{2}}]
=\tr[Y^{-1} W].
$$
\end{proof}

The Whittaker function of $N$ matrix arguments is introduced in \cite[Section 7.1]{OConnell}. In the special case $N=2$, for $\a, \b \in \C$, the Whittaker function of two matrix arguments is defined by:
\be\label{eq:whittaker function}
    \psi_{(\a, \b)}(X_1,X_2)=\abs{X_1X_2}^{-\b}\int_{\pd}\abs{Y}^{\b-\a}e^{-\tr[X_1^{-1}Y+X_2Y^{-1}]}\d\mu(  Y), \quad X_1, X_2\in\pd.
\ee
Note that it differs from the K-Bessel function \eqref{eq:K-Bessel function} only by a multiplicative determinant factor. Hence, the symmetry property of the K-Bessel function \eqref{eq:symmetry of Bessel function} implies 
the symmetry of Whittaker function of two matrix arguments  in its parameters $(\a, \b)$ by taking $\nu=\b-\a,\ V=X_1^{-1},\ W=X_2$. Precisely, for all $X_1, X_2\in\pd$, we have 
\be\label{eq:symmetry of whittaker function}\psi_{(\a, \b)}(X_1,X_2)=\psi_{(\b, \a)}(X_1,X_2).\ee

Recall that the Wishart distribution with parameter $\theta>\frac{d-1}{2}$ is the probability measure on $\pd$ with density 
$$\P^{+}_{\ta}(x)=\frac{\abs{x}^\theta e^{-\tr[x]}}{\Gamma_d(\theta)}$$ 
with respect to the reference measure $ \d\mu(x)$, and the normalization constant $\Gamma_d(\theta)$ is the d-variate Gamma function given by \eqref{eq:d-variate gamma function}. There is a useful integral identity: for all $S\in\pd$, 
\be\label{eq:laplace transform}\int_{\pd}\abs{x}^{\ta}e^{-\tr[Sx]}\d\mu(  x)=\Gamma_d(\ta)\abs{S}^{-\ta}.
\ee

Let $x,y\in \pd$, recall $x\star y= y^{1/2} x y^{1/2}\in \pd$. Remark that when $d\geq 2$, the operation $\star$ is neither commutative nor associative, and does not distribute over matrix addition. Due to this lack of distributivity, the partition functions no longer writes as a sum over paths \eqref{eq:sum over path}. Let $\nu$ be an $O_d$-invariant probability measure on $\pd$, the (left) random walk on $\pd$ with increments $\nu$ starting from $x_0\in\pd$ is the Markov process $\lb x_n\rb_{n\geq 0}$ defined via the recursion:
$$x_n:=W_n\star x_{n-1}, \quad n\geq 1,$$
where $(W_n)_{n\geq 1}$ are i.i.d. random variables with law $\nu$. 
The transition kernel $P$ of the process is $\GL_d$-invariant: for all $g\in\GL_d$, we have
$$\forall x,y\in \pd,\quad P(x,y)=P(gxg^t,gyg^t).$$
We shall focus on the particular case of $\Wis^{\pm}$ random walk (Definition \ref{def:IW random walk}) and two-sided inverse-Wishart random walk (Definition \ref{def:two sided random walk}) where $\nu$ is Wishart or inverse-Wishart distributed.
The right random walk on $\pd$ defined via $x_n:= x_{n-1}\star W_n $ could also be interesting, however, it is only $O_d$-invariant as remarked in \cite{AristaBisiOConnell}.

\subsection{Inhomogeneous model on a strip and two-layer matrix Whittaker process}\label{subsec:Inhomogeneous inverse-Wishart polymer on a strip and two-layer matrix Whittaker process}
Recall that $\strip$ denotes the strip of width $N$, on which the diagonal translation $\tau_l$ acts. A down-right path on $\strip$ refers to a sequence $\hp=(\p_k)_{0\leq k \leq N}\in\strip^{N+1}$ jointing a vertex $\p_0$ on the left boundary of the strip with a vertex $\p_N$ on the right boundary, such that increments $\p_{k+1}-\p_{k}$ are either $(1,0)=``\rightarrow"$ or $(0,-1)=``\downarrow"$.  The shape of $\hp$ can thus be recorded as a word $\bw(\hp)\in\{\rightarrow,\downarrow\}^{N}$. 

Conversely, for a fixed word $\bw\in\{\rightarrow,\downarrow\}^{N}$, there exists a unique down-right path on the strip with shape $\bw$ starting from the origin $\p_0=(0,0)$, which we denote by $\hp^{\bw}_{\circ}=(\p_k)_{0\leq k\leq N}=\{(n_k,m_k)\}_{0\leq k\leq N}$. A strip with bottom shape $\bw$ refers to the bulk lattice 
\be\label{eq:strip with bottom}\strip^{\bw}=\left\{   (n,m)\in\strip, \forall\  0\leq k\leq N, n> n_k \text{ or } m> m_k     \right\},\ee
as well as its bottom $\hp^{\bw}_{\circ}$. See figure \eqref{subfig:down right path on a strip} for illustration.

\begin{definition}[Inhomogeneous inverse-Wishart polymer on a strip]\label{def:inhomogeneous strip model}
Let $N\in \mathbb{N}$ and $\bw=(w_1,\dots,w_N)\in\{ \rightarrow, \downarrow\}^{N}$ be a word of length $N$. Let $\bta=\lb \theta_1,\dots,\theta_N\rb\in \RR^N, u, v\in\RR $ be parameters satisfying one of the following two regimes: for all $1\leq i,j\leq N$,
     \begin{subequations}
\begin{align}
    \text{(maximal current regime)} \quad & u+v>\frac{d-1}{2}, \theta_i+\theta_j>\frac{d-1}{2},\  \theta_i+u>\frac{d-1}{2}, \ \theta_i+v>\frac{d-1}{2}; \label{eq:parameters inhomogeneous general regime}\\
    \text{(equilibrium regime)}\qquad &  u+v=0,\  \theta_i+u>\frac{d-1}{2}, \ \theta_i+v>\frac{d-1}{2}.
    \label{eq:parameters inhomogeneous equilibrium regime}
\end{align}
\end{subequations}
We adopt the convention that for all $ p\in \ZZ_{\geq 0}, \theta_p :=\theta_k$ for the unique $1\leq k\leq N$ such that $p\equiv k\  \pmod{N}$.  Consider  a family of independent random variables indexed by
the bulk vertices of the strip  $\strip^{\bw}$: for all $m<n<m+N$ with $ (n,m)\in\strip^{\bw}$, 
     \begin{align*}W(n,m)&\sim \Wisv(\theta_n+\theta_m),\\
     W(m,m)\sim \Wisv(\ta_m+u), &\qquad 
       W(m+N,m)\sim \Wisv(\ta_m+v).\end{align*}
     
An initial condition refers to a stochastic process $\bB=\lb B_k\rb_{0\leq k\leq N}$ that is independent of the above random variables.  
Along the bottom boundary $\hp^{\bw}_{\circ}=(\p_k)_{0\leq k\leq N}$, the partition functions of inverse-Wishart directed polymer on a strip is given by initial condition  $$Z^{\bta,u,v}(\p_k)=B_k,\quad 1\leq k\leq N. $$
For bulk points $(n,m)\in\strip^{\bw}$, the partition functions are defined by recurrence: 

\begin{align}\label{eq:recurrence inhomogeneous strip}
Z^{\bta,u,v}(n,m) =\; & W(n,m) \star 
\begin{dcases*}
Z^{\bta,u,v}(n{-}1,m) + Z^{\bta,u,v}(n,m{-}1) & if $m<n<m+N$, \\
Z^{\bta,u,v}(n,m{-}1) & if $m=n$, \\
Z^{\bta,u,v}(n{-}1,m) & if $n=m+N$.
\end{dcases*}
\end{align}

\end{definition}

\begin{definition}[Inhomogeneous two-layer matrix Whittaker process]\label{def: general two-layer matrix Whittaker process}
 Let $N\in \mathbb{N}$ and $\bw=(w_1,\dots,w_N)\in\{ \rightarrow, \downarrow\}^{N}$ be a word of length $N$. Fix a matrix
$\la_1^0\in\pd$ and parameters $\bg=\lb\ga_1,\dots,\ga_N\rb\in \RR^N, u, v\in\RR$  such that for all $ 1\leq i\leq N$,
    $$ u+v, \ga_i+u, \ga_i+v>\frac{d-1}{2}.$$   The (inhomogeneous) two-layer matrix Whittaker process starting from $\la_1^0$, denoted by $\d\PP_{\la^0_1}^{\bg,\bw}$, is the probability measure  on 
    $\pd\times(\pd^2)^{N}$ with density
    \be\label{eq:general two layer matrix whittaker process}
        \begin{split}
            \PP^{\bg,\bw}_{\la^0_1}\lb\la_2^0, (\la^i)_{1\leq i \leq N}\rb
            &=\frac{1}{\mathcal{Z}^{\bg,u,v}} \abs{\la_2^0(\la_1^0)^{-1}}^u\abs{\la_2^N(\la_1^N)^{-1}}^v
            \prod_{\substack{1\leq i\leq N\\ w_i= \rightarrow}}\Psi_{\ga_i}(\la^i/\la^{i-1})
    \prod_{\substack{1\leq i\leq N\\ w_i= \downarrow}}\Psi_{\ga_i}(\la^{i-1}/\la^{i})
        \end{split}
    \ee
with respect to the reference measure $ \d\mu(  \la^0_2) \prod_{i=1}^{N} \d\mu(  \la^i)$. The normalization constant $\mathcal{Z}^{\bg,u,v}$ is finite and does not depend on $\bw$ or the value of $\la_1^0$, as we will prove in Proposition~\ref{prop:Finiteness of Partition function}. Let $\bP^{\bg,\bw}_{\la^0_1}\lb  (\la^j_1)_{1\leq j \leq N}\rb$ denote the density of the marginal $\lb  (\la^j_1)_{1\leq j \leq N}\rb$  with respect to the reference measure $\prod_{j=1}^{N} \d\mu(  \la^j_1)$. In particular, when $\bg=(\ga)^N$ and $\bw=(\rightarrow)^N$, the inhomogeneous two-layer matrix Whittaker process and the corresponding first-layer marginal reduce to the homogeneous ones defined in Definition \ref{def: homogeneous two layer matrix whittaker process}.
    $$\d\PP^{\bg,\bw}_{\la^0_1}=\d\PP^{\ga,u,v}_{\la^0_1}, \quad\d\bP^{\bg,\bw}_{\la^0_1}=\d
    \bP^{\ga,u,v}_{\la^0_1}.$$  We will give the two-layer matrix Whittaker process $\d\PP^{\bg,\bw}_{\la^0_1}$ a probabilistic interpretation in Section \ref{subsec:Probabilistic description of two-layer matrix Whittaker process}. 
\end{definition}
We label the edges of the strip $\strip$ in such a way that for any down-right path $\hp$ on lattice, the parameters $\bg(\hp)=(\gamma_1,\dots,\gamma_N)$ along the path read
    \be\label{eq:labels on strip edges maximal current}\gamma_i= \left\{  \begin{aligned}
\ta_n \quad & \text{if } \p_{i-1}=(n-1,m), \p_{i}= (n,m) \\
\ta_m \quad & \text{if } \p_{i-1}=(n,m), \p_{i}=(n,m-1) 
\end{aligned}
\right., \quad 1\leq i\leq N.\ee
See Figure \eqref{subfig:down right path on a strip} for illustration.

The goal of this section is to show that, the infinite measure given by the first-layer marginal of the two-layer matrix Whittaker process with $\mu$-distributed starting matrix is consistent with the polymer recurrence in the following sense.  

\begin{theorem}\label{thm: inhomogeneous maximal current strip}
 Consider the inverse-Wishart polymer on a strip with parameters $\bta , u, v $   satisfying the maximal current regime \eqref{eq:parameters inhomogeneous general regime}. Suppose the initial condition along the bottom boundary  $\hp^{\bw}_{\circ}$  is  given by
    $$ \d\mu\big(B_0\big)\times \d\bP^{\bg(\hp^{\bw}_{\circ}),\bw(\hp^{\bw}_{\circ})}_{ B_0} \Big(\big(B_i\big)_{1\leq i \leq N}\Big).$$
    Then along any down-right path $\hp=(\p_k)_{1\leq k\leq N}$ above $\hp^{\bw}_{\circ}$, the partition function process is distributed as 
     $$ \d\mu\big( Z^{\bta,u,v}(\p_0)\big)\times \d\bP^{\bg(\hp),\bw(\hp)}_{ Z^{\bta,u,v}(\p_0)} \lb\big(Z^{\bta,u,v}(\p_i)\big)_{1\leq i \leq N}\rb.$$
\end{theorem}
Theorem \ref{thm: inhomogeneous maximal current strip} is proven in Section \ref{subsec:Proof of inhomogeneous strip max current}.
The key idea is to define Markovian dynamics for one-step updates via a family of transition kernels indexed by pairs of down-right paths that differ at a single vertex (Definition~\eqref{def:local Markov kernels}), and to construct a corresponding family of consistent measures in the sense specified by~\eqref{eq:compatibility of local operator with wt IW}. These consistent measures are parametrized by two-layer graphs associated with the down-right paths (Definition \ref{def:two layer Gibbs measure}). Precise definitions are provided in Section \ref{subsec:Two-layer graph and two-layer Gibbs measure}. This approach was originally introduced in \cite{BarraquandCorwinYang} in the context of stationary measures for the geometric last passage percolation and the log-gamma polymer on a strip.

\subsection{Two-layer graph and two-layer Gibbs measure}\label{subsec:Two-layer graph and two-layer Gibbs measure}
\begin{definition}[Two-layer graph]\label{def:two layer graph}   

\begin{figure}
\centering
\begin{subfigure}[b]{0.45\textwidth}
\begin{tikzpicture}[scale=0.87]
\draw[gray](-0.5,-0.5)--(4.5,4.5);
\draw[gray] (5.5,-0.5)--(10.5,4.5);
\draw[gray] (0,0)--(6,0);
\draw[gray] (1,1)--(7,1);
\draw[gray] (2,2)--(8,2);
\draw[gray] (3,3)--(9,3);   
\draw[gray] (4,4)--(10,4); 
\draw[gray] (1,-0.5)--(1,1);
\draw[gray] (2,-0.5)--(2,2);
\draw[gray] (3,-0.5)--(3,3);
\draw[gray] (4,-0.5)--(4,4);
\draw[gray] (5,-0.5)--(5,4.5);
\draw[gray] (6,-0.5)--(6,4.5);
\draw[gray] (7,1)--(7,4.5);
\draw[gray] (8,2)--(8,4.5); 
\draw[gray] (9,3)--(9,4.5);
\draw[gray] (10,4)--(10,4.5);
\draw [black] (3,3)--(4.5,4.5);
\draw   (6,0)--(10.5,4.5);
\draw   (5,1)--(7,1);
\draw   (5,2)--(8,2);
\draw   (3,3)--(9,3);   
\draw   (4,4)--(10,4);
\draw   (4,3)--(4,4);
\draw   (5,3)--(5,4.5);
\draw   (6,1)--(6,4.5);
\draw   (7,1)--(7,4.5);
\draw   (8,2)--(8,4.5); 
\draw   (9,3)--(9,4.5);
\draw   (10,4)--(10,4.5);
\node at (0.5,0) {\textcolor{red}{\small $\ta_4$}};
\node at (1.5,0) {\textcolor{red}{\small $\ta_5$}};
\node at (2.5,0){\textcolor{red}{\small $\ta_6$}};
\node at (3.5,0) {\textcolor{red}{\small $\ta_1$}};
\node at (4.5,0) {\textcolor{red}{\small $\ta_2$}};
\node at (5.5,0) {\textcolor{red}{\small $\ta_3$}};

\node at (1.5,1) {\textcolor{red}{\small $\ta_5$}};
\node at (2.5,1) {\textcolor{red}{\small $\ta_6$}};
\node at (3.5,1) {\textcolor{red}{\small $\ta_1$}};
\node at (4.5,1) {\textcolor{red}{\small $\ta_2$}};
\node at (5.5,1) {\textcolor{red}{\small $\ta_3$}};
\node at (6.5,1) {\textcolor{red}{\small $\ta_4$}};

\node at (2.5,2) {\textcolor{red}{\small $\ta_6$}};
\node at (3.5,2) {\textcolor{red}{\small $\ta_1$}};
\node at (4.5,2) {\textcolor{red}{\small $\ta_2$}};
\node at (5.5,2){\textcolor{red}{\small $\ta_3$}};
\node at (6.5,2) {\textcolor{red}{\small $\ta_4$}};
\node at (7.5,2) {\textcolor{red}{\small $\ta_5$}};

\node at (3.5,3) {\textcolor{red}{\small $\ta_1$}};
\node at (4.5,3) {\textcolor{red}{\small $\ta_2$}};
\node at (5.5,3) {\textcolor{red}{\small $\ta_3$}};
\node at (6.5,3) {\textcolor{red}{\small $\ta_4$}};
\node at (7.5,3) {\textcolor{red}{\small $\ta_5$}};
\node at (8.5,3) {\textcolor{red}{\small $\ta_6$}};

\node at (4.5,4) {\textcolor{red}{\small $\ta_2$}};
\node at (5.5,4) {\textcolor{red}{\small $\ta_3$}};
\node at (6.5,4) {\textcolor{red}{\small $\ta_4$}};
\node at (7.5,4) {\textcolor{red}{\small $\ta_5$}};
\node at (8.5,4) {\textcolor{red}{\small $\ta_6$}};
\node at (9.5,4) {\textcolor{red}{\small $\ta_1$}};

\node at (1,0.5){\textcolor{red}{\small $\ta_4$}};
\node at (2,0.5){\textcolor{red}{\small $\ta_4$}};
\node at (3,0.5){\textcolor{red}{\small $\ta_4$}};
\node at (4,0.5) {\textcolor{red}{\small $\ta_4$}};
\node at (5,0.5) {\textcolor{red}{\small $\ta_4$}};
\node at (6,0.5) {\textcolor{red}{\small $\ta_4$}};

\node at (2,1.5) {\textcolor{red}{\small $\ta_5$}};
\node at (3,1.5) {\textcolor{red}{\small $\ta_5$}};
\node at (4,1.5) {\textcolor{red}{\small $\ta_5$}};
\node at (5,1.5) {\textcolor{red}{\small $\ta_5$}};
\node at (6,1.5) {\textcolor{red}{\small $\ta_5$}};
\node at (7,1.5) {\textcolor{red}{\small $\ta_5$}};

\node at (3,2.5){\textcolor{red}{\small $\ta_6$}};
\node at (4,2.5) {\textcolor{red}{\small $\ta_6$}};
\node at (5,2.5) {\textcolor{red}{\small $\ta_6$}};
\node at (6,2.5) {\textcolor{red}{\small $\ta_6$}};
\node at (7,2.5){\textcolor{red}{\small $\ta_6$}};
\node at (8,2.5){\textcolor{red}{\small $\ta_6$}};

\node at (4,3.5){\textcolor{red}{\small $\ta_1$}};
\node at (5,3.5) {\textcolor{red}{\small $\ta_1$}};
\node at (6,3.5) {\textcolor{red}{\small $\ta_1$}};
\node at (7,3.5) {\textcolor{red}{\small $\ta_1$}};
\node at (8,3.5){\textcolor{red}{\small $\ta_1$}};
\node at (9,3.5){\textcolor{red}{\small $\ta_1$}};

\draw[very thick] (3,3)--(5,3)--(5,1)--(6,1)--(6,0);
\fill (3,3) circle(0.08);
\node[left] at (3,3) {\small $(0,0)=\p_0$};
\fill (4,3) circle(0.08);
\fill (5,3) circle(0.08);
\fill (5,2) circle(0.08);
\fill (5,1) circle(0.08);
\fill (6,1) circle(0.08);
\fill (6,0) circle(0.08);
\node[below] at (6,0) {\small $\p_6$};
\end{tikzpicture}
         \caption{}
         \label{subfig:down right path on a strip}
 \end{subfigure} 
\hfill
\begin{subfigure}{0.45\textwidth}
\begin{tikzpicture}[scale=0.78]
			\draw[thick] (0,0) -- (2,2) -- (4,0) -- (5,1) -- (6,0);
			\draw[thick] (0,2) -- (2,4) -- (4,2) -- (5,3) -- (6,2);
			\draw[thick] (0,0) arc(270:90:1);
			\draw[thick] (6,2) arc(90:-90:1);
			\fill (0,0) circle(0.1);\node[below] at (0,0) {\small $\la_2^{0}$};
			\fill (1,1) circle(0.1);\node[below] at (1,1) {\small $\la_2^{1}$};
			\fill (2,2) circle(0.1);\node[below] at (2,2) {\small $\la_2^{2}$};
			\fill (3,1) circle(0.1);\node[below] at (3,1) {\small $\la_2^{3}$};
			\fill (4,0) circle(0.1);\node[below] at (4,0) {\small $\la_2^{4}$};
			\fill (5,1) circle(0.1);\node[below] at (5,1) {\small $\la_2^{5}$};
			\fill (6,0) circle(0.1);\node[below] at (6,0) {\small $\la_2^{6}$}; 
\fill (0,2) circle(0.1);\node[above] at (0,2) {\small $\la_1^{0}$};
\fill (1,3) circle(0.1);\node[above] at (1,3) {\small $\la_1^{1}$};
\fill (2,4) circle(0.1);\node[above] at (2,4) {\small $\la_1^{2}$};
\fill (3,3) circle(0.1);\node[above] at (3,3) {\small $\la_1^{3}$};
\fill (4,2) circle(0.1);\node[above] at (4,2) {\small $\la_1^{4}$};
\fill (5,3) circle(0.1);\node[above] at (5,3) {\small $\la_1^{5}$};
\fill (6,2) circle(0.1);\node[above] at (6,2) {\small $\la_1^{6}$}; 
\draw[dashed] (0,2) -- (1,1);
\draw[dashed] (1,3) -- (2,2);
\draw[dashed] (2,2) -- (3,3);
\draw[dashed] (3,1) -- (4,2);
\draw[dashed] (4,2) -- (5,1);
\draw[dashed] (5,1) -- (6,2); 
\node[left] at (-1,1) {\textcolor{red}{\small $u$}};
\node[right] at (7,1) {\textcolor{red}{\small $v$}};
\node[above] at (0.5,2.5) {\textcolor{red}{\small $\ta_1$}};
\node[above] at (1.5,3.5) {\textcolor{red}{\small $\ta_2$}};
\node[above] at (2.5,3.5) {\textcolor{red}{\small $\ta_6$}};
\node[above] at (3.5,2.5) {\textcolor{red}{\small $\ta_5$}};
\node[above] at (4.5,2.5) {\textcolor{red}{\small $\ta_3$}};
\node[above] at (5.5,2.5) {\textcolor{red}{\small $\ta_4$}};
\node[below] at (0.5,0.5) {\textcolor{red}{\small $\ta_1$}};
\node[below] at (1.5,1.5) {\textcolor{red}{\small $\ta_2$}};
\node[below] at (2.5,1.5) {\textcolor{red}{\small $\ta_6$}};
\node[below] at (3.5,0.5) {\textcolor{red}{\small $\ta_5$}};
\node[below] at (4.5,0.5) {\textcolor{red}{\small $\ta_3$}};
\node[below] at (5.5,0.5) {\textcolor{red}{\small $\ta_4$}};
\end{tikzpicture} 
  \caption{}
         \label{subfig:two layer Gibbs graph}
     \end{subfigure}
    \caption{(a): Strip $\mathbb{S}_{6}^{\bw}$ with the bottom path $\hp^{\bw}_{\circ}=(\p_0,\dots,\p_6)$ highlighted. We only consider the partition functions indexed by vertices above. Edge labels are assigned according to rule~\eqref{eq:labels on strip edges maximal current} (b) The two-layer graph $\gp^{\bw}_{\circ}$ associated with $\hp^{\bw}_{\circ}$ (up to a counterclockwise rotation by $\frac{\pi}{4}$), on which a two-layer configuration $\bl$ is specified.}
    \label{fig:down right path on a strip and its two layer graph}
\end{figure}
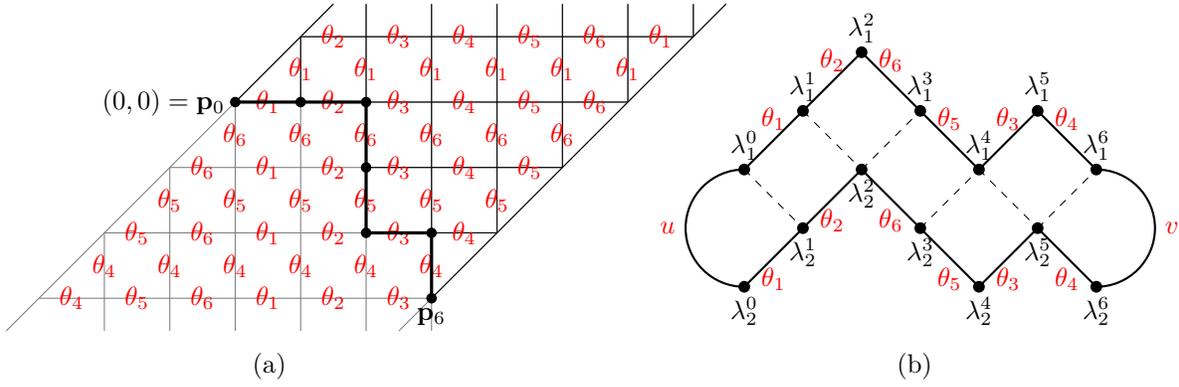
For a given down-right path $\hp$, we associated it with a two-layer graph $\gp$ composed of two parallel copies of   $\hp$, called the first and second layers, connected by dotted edges and enclosed by boundary arcs. See Figure~\ref{subfig:two layer Gibbs graph} for an example corresponding to the highlighted path in Figure~\ref{subfig:down right path on a strip}.

Formally, we define the associated two-layer graph $\gp = (V, E)$ as an undirected graph, where each edge $e \in E$ is assigned with some real number $\bg(e)$, referred to as its edge label. They are defined precisely as follows:
\begin{itemize}
    \item \textbf{Vertex set:}
    \begin{equation*}
        V = \{ \p_i, \p_i' : 0 \leq i \leq N \}, \quad \text{where } \p'_i:= \tau_{(-1,-1)}\p_i.
    \end{equation*}
    \item \textbf{Edge set:}
    \begin{equation*}
        \begin{aligned}
    E = & \;\{\{\p_{i-1}, \p_{i}\}, \{\p_{i-1}', \p_{i}'\} : 1 \leq i \leq N\} \cup \{\{\p_0, \p_0'\}, \{\p_N, \p_N'\}\},\\
    & \cup \{\{\p_i, \p_j'\} : 0 \leq i,j \leq N, |\p_i - \p_j'| = 1\} 
    \end{aligned}
    \end{equation*}
    where $|\cdot|$ refers to the Euclidean distance in the lattice.
    \item \textbf{Edge labels:}  
    \begin{itemize}
        \item $\bg(\{\p_i, \p_{i+1}\}) = \bg(\{\p_i', \p_{i+1}'\}) = \gamma_i$,  where $\gamma_i$ is defined by the labeling rule in~\eqref{eq:labels on strip edges maximal current} for $0\leq i\leq N-1$;
        \item $\bg(\{\p_i, \p_j'\}) = 0$ for $0\leq i,j\leq N$ such that $|\p_i - \p_j'| = 1$;
        \item $\bg(\{\p_0, \p_0'\}) = u$, \quad $\bg(\{\p_N, \p_N'\}) = v$.
    \end{itemize}
\end{itemize}
We use solid, dotted, and arc-shaped edges to represent edges with the three types of labels in the two-layer graph, see Figure~\eqref{subfig:two layer Gibbs graph} for an example.
We emphasize that an edge of label zero does not imply the absence of the corresponding edge.  
\end{definition}

\begin{definition}[Local Boltzmann weight and two-layer Gibbs measure]\label{def:two layer Gibbs measure}

A two-layer configuration $\bl=\begin{pmatrix}
    \la_1^j , \la_2^j
\end{pmatrix}_{0\leq j\leq N}$ refers to an assignment $\la_i^j\in\pd$ to each vertex of $\gp$. See Figure \eqref{subfig:two layer Gibbs graph}  for illustration. For a two-layer configuration $\bl$ on a two-layer graph $\gp$, we define the local Boltzmann weight of lines (solid, dotted or arc) appearing in the graph respectively as: 
\begin{subequations}
\label{eq:Boltzmann weights IW}
    \be\label{eq:arc weight IW}
\wt\lb\raisebox{-4.5pt}{\arcleft}\rb=\abs{yx^{-1}}^u,\quad \wt\lb\raisebox{-4.5pt}{\arcright}\rb=\abs{yx^{-1}}^v,\ee
\be  \label{eq:solid weight IW}
\wt\lb\bulkright\rb=\wt\lb\bulkleft\rb=\abs{xy^{-1}}^{-\ga} e^{-\tr[yx^{-1}]},\ee
\be  \label{eq:dotted weight IW}
\wt\lb\bulkrightdotted\rb=\wt\lb\bulkleftdotted\rb=e^{-\tr[yx^{-1}]}.\ee
\end{subequations}
The Boltzmann weight in \eqref{eq:solid weight IW} coincides with the densities of the inverse-Wishart distribution with parameter $\gamma$, evaluated at $x \star y^{-1}$ or $y \star x^{-1}$. In the scalar case $d=1$, up to a Cole-Hopf transformation $x'=e^x$, the Boltzmann weights in \eqref{eq:Boltzmann weights IW} reduce to the log-gamma weights introduced in \cite{BarraquandCorwinYang}.

Let $\mathcal{G}\subset \gp$ be a subgraph, we denote the number of its vertices by $\abs{\mathcal{G}}$. Let us define a function $\wt^{\mathcal{G}}:\bl\in\pd^{\abs{\mathcal{G}}}\mapsto \wt^{\mathcal{G}}(\bl)\in\RR_{>0}$  as the products of Boltzmann weight of all lines appearing in $\mathcal{G}$. For example, the  function $\Psi_{\ga}(\la/\mu)$ defined in \eqref{eq:Baxter Q operator} can be alternatively  expressed as:

$$\Psi_{\ga}(\la/\mu)=\wt\lb\vcenter{\hbox{
\raisebox{-35pt}{\begin{tikzpicture}
       \draw[thick] (-0.6,0.6)--(0,0);
       \draw[thick] (-0.6,-0.6)--(0,-1.2);
       \draw[dashed] (-0.6,-0.6)--(0,0);
        \node[ left] at (-0.5,0.55) {\small   $\la_1$};
        \node[ left] at (-0.5,-0.65) {\small  $\la_2$};
        \node[above] at (-0.3,0.3) {\textcolor{red}{\small  $\ga$}};
        \node[above] at (-0.3,-0.9) {\textcolor{red}{\small  $\ga$}};
        \node[ right] at (0,0) {\small  $\mu_1$};
        \node[ right] at (0,-1.2) {\small  $\mu_2$};
\end{tikzpicture}}
}}\rb=\wt\lb
\vcenter{\hbox{\raisebox{-35pt}{\begin{tikzpicture}
       \draw[thick] (0,0)--(0.6,0.6);
       \draw[thick](0,-1.2)--(0.6,-0.6);
       \draw[dashed] (0,0)--(0.6,-0.6);
        \node[ right] at (0.5,0.55) {\small $\la_1$};
        \node[ right] at (0.5,-0.65) {\small  $\la_2$};       
        \node[above] at (0.3,0.25) {\textcolor{red}{\small   $\ga$}};
        \node[above] at (0.3,-0.95) {\textcolor{red}{\small  $\ga$}};
        \node[left] at (0,0) {\small  $\mu_1$};
        \node[left] at (0,-1.2) {\small  $\mu_2$};
\end{tikzpicture}}
}}\rb.
$$
The Gibbs measure $\wt^{\mathcal{G}}(\d\bl)$ associated with a subgraph $\mathcal{G}\subset\gp$ is the $\sigma$-finite measure on $\pd^{\abs{\mathcal{G}}}$ with  density $\wt^{\mathcal{G}}(\bl)$ with respect to the reference measure $\d\mu(\bl)$. 

Note that when $\mathcal{G}=\gp$, up to the normalization constant $\mathcal{Z}^{\bta,u,v}$ that we will define in \eqref{eq: partition function}, the two-layer Gibbs measure $\wt^{\gp}(\d\bl)$ coincides with the two-layer matrix Whittaker process with $\mu$-distributed starting matrix:
\be\label{eq:maximal current stationary measure is wt gp}\frac{1}{\mathcal{Z}^{\bta,u,v}}\wt^{\gp}(\bl)\d\mu(\bl)=\d\mu(\la^0_1)\times \d\PP^{\bg(\hp),\bw(\hp)}_{\la^0_1}\lb\la_2^0, (\la^i)_{1\leq i \leq N}\rb.\ee 
Recall that an equality between two $\sigma$-finite Radon measures means that they have same finite value on every compact Borel set.
\end{definition}

\subsection{Properties of Gibbs measure}
\label{subsec:Properties of Gibbs measure}
In this section, we prove some properties of the specially chosen Boltzmann weights \eqref{eq:Boltzmann weights IW}, which turn out to be useful in defining the dynamics.

\begin{lemma}[Translation invariance]
    Let $\hp$ be a down-right path, let $\mathcal{G}$ be a subgraph of $ \gp$. The associated Gibbs measure  $\wt^{\mathcal{G}}(\d\bl)$  is (right-$\star$) translation invariant: for all $x\in\pd, \bl\in \pd^{\abs{\mathcal{G}}}$, 
\be\label{eq:translation invariance}\wt^{\mathcal{G}}(\bl)=\wt^{\mathcal{G}}(\bl\star x),\ee
where we write $\bl\star x= ( \la^v\star x)_{v\in\mathcal{G}}$.

In particular, when $\mathcal{G}=\gp$, the two-Gibbs measure is translation invariant: for all $x\in\pd,  \ \bl\in \pd^{2N+2}$, we have 
$$\wt^{\gp}(\bl)=\wt^{\gp}(\bl\star x),$$
where $\bl\star x= ( \la_1^j\star x , \la_2^j\star x)_{0\leq j\leq N}$. 
\end{lemma}
\begin{proof}
    Because for all $y,z\in\pd$, we have $$\forall x\in\pd, \quad \abs{(y\star x)(z\star x)^{-1}}=\abs{yz^{-1}},\quad \tr[(y\star x)(z\star x)^{-1}]=\tr[yx^{-1}].$$
    The local Boltzmann weight of a single line is preserved under congruence. As products of local Boltzmann weights, the Gibbs measure is also preserved. 
\end{proof}

\begin{lemma}[Symmetry of skew matrix Whittaker function]\label{lem:symmetry of skew whittaker function} Recall the skew matrix Whittaker function defined in \eqref{eq:branching rule, skew whittaker function}. 
For all parameters $\a,\b\in\RR$ and arguments $\la,\mu\in\pd^2$, we have 
$$\Psi_{\a, \b}(\la/\mu)=\Psi_{\b,\a}(\la/\mu) ,$$
with both sides finite. Equivalently,
$$\int_{\pd^2}\wt\lb
\vcenter{\hbox{\raisebox{-35pt}{\begin{tikzpicture}
    \draw[thick] (-0.6,0.6)--(0,1.2)--(0.6,1.8);
    \draw[dashed] (-0.6,0.6)--(0,0);
    \draw[dashed] (0,1.2)--(0.6,0.6);
    \draw[thick] (-0.6,-0.6)--(0,0)--(0.6,0.6);
    \node[below right] at (0.5,0.65) {\small $\la_2$};
    \node[above right] at (0.5,1.75) {\small   $\la_1$};
    \node[above left] at (-0.45,0.66) {\small  $\mu_1$};
    \node[below left] at (-0.45,-0.55) {\small  $\mu_2$};
    \node[below] at (0.1,0) {\small  $\k_2$};
    \node[above] at (-0.1,1.2) {\small  $\k_1$};
    \node[above] at (0.3,0.25){\textcolor{red}{\small  $\b$}};
    \node[above] at (0.3,1.4){\textcolor{red}{\small  $\b$}};
    \node[above] at (-0.35,-0.37){\textcolor{red}{\small  $\a$}};
    \node[above] at (-0.35,0.83){\textcolor{red}{\small  $\a$}};
\end{tikzpicture}}
}}\rb\d\mu(\k)=
        \int_{\pd^2}\wt\lb\vcenter{\hbox{
        \raisebox{-35pt}{\begin{tikzpicture}
    \draw[thick] (-0.6,0.6)--(0,1.2)--(0.6,1.8);
    \draw[dashed] (-0.6,0.6)--(0,0);
    \draw[dashed] (0,1.2)--(0.6,0.6);
    \draw[thick] (-0.6,-0.6)--(0,0)--(0.6,0.6);
    \node[below right] at (0.5,0.65) {\small $\la_2$};
    \node[above right] at (0.5,1.75) {\small   $\la_1$};
    \node[above left] at (-0.45,0.66) {\small  $\mu_1$};
    \node[below left] at (-0.45,-0.55) {\small  $\mu_2$};
    \node[below] at (0.1,0) {\small  $\k_2$};
    \node[above] at (-0.1,1.2) {\small  $\k_1$};
    \node[above] at (0.3,0.28){\textcolor{red}{\small  $\a$}};
    \node[above] at (0.3,1.5){\textcolor{red}{\small  $\a$}};
    \node[above] at (-0.35,-0.4){\textcolor{red}{\small  $\b$}};
    \node[above] at (-0.35,0.8){\textcolor{red}{\small  $\b$}};
\end{tikzpicture}}
}}\rb\d\mu(  \k).$$
    More generally, for any  parameters $\ga_1,\dots, \ga_N\in\RR$ and arguments $\la, \mu\in\pd^2$, let  $\sigma$ be a permutation 
 on $\{1, \dots, n\}$,  we have 
    $$\Psi_{\ga_1,\dots,\ga_N}(\la/\mu)=\Psi_{\ga_{\sigma(1)},\dots,\ga_{\sigma(N)}}(\la/\mu).$$
\end{lemma}
\begin{proof} Recall the definition of Whittaker function of matrix arguments \eqref{eq:whittaker function}, a direct computation shows that
\begin{align*}
\Psi_{\a, \b}(\la/\mu)
&= \int_{\pd^2} \Psi_{\a}(\k/\mu) \Psi_{\b}(\la/\k) \d\mu(  \k) \\
&= \abs{\la_1^{-1}}^{\b} \abs{\mu_1}^{\a} 
    \int_{\pd} \abs{\k_1}^{\b - \a} 
    e^{-\tr[\k_1 \la_1^{-1} + \k_1^{-1}(\mu_1 + \la_2)]} \d\mu(  \k_1) \\
&\quad \times
    \abs{\la_2^{-1}}^{\b} \abs{\mu_2}^{\a}
    \int_{\pd} \abs{\k_2}^{\b - \a}
    e^{-\tr[\k_2(\la_2^{-1} + \mu_1^{-1}) + \k_2^{-1} \mu_2]} \d\mu(  \k_2) \\
&= \abs{\la_1^{-1}}^{\b} \abs{\mu_1}^{\a} 
    \abs{\la_2^{-1}}^{\b} \abs{\mu_2}^{\a}
    \abs{\la_1(\mu_1 + \la_2)}^{\b}
    \abs{\mu_2(\la_2^{-1} + \mu_1^{-1})^{-1}}^{\b} \\
&\quad \times \psi_{(\a,\b)}\left(\la_1^{-1}, \mu_1 + \la_2\right)
    \psi_{(\a,\b)}\left(\la_2^{-1} + \mu_1^{-1}, \mu_2\right).
\end{align*}

The determinant factors simplify to $\abs{\mu_1\mu_2}^{\a+\b}$, which is invariant under $(\a,\b)\mapsto(\b,\a)$. Recall that the Whittaker function of matrix arguments $\psi_{(\a,\b)}$ is well-defined for all $\a,\b\in\RR$ and is symmetric in its parameters (see \eqref{eq:symmetry of whittaker function}),  hence $\Psi_{\a, \b}(\la/\mu)=\Psi_{\b, \a}(\la/\mu)$ and is finite as desired.

For general skew Whittaker functions, note that any permutation on $\{1,\dots, N\}$ can be generated by a finite sequence of elementary transpositions $\{(i,i+1) :  1\leq i\leq N\}$. By repeatedly applying the first part of the lemma, it follows that $\Psi_{\ga_1,\dots,\ga_N}$ is symmetric in its arguments.
\end{proof}

\begin{lemma}[Cauchy type identity]\label{lem:Cauchy identity}
For parameters $\alpha,\beta\in \RR$ with $\alpha+\beta>\frac{d-1}{2}$ and arguments $\la,\mu\in\pd^2$, we have  
\be\label{eq:Cauchy identity}\int_{\pd^2}\Psi_{\b}(\pi/\la)\Psi_{\a}(\pi/\mu)\d\mu(  \pi)=\int_{\pd^2}\Psi_{\a}(\la/\k)\Psi_{\b}(\mu/\k)\d\mu(  \k), \ee
with both integrals finite. Equivalently, 
$$
        \int_{\pd^2}\wt\lb
        \vcenter{\hbox{\raisebox{-35pt}{\begin{tikzpicture}
    \draw[thick] (-0.6,0.6)--(0,1.2)--(0.6,0.6);
    \draw[dashed] (-0.6,0.6)--(0,0)--(0.6,0.6);
    \draw[thick] (-0.6,-0.6)--(0,0)--(0.6,-0.6);
    \node[below right] at (0.5,0.65) {\small $\mu_1$};
    \node[below right] at (0.5,-0.55) {\small   $\mu_2$};
    \node[below left] at (-0.45,0.66) {\small  $\la_1$};
    \node[below left] at (-0.45,-0.55) {\small  $\la_2$};
    \node[above] at (0,0) {\small  $\pi_2$};
    \node[above] at (0,1.2) {\small  $\pi_1$};
    \node[above] at (0.35,-0.35){\textcolor{red}{\small  $\a$}};
    \node[above] at (0.35,0.85){\textcolor{red}{\small  $\a$}};
    \node[above] at (-0.35,-0.4){\textcolor{red}{\small  $\b$}};
    \node[above] at (-0.35,0.8){\textcolor{red}{\small  $\b$}};
\end{tikzpicture}}
}}\rb\d\mu(  \pi)=\int_{\pd^2}\wt\lb
\vcenter{\hbox{\raisebox{-35pt}{\begin{tikzpicture}
       \draw[thick] (-0.6,0.6)--(0,0)--(0.6,0.6);
       \draw[thick] (-0.6,-0.6)--(0,-1.2)--(0.6,-0.6);
       \draw[dashed] (-0.6,-0.6)--(0,0)--(0.6,-0.6);
        \node[above right] at (0.5,0.55) {\small $\mu_1$};
        \node[above right] at (0.5,-0.65) {\small  $\mu_2$};
        \node[above left] at (-0.5,0.55) {\small   $\la_1$};
        \node[above left] at (-0.5,-0.65) {\small  $\la_2$};
        \node[above] at (0.3,0.25) {\textcolor{red}{\small   $\b$}};
        \node[above] at (0.3,-0.95) {\textcolor{red}{\small  $\b$}};
        \node[above] at (-0.3,0.3) {\textcolor{red}{\small  $\a$}};
        \node[above] at (-0.3,-0.9) {\textcolor{red}{\small  $\a$}};
        \node[below] at (0,0) {\small  $\k_1$};
        \node[below] at (0,-1.2) {\small  $\k_2$};
\end{tikzpicture}}
}}\rb\d\mu(  \k).$$
\end{lemma}

\begin{proof} Let $V=\la_1^{-1}+\mu_{1}^{-1}$ and $W=\la_2+\mu_2$. Using the integral formula \eqref{eq:laplace transform}, we have
\be\label{eq:Bessel integral 1}\text{LHS}\eqref{eq:Cauchy identity}=
\abs{\la_1\la_2}^{\a}\abs{\mu_1\mu_2}^{\b}\abs{(\la_1+\mu_1)^{-1}}^{\a+\b}\Gamma_d(\a+\b)\int_{\pd}\abs{\pi_2^{-1}}^{\a+\b}e^{-\tr[V\pi_2+W\pi_2^{-1}]}\d\mu(\pi_2),
\ee
\be\label{eq:Bessel integral 2}\text{RHS}\eqref{eq:Cauchy identity}=
\abs{\la_1^{-1}\la_2^{-1}}^{\b}\abs{\mu^{-1}_1\mu_2^{-1}}^{\a}\abs{\la_2^{-1}+\mu_2^{-1}}^{-(\a+\b)}\Gamma_d(\a+\b)\int_{\pd}\abs{\k_1}^{\a+\b}e^{-\tr[V\k_1+W\k_1^{-1}]}\d\mu(\k_1).
\ee
Let $U=W^{\frac{1}{2}}VW^{\frac{1}{2}}$, and consider the change of variable: \be\label{eq:smart change of variable}
\pi_2^{-1}:=W^{-\frac{1}{2}}U^{\frac{1}{2}}W^{-\frac{1}{2}}\k_1 W^{-\frac{1}{2}}U^{\frac{1}{2}}W^{-\frac{1}{2}}\in\pd.
\ee
By Lemma \ref{lem:change of variable}, we have $\d\mu(\pi_2)=\d\mu(\k_1)$, and $\tr[V\pi_2]=\tr[W\k_1^{-1}], \tr[W\pi_2^{-1}]=\tr[V\k_1]$. The determinantal  factors also match:  
$$\abs{\la_1\la_2}^{\a}\abs{\mu_1\mu_2}^{\b}\abs{(\la_1+\mu_1)^{-1}}^{\a+\b}\abs{W^{-1}V}^{\a+\b}=\abs{\la_1^{-1}\la_2^{-1}}^{\b}\abs{\mu^{-1}_1\mu_2^{-1}}^{\a}\abs{\la_2^{-1}+\mu_2^{-1}}^{-(\a+\b)}.$$
Finally, the finiteness of the integral factor is ensured under the assumption $\a+\b>\frac{d-1}{2}$. For instance,
$$\int_{\pd}\abs{\pi_2^{-1}}^{\a+\b}e^{-\tr[V\pi_2+W\pi_2^{-1}]}\d\mu(\pi_2)<\int_{\pd}\abs{\pi_2^{-1}}^{\a+\b}e^{-\tr[W\pi_2^{-1}]}\d\mu(\pi_2)=\abs{W^{-1}}^{\a+\b}\Gamma_d(\a+\b)<\infty.$$
\end{proof}

\begin{remark}
The two integrals in \eqref{eq:Bessel integral 1} and \eqref{eq:Bessel integral 2} can be identified with the $K$-Bessel functions defined in \eqref{eq:K-Bessel function}. Consequently, Lemma \ref{lem:Cauchy identity} could alternatively be proved by invoking the well-definedness and symmetry property of Whittaker functions \eqref{eq:symmetry of whittaker function}, following the same approach in the proof of Lemma~\ref{lem:symmetry of skew whittaker function}. 
Here, we adopt an equivalent proof based on a change of variables, a general method in deriving certain Cauchy/Littlewood-type identities, such as those arising in the context of the geometric last passage percolation and the log-gamma polymer \cite{BarraquandCorwinYang}.
In our setting, a straightforward change of variables  such as $\pi_2^{-1}=(\k_1\star V)\star W^{-1}$ would fail due to the lack of commutativity. The appropriate change of variable \eqref{eq:smart change of variable} is motivated by the symmetry of K-Bessel functions and Whittaker functions of matrix arguments, as explained in Section \ref{subsec: preliminary}.
\end{remark}

\begin{remark}
   By the same line of reasoning, the Cauchy type identity also holds for Whittaker functions with $n$ coordinates. Precisely, for $\la, \mu\in \pd^n$, the same Cauchy type identity \eqref{eq:Cauchy identity} holds for Whittaker function defined by:
\begin{equation*}
    \Psi_{\ga}(\la/\mu)=\prod_{i=1}^{n} \abs{\mu_i(\la_i)^{-1}}^\ga \cdot
      e^{-\tr[\sum_{i=1}^{n}\mu_i(\la_i)^{-1}+\sum_{i=1}^{n-1}\la_{i+1}(\mu_i)^{-1}]}.
\end{equation*}
\end{remark}

\begin{lemma}[Littlewood type identity]\label{lem:Littlewood identity}
For parameters $u,\alpha\in \RR$ with $u+\alpha>\frac{d-1}{2}$ and all $\k\in\pd^2$, we have 
\be \label{eq:Littlewood identity}
\int_{\pd^2}\abs{\pi_2\pi_1^{-1}}^{u}\Psi_{\a}(\pi/\k)\d\mu( \pi)=\int_{\pd^2}\abs{\la_2\la_1^{-1}}^{u}\Psi_{\a}(\k/\la)\d\mu( \la),
\ee
with both integrals finite. Equivalently, 
$$
\int_{\pd^2}\wt\lb\raisebox{-35pt}{\begin{tikzpicture}
    \draw[dashed](0.6,0)--(0,-0.6);
    \draw[thick] (0,0.6)--(0.6,0);
    \draw[thick] (0,-0.6)--(0.6,-1.2);
    \node[below right] at (0.5,-1.1) {\small $\k_2$};
    \node[below right] at (0.5,.1) {\small $\k_1$};
    \node[below] at (0,-0.6) {\small $\pi_2$};
    \node[above] at (0,0.6) {\small $\pi_1$};
    \node[above] at (0.35,0.3) {\textcolor{red}{\small $\alpha$}};
    \node[above] at (0.35,-0.9) {\textcolor{red}{\small $\alpha$}};
    \node[left] at (-0.6,0) {\textcolor{red}{\small $u$}};
    \draw[thick] (0,0.6) arc (90:270:0.6);
\end{tikzpicture}}\rb\d\mu( \pi)=
 \int_{\pd^2}\wt\lb\raisebox{-35pt}{\begin{tikzpicture}
    \draw[dashed](0,0.6)--(0.6,0);
    \draw[thick] (0,-0.6)--(0.6,0);
    \draw[thick] (0,0.6)--(0.6,1.2);
    \node[above right] at (0.5,-.1) {\small $\k_2$};
    \node[above right] at (0.5,1.1) {\small $\k_1$};
    \node[below] at (0,-0.6) {\small $\la_2$};
    \node[above] at (0,0.6) {\small $\la_1$};
    \node[above] at (0.3,-0.3) {\textcolor{red}{\small $\alpha$}};
    \node[above] at (0.3,0.9) {\textcolor{red}{\small $\alpha$}};
    \node[left] at (-0.6,0) {\textcolor{red}{\small $u$}};
    \draw[thick] (0,0.6) arc (90:270:0.6);
\end{tikzpicture}} \rb\d\mu( \la).$$
\end{lemma}

\begin{proof}
Using integral formula \eqref{eq:laplace transform}, we have
$$\text{LHS}\eqref{eq:Littlewood identity}=
\abs{\k_1}^{-u}\abs{\k_2}^{\a}\Gamma_d(\a+u)\int_{\pd}\abs{\pi_2}^{u-\a}e^{-\tr[\k_1^{-1}\pi_2+\k_2\pi_2^{-1}]}\d\mu(\pi_2),$$

$$\text{RHS}\eqref{eq:Littlewood identity}=
\abs{\k_1}^{-\a} \abs{\k_2}^{u}\Gamma_d(\a+u)\int_{\pd}\abs{\la_1}^{\a-u}e^{-\tr[\k_1^{-1}\la_1+\k_2\la_1^{-1}]}\d\mu(\la_1).
$$
Let $U=\k_2^{\frac{1}{2}}\k_1^{-1}\k_2^{\frac{1}{2}}$, and consider the change of variable: 
$$
\pi_2^{-1}:=\k_2^{-\frac{1}{2}}U^{\frac{1}{2}}\k_2^{-\frac{1}{2}}\la_1 \k_2^{-\frac{1}{2}}U^{\frac{1}{2}}\k_2^{-\frac{1}{2}}.
$$
By Lemma \ref{lem:change of variable}, we have $\d\mu(\pi_2)=\d\mu(\la_1)$, and $\tr[\k_2\pi_2^{-1}]=\tr[\k_1^{-1}\la_1]$,  $\tr[\k_1^{-1}\pi_2]=\tr[\k_2\la_1^{-1}]$. The determinantal factors match:  
$$\abs{\k_1}^{-u}\abs{\k_2}^{\a}\abs{\k_2^2U^{-1}}^{u-\a}=\abs{\k_1}^{-\a} \abs{\k_2}^{u}.$$
Finally, the finiteness of the integral factors is ensured under the assumption $\a+\b>\frac{d-1}{2}$. For instance,
$$\int_{\pd}\abs{\pi_2}^{u-\a}e^{-\tr[\k_1^{-1}\pi_2+\k_2\pi_2^{-1}]}\d\mu( \pi_2)<\int_{\pd}\abs{\pi_2}^{u-\a}e^{-\tr[\k_2\pi_2^{-1}]}\d\mu( \pi_2)=\abs{\k_2}^{u-\a}\Gamma_d(\a-u)<\infty.$$
\end{proof}
\begin{lemma}\label{lem:induction inequality}
    For all $u, v, \a\in\RR$ with $\min \lb \a+u,\a+v, u+v \rb >\frac{d-1}{2}$, there exists constants $C>0$ and $w\in\RR$ with $\min \lb w+u, \a+w \rb >\frac{d-1}{2}$, such that for all $\la_1, \la_2\in\pd$, we have 
\begin{equation}\label{eq:induction inequality}
     \int_{\pd^2}\wt\lb\raisebox{-40pt}{\begin{tikzpicture}
    \draw[dashed](0,0.6)--(0.6,0);
    \draw[thick] (0,-0.6)--(0.6,0);
    \draw[thick] (0,0.6)--(0.6,1.2);
    \node[below right] at (0.5,-0) {\small $\mu_2$};
    \node[above right] at (0.5,1.1) {\small $\mu_1$};
    \node[below] at (-0.1,-0.6) {\small $\la_2$};
    \node[above] at (-0.1,0.6) {\small $\la_1$};
    \node[above] at (0.3,-0.3) {\textcolor{red}{\small $\alpha$}};
    \node[below] at (0.3,0.95) {\textcolor{red}{\small $\alpha$}};
    \node[right] at (1.2,0.6) {\textcolor{red}{\small $v$}};
    \draw[thick] (0.6,0) arc (-90:90:0.6);
\end{tikzpicture}}\rb\d\mu( \mu_1)\d\mu( \mu_2)
\leq C\wt\lb\raisebox{-33pt}{\begin{tikzpicture}
    \node[below] at (0.2,-0.6) {\small $\la_2$};
    \node[above] at (0.2,0.6) {\small $\la_1$};
    \node[right] at (0.6,0){\textcolor{red}{\small $w$}};
    \draw[thick] (0,-0.6) arc (-90:90:0.6);
\end{tikzpicture}}\rb.
\end{equation}

\end{lemma}
\begin{proof}
Integrate out $\mu_1$ using formula \eqref{eq:laplace transform}, the inequality \eqref{eq:induction inequality} reduces to
$$\abs{\la_1^{-1}}^{v}\abs{\la_2}^{\a}\int_{\pd}\abs{\mu_2}^{v-\a}e^{-\tr[\mu_2\la_1^{-1}+\la_2\mu_2^{-1}]}\d\mu_2\leq C \abs{\la_2\la_1^{-1}}^{w}.$$
(The constant $C$ may vary throughout the proof.) Denote $ \Delta:=\la_1^{-1}\star\la_2$, and consider the change of variable $s=\mu_2^{-1}\star \la_2$. The original lemma is equivalent to finding the required $C,w$ such that for all $\Delta\in\pd$, 
\be\label{eq:inequality for induction}\int_{\pd} \abs{s}^{\a-v}e^{-\tr[s+\Delta s^{-1}]}\d\mu(  s) \leq C\abs{\Delta}^{w-v}.\ee
We split the proof into two cases, choosing $w$ in each case so that 
$\min(w+u,w+\a)>\frac{d-1}{2}$.
\begin{itemize}
    \item Case $v-\a>\frac{d-1}{2}.$ Take $w:=\a$.
    $$\text{LHS}\eqref{eq:inequality for induction}\leq \int_{\pd}\abs{s}^{\a-v}e^{-\tr[\Delta s^{-1}]}\d\mu(  s)= \Gamma_d\lb v-\a\rb \abs{\Delta}^{\a-v}.$$
    
    \item Case $v-\a\leq\frac{d-1}{2}.$ Fix any $0<\ep<\min(u+v,v+\a)-\frac{d-1}{2}$, take $w:=v-(\frac{d-1}{2}+\ep)$.
    $$\text{LHS}\eqref{eq:inequality for induction}\leq \int_{\abs{s}\leq 1} \abs{s}^{\a-v}e^{-\tr[\Delta s^{-1}]}\d\mu(  s) +
    \int_{\abs{s}\geq 1} \abs{s}^{\a-v}e^{-\tr[s]}e^{-\tr[\Delta s^{-1}]}\d\mu(  s).$$
    
    For the integral over $\{\abs{s}\leq 1\}$, we have
    \begin{equation*}
        \begin{split}
            \int_{\abs{s}\leq 1} \abs{s}^{\a-v}e^{-\tr[\Delta s^{-1}]}\d\mu(  s)&=
    \int_{\abs{s}\geq 1} \abs{s}^{v-\a}e^{-\tr[\Delta s]}\d\mu(  s)< 
    \int_{\abs{s}\geq 1} \abs{s}^{\frac{d-1}{2}+\ep}e^{-\tr[\Delta s]}\d\mu(  s)\\
    &< \int_{\pd} \abs{s}^{\frac{d-1}{2}+\ep}e^{-\tr[\Delta s]}\d\mu(  s)
    =\Gamma_d\lb\frac{d-1}{2}+\ep\rb\abs{\Delta}^{-(\frac{d-1}{2}+\ep)},
        \end{split}
    \end{equation*}
    where we use the assumption $v-\a\leq \frac{d-1}{2}$ in the first inequality. For the integral over $\{\abs{s}\geq 1\}$, we have
    $$\int_{\abs{s}\geq 1} \abs{s}^{\a-v}e^{-\tr[s]}e^{-\tr[\Delta s^{-1}]}\d\mu(  s)\leq
    \int_{\abs{s}\geq 1} C\abs{s}^{w-v}e^{-\tr[\Delta s^{-1}]}\d\mu(  s)< C\Gamma_d\lb v-w \rb\abs{\Delta}^{w-v},
    $$
    where we use the following auxiliary Lemma \ref{lem:elementary inequality} in the first inequality by taking $\b=\a-w=\frac{d-1}{2}+\ep-(u-\a)>0$.  
\end{itemize}
\end{proof}

\begin{lemma}\label{lem:elementary inequality}
    For all $d\in\NN_+, \b>0$, there exists constant $C=C(d,\b)>0$ such that for all $x\in\pd$,
    $e^{\tr[x]}\geq C\abs{x}^{\b}$.
\end{lemma}
\begin{proof}
    The $d=1$ case is elementary: since $e^x>1$  for all $x>0$, we only need to consider the inequality for $x\geq 1$. Indeed, $e^x=\sum_{k=0}^{\infty}x^{k}$, it suffices to take $C=\frac{1}{n!}$ for $n=\lfloor\b\rfloor+1$.

    For general $d\geq 1$ and for any $x\in\pd$, let $\la_1,\dots,\la_d>0$ denote its $d$ eigenvalues. Using Arithmetic-Geometric mean inequality and the result for $d=1$, 
    $$e^{\tr[x]}\geq e^{d(\prod_{i=1}^d \la_i)^{1/d}}\geq e^{\abs{x}^{1/d}}\geq C(1,d\b)\abs{x}^{\b}.$$
\end{proof}

\begin{proposition}[Finiteness of normalization constant]\label{prop:Finiteness of Partition function} 
Let $\bta=(\ta_1, \cdots, \ta_N), u, v$ be parameters such that for all $1\leq k\leq N$,
\be\label{eq:parameters constraints for finite normalization constant}
u+v, u+\ta_i, v+\ta_i>\frac{d-1}{2}.
\ee
Consider the strip $\strip$ whose edges are labeled according to the rule \eqref{eq:labels on strip edges maximal current}. 
For any down-right path $\hp$ and any fixed $\la_1^0\in\pd$, the normalization constant $\mathcal{Z}^{\bta,u,v}$ defined by 
  \be\label{eq: partition function}
\mathcal{Z}^{\bta,u,v}=\int_{\pd^{2N+1}} \wt^{\mathcal{GP}}\lb \bl\rb \prod_{\substack{i=1,2, \;\; j=0,\dots, N \\ (i,j)\neq (1,0)}}\d\mu( \lambda_i^{j}),
   \ee is finite, and does not depend on $\hp$ nor $\la^0_1$.
\end{proposition}
\begin{proof}
Due to the Cauchy/Littlewood type identities (Lemma \ref{lem:Cauchy identity} and Lemma \ref{lem:Littlewood identity}), the normalization constant  is preserved under one-step path updates. Hence the value of $\mathcal{Z}^{\bta,u,v}$ does not depend on the down-right path $\hp$.
Due to the translation invariance of $\wt^{\gp}$ \eqref{eq:translation invariance} and the $\GL_d$-invariance of $\mu$ \eqref{eq:GL_d invariance of mu}, the integral \eqref{eq:value of Z for N=1} does not depend on the value of $\la^0_1$:  
\begin{multline*}
    \mathcal{Z}^{\bta,u,v} =\int_{\pd^{2N+1}} \wt^{\mathcal{GP}}\lb  \bl\star (\la_1^0)^{-1}\rb \prod_{\substack{i=1,2, \;\; j=0,\dots, N \\ (i,j)\neq (1,0)}}\d\mu\lb\lambda_i^{j}\star (\la_1^0)^{-1}\rb \\
    =\int_{\pd^{2N+1}} \wt^{\mathcal{GP}}\lb \bl\star (\la_1^0)^{-1}\rb \prod_{\substack{i=1,2, \;\; j=0,\dots, N \\ (i,j)\neq (1,0)}}\d\mu( \lambda_i^{j}).
\end{multline*}
 It remains to prove its finiteness for a horizontal path, we prove by induction on its length $N$. When $N=1$, a direct computation shows that: for any fixed $\la^0_1\in\pd$,
    \be\label{eq:value of Z for N=1}\int_{\pd^3}\wt\lb\raisebox{-40pt}{\begin{tikzpicture}
    \draw[dashed] (0,0.6)--(0.6,0);
    \draw[thick] (0,-0.6)--(0.6,0);
    \draw[thick] (0,0.6)--(0.6,1.2);
    \node[below  ] at (0.8,0) {\small$\la_2^{1}$};
    \node[above  ] at (0.8,1.2) {\small $\la_1^{1}$};
    \node[below] at (0,-0.6) {\small $\la_2^{0}$};
    \node[above] at (0,0.6) {\small $\la_1^{0}$};
    \node[below] at (0.35,-0.3) {\textcolor{red}{\small $\ta_1$}};
    \node[above] at (0.35,0.9) {\textcolor{red}{\small $\ta_1$}};
    \node[left] at (-0.6,0) {\textcolor{red}{\small $u$}};
    \node[right] at (1.2,0.8) {\textcolor{red}{\small $v$}};
    \draw[thick] (0,0.6) arc (90:270:0.6);
    \draw[thick] (0.6,0) arc (-90:90:0.6);
\end{tikzpicture}}\rb\d\mu( \la_1^{1}) \d\mu( \la_2^{0}) \d\mu( \la_2^{1})=\Gamma_d(u+v)\Gamma_d(\ta_1+v)\Gamma_d(\ta_1+u)<\infty.
\ee
Let us prove the induction step. By Lemma \ref{lem:induction inequality}, for all $u,\a$ with $\min \lb \a+u,\a+v, u+v \rb >\frac{d-1}{2}$, we can find $C>0$ and $w\in\RR$ with $\min \lb w+u, \a+w \rb >\frac{d-1}{2}$, such that 
\be\label{eq:induction inequality 2}\forall \la_1^{N-1}, \la_2^{N-1}\in\pd, \qquad
 \int_{\pd^2}\wt\lb\raisebox{-40pt}{\begin{tikzpicture}
    \draw[dashed](0,0.6)--(0.6,0);
    \draw[thick] (0,-0.6)--(0.6,0);
    \draw[thick] (0,0.6)--(0.6,1.2);
    \node[below right] at (0.5,-0) {\small $\la_2^{N}$};
    \node[above right] at (0.5,1.1) {\small $\la_1^{N}$};
    \node[below] at (-0.1,-0.6) {\small $\la_2^{N-1}$};
    \node[above] at (-0.1,0.6) {\small $\la_1^{N-1}$};
    \node[above] at (0.3,-0.3) {\textcolor{red}{\small $\alpha$}};
    \node[below] at (0.3,0.95) {\textcolor{red}{\small $\alpha$}};
    \node[right] at (1.2,0.6) {\textcolor{red}{\small $v$}};
    \draw[thick] (0.6,0) arc (-90:90:0.6);
\end{tikzpicture}}\rb\d\mu( \la_1^{N})\d\mu( \la_2^{N})
\leq C\wt\lb\raisebox{-33pt}{\begin{tikzpicture}
    \node[below] at (0.2,-0.6) {\small $\la_2^{N-1}$};
    \node[above] at (0.2,0.6) {\small $\la_1^{N-1}$};
    \node[right] at (0.6,0){\textcolor{red}{\small $w$}};
    \draw[thick] (0,-0.6) arc (-90:90:0.6);
\end{tikzpicture}}\rb.
\ee
Thanks to Lemma \ref{lem:symmetry of skew whittaker function}, we may suppose without loss of generality that $\ta_N=\min_{1\leq i\leq N}(\ta_i)$. Take $\a=\ta_N$ in \eqref{eq:induction inequality 2}, the normalization constant for path of length $N$ is bounded by constant times that for path of length $N-1$, which is finite by hypothesis of recurrence. Remark that the parameters taken above is such that for all $1\leq i\leq N-1$, $w+\ta_i\geq w+\ta_N>\frac{d-1}{2}$.
\end{proof}

\begin{remark}
    Note that the condition $\ta_i+\ta_j>\frac{d-1}{2}$ for the parameters in the maximal current regime \eqref{eq:parameters inhomogeneous general regime},  which ensures that the random variables  $W(n,m)$ are well-defined, is not required to guarantee the finiteness of the normalization constant $\mathcal{Z}^{\bta,u,v}$. 
\end{remark}

\subsection{Push-block dynamics and local Markov kernels}
\label{subsec:Push-block dynamics and local Markov kernels}
Recall the two-layer Gibbs measure  $\wt^{\gp}(\d\bl)$ defined in Definition \ref{def:two layer Gibbs measure}. The goal of this section is to construct a family of consistent Markov kernels $\{\U^{\gp,\gq}\}$, in the sense that for all $\bl'\in \pd^{2N+2}$ and all down-right paths $\hp,\hq$ such that  $\hq$ is above $\hp$, we have 
\be\label{eq:compatibility of local operator with wt IW} \int_{\pd^{2N+2}}\U^{\gp,\mathcal{G}\hq}\lb\bl'|\bl\rb\wt^{\gp}\lb\bl\rb\d\mu(\bl)
    =\wt^\gq\lb\bl'\rb.
\ee
A priori, multiple families of consistent kernels may exist. We choose one in the following, corresponding to what we refer to as the push-block dynamics.

We first define the one-step dynamics, i.e. when $\hq=\Tilde{\hp}$ differs from $\hp$ by a single vertex $\p_j$. We require the dynamics to be local, in the sense that the transition kernel $\U^{\gp,\mathcal{G}\Tilde{\hp}}$ depends only on the neighboring vertices and adjacent edge labels, that is, $\p_{j-1}, \p_j,\p_{j+1}$ and $\ga_{j-1},\ga_{j+1}$ for a bulk vertex ($1\leq j\leq N-1$). That is, we consider one-step kernels of the form 
$$\U^{\gp,\mathcal{G}\widetilde{\hp}}\lb\widetilde{\bl}|\bl\rb = \lb
\prod_{i\neq j} \delta_{\widetilde{\lambda}^{i}=\lambda^{i}} \rb\,\cdot\,\Ubw(\widetilde{\lambda}^{j}|\lambda^{j-1},\lambda^{j},\lambda^{j+1};\ga_j,\ga_{j+1}),
$$
where the superscript of the kernel reflects the local updates of the path shape.
Furthermore, we require the local kernel  $\Ubw(\cdot|\lambda^{j-1},\lambda^{j},\lambda^{j+1};\ga_{j-1},\ga_{j+1})$ be idempotent \cite{Blackwell}, that is, does not depend on $\la^j$. The kernel $\Ubw\lb\cdot|\la^{j-1},\la^{j+1};\ga_{j-1},\ga_{j+1}\rb$ is hence uniquely determined. Local Markov kernels $\Ulw,\Urw$, corresponding to updates at the left and right boundaries, are defined in a similar way. We formalize the above construction in the following definition.

\begin{definition}[Local Markov kernels]\label{def:local Markov kernels}
Let $\a,\b, u, v\in\RR $ be parameters such that $\a+\b, \a+u, \a+v>\frac{d-1}{2}$. For $\la,\mu,\pi\in\pd^2$, we define the local Markov kernels to be:
    \begin{subequations}
\be\label{eq:bulk local operator IW}
\Ubw(\pi|\la,\mu;\alpha,\beta):=\wt\lb\raisebox{-35pt}{\begin{tikzpicture}
    \draw[thick] (-0.6,0.6)--(0,1.2)--(0.6,0.6);
    \draw[dashed] (-0.6,0.6)--(0,0)--(0.6,0.6);
    \draw[thick] (-0.6,-0.6)--(0,0)--(0.6,-0.6);
    \node[below right] at (0.5,0.65) {\small $\mu_1$};
    \node[below right] at (0.5,-0.55) {\small   $\mu_2$};
    \node[below left] at (-0.45,0.66) {\small  $\la_1$};
    \node[below left] at (-0.45,-0.55) {\small  $\la_2$};
    \node[above] at (0,0) {\small  $\pi_2$};
    \node[above] at (0,1.2) {\small  $\pi_1$};
    \node[above] at (0.35,-0.35){\textcolor{red}{\small  $\alpha$}};
    \node[above] at (0.35,0.85){\textcolor{red}{\small  $\alpha$}};
    \node[above] at (-0.35,-0.4){\textcolor{red}{\small  $\beta$}};
    \node[above] at (-0.35,0.8){\textcolor{red}{\small  $\beta$}};
\end{tikzpicture}}\rb\bigg/
\int_{\pd^2}\wt\lb\raisebox{-35pt}{\begin{tikzpicture}
        \draw[thick] (-0.6,0.6)--(0,0)--(0.6,0.6);
        \draw[thick] (-0.6,-0.6)--(0,-1.2)--(0.6,-0.6);
        \draw[dashed] (-0.6,-0.6)--(0,0)--(0.6,-0.6);
        \node[above right] at (0.5,0.55) {\small $\mu_1$};
        \node[above right] at (0.5,-0.65) {\small  $\mu_2$};
        \node[above left] at (-0.5,0.55) {\small   $\la_1$};
        \node[above left] at (-0.5,-0.65) {\small  $\la_2$};
        \node[above] at (0.3,0.25) {\textcolor{red}{\small   $\beta$}};
        \node[above] at (0.3,-0.95) {\textcolor{red}{\small  $\beta$}};
        \node[above] at (-0.3,0.3) {\textcolor{red}{\small  $\alpha$}};
        \node[above] at (-0.3,-0.9) {\textcolor{red}{\small  $\alpha$}};
        \node[below] at (0,0) {\small  $\k_1$};
        \node[below] at (0,-1.2) {\small  $\k_2$};
\end{tikzpicture}}\rb\d\mu( \k),
\ee

\be\label{eq:local operator IW left boundary}
\Ulw(\pi|\la;u,\alpha):=\wt\lb\raisebox{-35pt}{\begin{tikzpicture}
    \draw[dashed](0.6,0)--(0,-0.6);
    \draw[thick] (0,0.6)--(0.6,0);
    \draw[thick] (0,-0.6)--(0.6,-1.2);
    \node[below right] at (0.5,-1.1) {\small $\la_2$};
    \node[below right] at (0.5,.1) {\small $\la_1$};
    \node[below] at (0,-0.6) {\small $\pi_2$};
    \node[above] at (0,0.6) {\small $\pi_1$};
    \node[above] at (0.35,0.3) {\textcolor{red}{\small $\alpha$}};
    \node[above] at (0.35,-0.9) {\textcolor{red}{\small $\alpha$}};
    \node[left] at (-0.6,0) {\textcolor{red}{\small $u$}};
    \draw[thick] (0,0.6) arc (90:270:0.6);
\end{tikzpicture}}\rb 
\bigg/
\int_{\pd^2}\wt\lb\raisebox{-35pt}{\begin{tikzpicture}
    \draw[dashed](0,0.6)--(0.6,0);
    \draw[thick] (0,-0.6)--(0.6,0);
    \draw[thick] (0,0.6)--(0.6,1.2);
    \node[above right] at (0.5,-.1) {\small $\la_2$};
    \node[above right] at (0.5,1.1) {\small $\la_1$};
    \node[below] at (0,-0.6) {\small $\k_2$};
    \node[above] at (0,0.6) {\small $\k_1$};
    \node[above] at (0.3,-0.3) {\textcolor{red}{\small $\alpha$}};
    \node[above] at (0.3,0.9) {\textcolor{red}{\small $\alpha$}};
    \node[left] at (-0.6,0) {\textcolor{red}{\small $u$}};
    \draw[thick] (0,0.6) arc (90:270:0.6);
\end{tikzpicture}}\rb\d\mu( \k),
\ee

\be\label{eq:local operator IW right boundary}
\Urw(\pi|\la;\alpha,v):=\wt\lb\raisebox{-35pt}{\begin{tikzpicture}
    \draw[dashed](0.6,0)--(0,0.6);
    \draw[thick] (0,0.6)--(0.6,1.2);
    \draw[thick] (0,-0.6)--(0.6,0);
    \node[ left] at (0.1,-0.6) {\small  $\la_2$};
    \node[ left] at (0.1,0.6) {\small $\la_1$};
    \node[below  ] at (0.7,0) {\small $\pi_2$};
    \node[above  ] at (0.7,1.2) {\small  $\pi_1$};
    \node  at (0.25,-0.1) {\textcolor{red}{\small $\alpha$}};
    \node  at (0.25,1.1) {\textcolor{red}{\small $\alpha$}};
    \node[right] at (1.2,0.6) {\textcolor{red}{\small $v$}};
    \draw[thick] (0.6,0) arc (-90:90:0.6);
\end{tikzpicture}}\rb
\bigg/
\int_{\pd^2}\wt\lb\raisebox{-35pt}{\begin{tikzpicture}
    \draw[dashed] (0,0)--(0.6,0.6);
    \draw[thick] (0.6,0.6)--(0,1.2);
    \draw[thick] (0,0)--(0.6,-0.6);
    \node[above left] at (0.1,-0.1) {\small  $\la_2$};
    \node[above left] at (0.1,1) {\small  $\la_1$};
    \node[below ] at (0.7,-0.6) {\small $\k_2$};
    \node[above ] at (0.7,0.6) {\small  $\k_1$};
    \node  at (0.35,-0.1) {\textcolor{red}{\small $\alpha$}};
    \node  at (0.35,1.1) {\textcolor{red}{\small $\alpha$}};
    \node[right] at (1.2,0) {\textcolor{red}{\small  $v$}};
    \draw[thick] (0.6,-0.6) arc (-90:90:0.6);
\end{tikzpicture}}\rb\d\mu( \k).
\ee
\end{subequations} 
These are well-defined Markov kernels due to the finiteness of the integral appearing in the denominator, the Cauchy type identity (Lemma \ref{lem:Cauchy identity}) and the Littlewood type identity (Lemma \ref{lem:Littlewood identity}).
\end{definition}

\begin{definition}[Two-layer inverse-Wishart Markov dynamics] \label{def:two layer IW dynamics}
The two-layer inverse-Wishart Markov dynamics refers to a family of Markov kernels $\{ \U^{\gp,\mathcal{G}\widetilde{\hp}} \}$ on the state space $\pd^{2N+2}$, where $\hp$ and $\widetilde{\hp}$ are down-right paths that differ by a single vertex.
When the difference is a bulk vertex $\p_j$, the kernel is defined as:
\be\label{eq:two layer transition kernel}\U^{\gp,\mathcal{G}\widetilde{\hp}}\lb\widetilde{\bl}|\bl\rb = \lb
\prod_{i\neq j} \delta_{\widetilde{\lambda}^{i}=\lambda^{i}} \rb\,\cdot\,\Ubw(\widetilde{\lambda}^{j}|\lambda^{j-1},\lambda^{j+1};\gamma_j,\gamma_{j+1}).
\ee
Recall that  $\gamma_j,\gamma_{j+1}$ are labels on the edges $(\p_{j-1},\p_j), (\p_j,\p_{j+1})$ respectively. When the difference is a  left/right boundary vertex, the kernel is defined respectively as:
\begin{align*}
\U^{\gp,\mathcal{G}\widetilde{\hp}}\lb\widetilde{\bl}|\bl\rb &= \lb
\prod_{i\neq 0} \delta_{\widetilde{\lambda}^{i}=\lambda^{i}} \rb\,\cdot\,\Ulw(\widetilde{\lambda}^{0}|\lambda^{1};u,\gamma_1),\\   
\U^{\gp,\mathcal{G}\widetilde{\hp}}\lb\widetilde{\bl}|\bl\rb &= \lb
\prod_{i\neq N} \delta_{\widetilde{\lambda}^{i}=\lambda^{i}} \rb\,\cdot\,\Urw(\widetilde{\lambda}^{N}|\lambda^{N-1};\gamma_N,v).
\end{align*}
 
For any pair of down-right paths $\big(\hp, \hq\big)$, the Markov kernels  $\U^{\gp,\mathcal{G}\hq}\big(\widetilde{\bl}|\bl\big)$ are defined by composing the one-step kernels defined above. Note that the definition does not depend on the choice of intermediate paths. Suppose $\hp_1\rightarrow\hp_2\rightarrow\hp_4$ and $\hp_1\rightarrow\hp_3\rightarrow\hp_4$ are two sequences of one-step updates from $\hp_1$ to $\hp_4$,  then $$\U^{\gp_1,\gp_4}=\U^{\gp_2,\gp_4}\U^{\gp_1,\gp_2}=\U^{\gp_3,\gp_4} \U^{\gp_1,\gp_3},$$
due to the delta function structure in the definition of one-step kernels. 
Hence the two-layer inverse-Wishart Markov dynamics $\U^{\gp,\mathcal{G}\hq}\lb\widetilde{\bl}|\bl\rb$ is  well-defined and consistent with the two-layer Gibbs measure $\wt^{\gp}\lb\d\bl\rb$ in the sense given by \eqref{eq:compatibility of local operator with wt IW}.
\end{definition}

\begin{remark} 
    The name `push-block' in this section's title comes from the dynamics defined on Gelfand-Tsetlin patterns, where the $\frac{N(N-1)}{2}$ particles $X(t)=\lb x_{i,j}(t)\rb_{1\leq j\leq i\leq N}$ experience two sorts of interactions: they are `pushed' by the particle behind and `blocked' by the particle in front \cite{Warren,WarrenWindridge,BorodinCorwin}. When the particles are positive definite matrices-valued, there exists a Markovian dynamics on the triangular arrays whose boundary marginal has interesting autonomous Markovian dynamics, see \cite[Remark 3.4]{AristaBisiOConnell} and references therein. 
\end{remark}

\subsection{First-layer marginal}
\label{subsec:First-layer marginal}
The connection between the polymer recurrence \eqref{eq:recurrence inhomogeneous strip} and the constructed push-block dynamics (Definition \ref{def:two layer IW dynamics}) surfaces when looking at the first layer marginal of the two-layer Gibbs measure. More precisely, we show that under the dynamics $\U^{\gp,\mathcal{G}\hq}\lb{\bl'}|\bl\rb$, the first layer marginal $\bl_1\in\pd^N$ has an autonomous Markov dynamics $\rmU^{\hp,\hq}\lb\bl_1'|\bl_1\rb$ which coincides with the inverse-Wishart polymer recurrence in the sense precised in Lemma \ref{lem: IW First layer operator}. 

\begin{lemma}[Marginal dynamics]\label{lem: IW First layer operator}
We have the following:
\begin{enumerate}
    \item [(1)] Under $\Ubw\lb\pi|\la,\mu;\a,\b\rb$, the distribution of $\pi_1$ depends on $\la,\mu$ only through $\la_1,\mu_1$. The transition kernel is given explicitly by
    \be\label{eq:first layer dynamics operator bulk}\rmUbw(\pi_1|\la_1,\mu_1;\alpha,\beta)=Z(\la_1,\mu_1;\alpha,\beta)^{-1}\abs{\pi_1^{-1}}^{\a+\b}e^{-\tr[\pi_1^{-1}(\la_1+\mu_1)]},\ee
     for some finite normalization constant $Z(\la_1,\mu_1;\a,\b)>0$. Equivalently,
     for an independent random variable $W\sim\Wisv(\a+\b)$,
$$\pi_1\stackrel{d}{=}W\star(\la_1+\mu_1).$$
    
    \item[(2)] Under $\Ulw\lb\pi|\la;u,\a\rb$, or  $\Urw\lb\pi|\la;\a,v\rb$ respectively, the law of $\pi_1$ depends on $\la$ only through $\la_1$. The transition kernel writes respectively as:
    $$\rmUlw(\pi_1|\la_1;u,\alpha)=\frac{\abs{\pi_1^{-1}}^{\a+u}e^{-\tr[\la_1\pi_1^{-1}]}}{Z(\la_1;u,\alpha)},\qquad 
\rmUrw(\pi_1|\la_1;\alpha,v)=\frac{\abs{\pi_1^{-1}}^{\a+v}e^{-\tr[\la_1\pi_1^{-1}]}}{Z(\la_1;v,\alpha)},$$
for some finite normalization constants $Z(\la_1; u, \alpha) > 0$ and $Z(\la_1; v, \alpha) > 0$, respectively. Equivalently, for an independent random variable $W\sim\Wisv(\a+u)$ (or ~$\Wisv(\a+v)$ resp.),
$$\pi_1\stackrel{d}{=}W\star\la_1. $$
\end{enumerate} 
 As a consequence, for any pair of down-right path $(\hp,\hq)$ such that $\hq$ is above $\hp$, under the push-block dynamics 
    $\U^{\gp,\gq}\lb\bl'|\bl\rb$, the first-layer configuration $\bl_1$ has an autonomous dynamics which we denote by $\mathrm{U}^{\hp,\hq}\lb\bl'_1|\bl_1\rb$. This dynamics coincides with the inverse-Wishart polymer recurrence, in the sense that if we fix an initial configuration $\bl_1\in\pd^{N+1}$ along path $\hp$, then under polymer recurrence  \eqref{eq:recurrence inhomogeneous strip}, the probability density of configuration $\bl'_1$ along path $\hq$ is $\mathrm{U}^{\hp,\hq}\lb\bl'_1|\bl_1\rb$.
\end{lemma}
\begin{proof}
    We prove the claim for bulk kernel $\Ubw$ and left boundary kernel $\Ulw$, the argument for the right boundary kernel $\Urw$ is analogous.
Because of the finiteness of denominator \eqref{eq:bulk local operator IW}, the bulk local kernel \eqref{eq:bulk local operator IW} decomposes as:
\begin{multline*}
    \Ubw\lb\pi|\la,\mu;\alpha,\beta\rb 
    =
    \frac{\abs{\pi_1^{-1}}^{\a+\b}
    e^{-\tr[\pi_1^{-1}(\la_1+\mu_1)]}}{ \int_{\pd}\d\mu( \pi_1) 
    \abs{\pi_1^{-1}}^{\a+\b}
    e^{-\tr[\pi_1^{-1}(\la_1+\mu_1)]}}
\cdot \\
    \frac{\abs{\pi_2^{-1}}^{\a+\b}
    e^{-\tr[\pi_2(\mu_1^{-1}+\la_1^{-1})+\pi_2^{-1}(\la_2+\mu_2)]}
     }{\int_{\pd}\d\mu( \pi_2)
    \abs{\pi_2^{-1}}^{\a+\b}
    e^{-\tr[\pi_2(\mu_1^{-1}+\la_1^{-1})+\pi_2^{-1}(\la_2+\mu_2)]}}.
\end{multline*}
Integrating over $\pi_2$ yields the marginal kernel
\begin{equation*}
    \int_{\pd}\Ubw\lb\pi|\la,\mu;\alpha,\beta\rb \d\mu( \pi_2)=\frac{\abs{\pi_1^{-1}}^{\a+\b}
    e^{-\tr[\pi_1^{-1}(\la_1+\mu_1)]}}{Z(\la_1,\mu_1;\alpha,\beta)}=:\rmUbw(\pi_1|\la_1,\mu_1;\alpha,\beta).
\end{equation*}
    
Because of the finiteness of the denominator (Lemma \ref{lem:Littlewood identity}), the left boundary local kernel \eqref{eq:local operator IW left boundary} 
decomposes as:
$$\Ulw\lb\pi|\k;u,\alpha\rb=
\frac{\abs{\pi_1^{-1}}^{\a+u} e^{-\tr[\k_1\pi_1^{-1}]} }{\int_{\pd}\d\mu( \pi_1)
\abs{\pi_1^{-1}}^{\a+u} e^{-\tr[\k_1\pi_1^{-1}]} }
\frac{\abs{\pi_2^{-1}}^{\a-u} e^{-\tr[\pi_2\k_1^{-1}+\k_2\pi_2^{-1}]} }{\int_{\pd} \d\mu( \pi_2)\abs{\pi_2^{-1}}^{\a-u} e^{-\tr[\pi_2\k_1^{-1}+\k_2\pi_2^{-1}]}}.$$
Integrating over $\pi_2$ yields the marginal kernel
$$\int_{\pd}\Ulw\lb\pi|\k;u,\alpha\rb\d\mu( \pi_2)=Z(\k_1;u,\alpha)^{-1}\abs{\pi_1^{-1}}^{\a+u} e^{-\tr[\k_1\pi_1^{-1}]}=:\rmUlw(\pi_1|\k_1;u,\alpha).$$
Because the kernel $\U^{\gp,\Tilde{\gp}}$ is local \eqref{eq:two layer transition kernel}, so is the kernel of the first-layer dynamics: the transition kernel $\mathrm{U}^{\hp,\Tilde{\hp}}$ factors as products of local kernels $\Ubw, \Ulw, \Urw$ and delta functions. For general down-right paths $\big(\hp, \hq\big)$ with $\hq$ above $\hp$, the Markov  kernel $\mathrm{U}^{\hp, \hq}$ is well-defined by composing the Markov kernels defined above and does not depend on the choice of intermediate paths.
\end{proof}

\subsection{Proof of Theorem \ref{thm: inhomogeneous maximal current strip}}
In Section \ref{subsec:Push-block dynamics and local Markov kernels}, we have constructed the push-block dynamics $\U^{\gp,\gq}\lb\bl'|\bl\rb$  under which the two-layer Gibbs measure $\wt^{\gp}\lb\d\bl\rb$ is consistent \eqref{eq:compatibility of local operator with wt IW}. Recall the decomposition $\bl=(\bl_1, \bl_2)\in\left(\pd^{N+1}\right)^2$, we define the first-layer marginal density by
\be\label{eq:first layer marginal}
\widehat{\wt}^{\hp}(\bl_1):=\int_{\pd^{N+1}}\wt^{\gp}\lb\bl\rb\d\mu(\bl_2).
\ee
By Proposition \ref{prop:Finiteness of Partition function}, for any fixed $\la^0_1\in\pd$, 
$$\mathcal{Z}^{\bta,u,v}=\int_{\pd^{N}} \widehat{\wt}^{\mathcal{P}}\lb \bl_1\rb \prod_{k=1}^N\d\mu( \lambda_1^k)<\infty.$$
Hence 
$\wt^{\hp}(\d\bl_1)=\wt^{\hp}(\bl_1)\d\mu(\bl_1)$ 
is a well-defined $\sigma$-finite measure on $\pd^{N+1}$, which is a requirement for applying Tonelli's theorem below. We are now ready to prove the main theorem of this section.

\begin{proof}[Proof of Theorem \ref{thm: inhomogeneous maximal current strip}] \label{subsec:Proof of inhomogeneous strip max current}
As remarked in \eqref{eq:maximal current stationary measure is wt gp}, up to the normalization constant $\mathcal{Z}^{\bta,u,v}$, the two-layer Whittaker process with $\mu$-distributed starting matrix equals the two-layer Gibbs measure $\wt^{\gp}(\d\bl)$. As a consequence, their first-layer marginal coincides: $$\frac{1}{\mathcal{Z}^{\bta,u,v}}\widehat{\wt}^{\hp}(\bl_1)\d\mu(\bl_1)=
\d\mu(\la^0_1)\times
\d\bP^{\bg(\hp),\bw(\hp)}_{\la^0_1}\lb (\la_1^i)_{1\leq i \leq N}\rb.$$
On the other hand, by Lemma \ref{lem: IW First layer operator}, the first-layer dynamics defined by Markov kernels $\{\rmU^{\hp,\Tilde{\hp}}\}$ coincides with the polymer recurrence. Therefore, the consistency in Theorem \ref{thm: inhomogeneous maximal current strip} is equivalent to the consistency of the first layer marginal of the two-layer Gibbs measure with the first layer dynamics. In other words, it suffice to prove that: for any down-right paths $\hp,\hq$ with $\hq$ above $\hp$ and all $\bl_1'\in\pd^{N+1}$,  
$$\int_{\pd^{N+1}}\mathrm{U}^{\hp,\hq}\lb\bl'_1|\bl_1\rb\widehat{\wt}^{\hp}(\bl_1)\d\mu(\bl_1)=\widehat{\wt}^\hq\lb\bl'_1\rb.$$
Indeed, 
\begin{equation*}
    \begin{split}
        \widehat{\wt}^\hq\lb\bl'_1\rb&= 
    \int_{\pd^{N+1}}\d\mu(\bl'_2)\wt^{\gq}(\bl')\\
    &=  \int_{\pd^{N+1}}\d\mu(\bl'_2) \int_{\pd^{2N+2}}\d\mu(\bl)
    \wt^{\gp}(\bl) \U^{\gp,\gq}\lb\bl'|\bl\rb\\
    &=  \int_{\pd^{N+1}}\d\mu(\bl_1)\int_{\pd^{N+1}}\d\mu(\bl_2)
    \wt^{\gp}(\bl) \int_{\pd^{N+1}}\d\mu(\bl'_2)\U^{\gp,\gq}\lb\bl'|\bl\rb\\
&=\int_{\pd^{N+1}}\d\mu(\bl_1)\widehat{\wt}^\hp(\bl_1)\mathrm{U}^{\hp,\hq}\lb\bl'_1|\bl_1\rb,
    \end{split}
\end{equation*}
where we use the consistency relation \eqref{eq:compatibility of local operator with wt IW} in the second equality and Tonelli's Theorem in the third equality.
\end{proof}
In the special case where the model is homogeneous $(\bta=\ta^N)$, we recover Theorem \ref{thm: homogeneous strip maximal current}. 

\begin{proof}[Proof of Theorem \ref{thm: homogeneous strip maximal current}]
When $\bta=(\ta)^N$, consider the model on a strip with horizontal bottom boundary $\bw(\hp^{\bw}_{\circ})=\lb\rightarrow\rb^N$ and initial condition $\bl_1=(\la^j_1)_{0\leq j \leq N}$ given by  \eqref{eq:stationary measure for homogeneous strip model in the maximal current regime}, i.e.
$$ \d\mu( \la_1^0)\times\d\bP^{\ta,u,v}_{\la^0_1}\lb  (\la^j_1)_{1\leq j \leq N}\rb .$$ 
For all $k\geq 0$, we have  $\bg\lb\tau_k\hp^{\bw}_{\circ}\rb=(\ta)^N$ and $\bw\lb\tau_k\hp^{\bw}_{\circ}\rb=(\rightarrow)^N$, by Theorem \ref{thm: inhomogeneous maximal current strip},  along down-right paths  $\tau_k\hp^{\bw}_{\circ}$, the partition function process has the same distribution. In other words, 
the $\sigma$-finite measure \eqref{eq:stationary measure for homogeneous strip model in the maximal current regime}
is a stationary measure in the sense precised in Definition \ref{def:homogeneous strip model}.
\end{proof}

\subsection{Markovian description of the two-layer matrix Whittaker process}\label{subsec:Probabilistic description of two-layer matrix Whittaker process} 
In this section we give the two-layer matrix Whittaker process $\d\PP^{\bg,\bw}_{\la^0_1}\lb\bl\rb$ \eqref{eq:general two layer matrix whittaker process} a probabilistic interpretation, which is adapted from the scalar case studied in \cite{Barraquand}. Recall the skew Whittaker function $\Psi_{\ga_1,\dots,\ga_k}(\la/\mu)$ defined by \eqref{eq:branching rule, skew whittaker function}. For $1\leq x\leq N$ and $\la,\mu\in\pd^2$, define
\be\label{eq: initial condition of survived inhomogeneous random walk}
\P_0(\la)=\frac{1}{\mathcal{Z}^{\bg,u,v}}\abs{\la_2}^{u}\delta_{\la_1^0}(\la_1)H_{0,N}(\la_2\star \la_1^{-1}),\ee
and 
\begin{subequations}\label{eq:transition kernel of survived inhomogeneous random walk}
\begin{align}
\P_{x,N}^{\rightarrow}(\la,\mu)&=\frac{H_{x,N}(\mu_2\star\mu_1^{-1})}{H_{x-1,N}(\la_2\star\la_1^{-1})}
    \Psi_{\ga_{x}}(\mu/\la),
    \\
    \P_{x,N}^{\downarrow}(\la,\mu)&=\frac{H_{x,N}(\mu_2\star\mu_1^{-1})}{H_{x-1,N}(\la_2\star\la_1^{-1})}
    \Psi_{\ga_{x}}(\la/\mu),
\end{align}
\end{subequations}  
where the function $H_{x,N}:\pd\rightarrow\RR_{+}$ is defined by 
$$H_{x,N}(\mu)=\int_{\pd^2} \Psi_{\ga_{x+1},\dots,\ga_N}\lb 
 (\la_1^N, \la_2^N)/(\id,\mu)\rb \abs{\la_2^N(\la_1^N)^{-1}}^v \d\mu(\la_1^N)\d\mu(\la_2^N).$$

\begin{proposition}\label{prop: inhomogeneous markovian description}
    Under the probability measure $\d\PP^{\bg,\bw}_{\la^0_1}\lb\bl\rb$, the process $x\mapsto \la^x\in\pd^2$ is an inhomogeneous Markov chain with initial distribution $\P_0(\la^0)$ \eqref{eq: initial condition of survived inhomogeneous random walk} and transition kernels $\P^{w_x}_{x,N}(\la^{x-1},\la^{x})$ \eqref{eq:transition kernel of survived inhomogeneous random walk}.  
\end{proposition}
\begin{proof}
    By definition, it suffices to prove that: for all $1\leq x \leq N, \ \mu^x, \mu^{x-1},\dots , \mu^0\in\pd^2$, 
    \be\label{eq:check transition kernel}\PP[\la^x=\mu^x|\la^{x-1}=\mu^{x-1},\dots, \la^0=\mu^0]=\PP[\la^x=\mu^x|\la^{x-1}=\mu^{x-1}]=\P^{w_x}_{x,N}(\mu^{x-1},\mu^x).\ee
 
    Indeed, by definition of conditional expectation, the  L.H.S. of  \eqref{eq:check transition kernel} can be written as a ratio $A/B$ where
    \begin{equation*}
        \begin{split}
            A=\delta_{\mu_1^0=\id}
    \abs{\mu_2^0(\mu_1^0)^{-1}}^u&
    \prod\limits_{\substack{1\leq i\leq x\\\p_i-\p_{i-1}=\rightarrow}}\Psi_{\ga_i}(\mu^i/\mu^{i-1})\prod\limits_{\substack{1\leq i\leq x\\\p_i-\p_{i-1}=\downarrow}}    \Psi_{\ga_i}(\mu^{i-1}/\mu^{i})\\
    \int_{(\pd^2)^{N-x}} \prod_{i=x+1}^N\d\mu(  \mu^i)&\prod\limits_{\substack{x+1\leq i\leq N\\\p_i-\p_{i-1}=\rightarrow}}\Psi_{\ga_i}(\mu^i/\mu^{i-1})\prod\limits_{\substack{x+1\leq i\leq N\\\p_i-\p_{i-1}=\downarrow}}
    \Psi_{\ga_i}(\mu^{i-1}/\mu^{i})
    \abs{\mu_2^N(\mu_1^N)^{-1}}^v
    ,
        \end{split}
    \end{equation*}
    
    \begin{equation*}
        \begin{split}
            B=\delta_{\mu_1^0=\id}
    \abs{\mu_2^0(\mu_1^0)^{-1}}^u&
    \prod\limits_{\substack{1\leq i\leq x-1\\\p_i-\p_{i-1}=\rightarrow}}\Psi_{\ga_i}(\mu^i/\mu^{i-1})\prod\limits_{\substack{1\leq i\leq x-1\\\p_i-\p_{i-1}=\downarrow}}    \Psi_{\ga_i}(\mu^{i-1}/\mu^{i})\\
    \int_{(\pd^2)^{N-x+1}} \prod_{i=x}^N\d\mu(  \mu^i) &\prod\limits_{\substack{x\leq i\leq N\\\p_i-\p_{i-1}=\rightarrow}}\Psi_{\ga_i}(\mu^i/\mu^{i-1})\prod\limits_{\substack{x\leq i\leq N\\\p_i-\p_{i-1}=\downarrow}}
    \Psi_{\ga_i}(\mu^{i-1}/\mu^{i})
    \abs{\mu_2^N(\mu_1^N)^{-1}}^v   .
        \end{split}
    \end{equation*}
By the Cauchy/Littlewood type identity  (Lemma \ref{lem:Cauchy identity} and Lemma \ref{lem:Littlewood identity}), the two integrals do not depend on the shape of $\hp$ and thus equal to $H_{x,N}(\mu^x), H_{x-1,N}(\mu^{x-1})$  respectively, the quotient simplifies into 
$$\frac{
\lb\Psi_{\ga_x}(\mu^x/\mu^{x-1})\rb^{\mathds{1}_{\p_x-\p_{x-1}=\rightarrow}}
\lb\Psi_{\ga_x}(\mu^{x-1}/\mu^{x})\rb^{\mathds{1}_{\p_x-\p_{x-1}=\downarrow}}
H_{x,N}(\mu^x_2\star(\mu^x_1)^{-1})}
{H_{x-1,N}(\mu^{x-1}_2\star(\mu^{x-1}_1)^{-1})},
$$
which equals $\P^{w_x}_{x,N}(\mu^{x-1},\mu^x)$. 
\end{proof}

 The skew Whittaker function $\Psi_{\ga_x}(\la^x/\la^{x-1})$ (resp.\ $\Psi_{\ga_x}(\la^{x-1}/\la^{x})$) can be interpreted as the sub-Markovian transition kernel for two independent $\Wis^{\pm}$ random walks $\bl_1, \bl_2$, killed at time $x$ with probability $e^{-\tr[\la_2^{x}(\la_1^{x-1})^{-1}]}$ (resp.\ $e^{-\tr[\la_2^{x-1}(\la_1^{x})^{-1}]}$). The Doob-transformed kernels ~\eqref{eq:transition kernel of survived inhomogeneous random walk}, which are Markovian due to the branching rule~\eqref{eq:branching rule, skew whittaker function}, then correspond to the transition kernels for the walks conditioned to survive up to time~$N$.

\section{Stationary measure for equilibrium systems}\label{sec:Stationary measure for equilibrium systems}
 
In this section, we prove that the $\Wis^{\pm}$ random walk (Definition \ref{def:IW random walk}) with $\mu$-distributed  starting matrix is consistent with the inhomogeneous inverse-Wishart polymer on a strip in the equilibrium regime \eqref{eq:parameters inhomogeneous equilibrium regime}.  We also show that with fixed  starting matrix, the $\Wis^{\pm}$ random walk is locally consistent in a sense that we will make precise.  
 As an application, we derive stationarity results for the model  on the quadrant with boundary conditions by suitably embedding a portion of the quadrant into the strip.
 
\subsection{Equilibrium regime of the inhomogeneous model on a strip}
\label{subsec:Equilibrium regime of the strip model}
In this section, we consider the inhomogeneous inverse-Wishart polymer on a strip (Definition \ref{def:inhomogeneous strip model}) with parameters $\bta=(\ta_i)_{1\leq i\leq N}, u, v$ in the equilibrium regime \eqref{eq:parameters inhomogeneous equilibrium regime}, i.e., for all
$1\leq i \leq N$, 
$$u+v=0,\  \theta_i+u>\frac{d-1}{2}, \ \theta_i+v>\frac{d-1}{2}.$$ 
The proof of the stationarity follows a strategy analogous to that used in the maximal current regime. We first introduce the Gibbs measure associated with a down-right path on the strip, then construct a Markovian dynamics which is consistent with the polymer recurrence, and finally identify this Gibbs measure with the distribution of a simple stochastic process. The main technical difference is that instead of the two-layer Gibbs measure, we will employ a one-layer version here.  Let us begin by defining the corresponding one-layer graph and one-layer Gibbs measure.
 
\begin{definition}[One-layer graph and one-layer Gibbs measure]
 We label the edges of $\strip$ in a way that along any down-right path $\hp=(\p_i)_{0\leq i\leq N}$ on the lattice, the parameters $\bg(\hp)=(\gamma_1,\dots,\gamma_N)$  reads
    \be\label{eq:labels on strip edges equilibrium}\gamma_i= \left\{  \begin{aligned}
\a_n:=\ta_n+v \quad & \text{if } \p_{i-1}=(n-1,m), \p_{i}= (n,m) \\
\b_m:=\ta_m+u \quad & \text{if } \p_{i-1}=(n,m), \p_{i}=(n,m-1) 
\end{aligned}
\right., \quad 1\leq i\leq N.\ee
Note the difference between this labeling rule and the one  \eqref{eq:labels on strip edges maximal current} defined for the model with parameters in the maximal current regime. 

For a given down-right path $\hp$, we define the associated one-layer graph as the first layer of the corresponding two-layer graph $\gp$. With a slight abuse of notation, we also denote this one-layer graph by $\hp$. Formally, $\hp$ is the labeled undirected graph $\hp = (V, E, \bg)$ with: 
\begin{itemize}
    \item \textbf{Vertex set:} $ V = \{ \p_i : 0 \leq i \leq N \}$.
    \item \textbf{Edge set:}
    $E =  \;\{\{\p_{i-1}, \p_{i}\}  : 1 \leq i \leq N\}$.
    \item \textbf{Edge labels:}
    $\bg(\{\p_{i-1}, \p_{i}\}) = \gamma_i$,  where $\gamma_i$ is defined by the labeling rule in~\eqref{eq:labels on strip edges equilibrium} for $1\leq i\leq N$.
\end{itemize}

Recall the definition of Gibbs measure associated with a subgraph $\mathcal{G}\subset\gp$ (Definition \ref{def:two layer Gibbs measure}). In particular, when $\mathcal{G}=\hp$, and up to a finite normalization constant $
\mathcal{Z}^{\hp}$  that we will define in \eqref{eq:partition function equilibrium regime}, the one-layer Gibbs measure $\wt^{\hp}(\d\bl)$ coincides with the  $\Wis^{\pm}$ random walk with $ \mu$-distributed starting matrix:
\be\label{eq:equilibrium stationary measure is wt hp}\frac{1}{\mathcal{Z}^{\hp}}\wt^{\hp}(\bl)\d\mu( \bl)=\d\mu( \la^0)\times \d\bR^{\bg(\hp),\bw(\hp)}_{\la^0} \lb(\la^i)_{1\leq i\leq N}\rb.\ee
 We stress that the one-layer Gibbs measure $\wt^{\hp}$ differs from the first-layer marginal of the two-layer Gibbs measure $\widehat{\wt}^{\hp}$~\eqref{eq:first layer marginal}. 
  Although both are infinite measures on $\pd^{N+1}$, the former is a product measure, while the latter is not, due to interactions between the two layers. 
\end{definition}
We now establish the corresponding one-layer Cauchy/Littlewood type identities and define local dynamics, in parallel with Sections \ref{subsec:Properties of Gibbs measure} and \ref{subsec:Push-block dynamics and local Markov kernels}.
\begin{lemma}[One-layer Cauchy type identity]\label{lem:1-layer Cauchy}
For any $\a,\b\in\RR$ with $\a+\b>\frac{d-1}{2}$ and any $\la,\mu\in\pd$, we have
    \begin{equation}
    \label{eq:1-layer Cauchy}
         \int_{\pd}\wt\lb\raisebox{-20pt}{\begin{tikzpicture}
        \draw[thick] (-0.6,-0.6)--(0,-1.2)--(0.6,-0.6);
        \node[above right] at (0.5,-0.65) {\small  $\mu$};
        \node[above left] at (-0.5,-0.65) {\small  $\la$};
        \node[above] at (0.3,-0.9) {\textcolor{red}{\small  $\a$}};
        \node[above] at (-0.3,-0.95) {\textcolor{red}{\small  $\b$}};
        \node[below] at (0,-1.2) {\small  $\k$};
\end{tikzpicture}}\rb\d\mu( \k)
    =
     \int_{\pd}\wt\lb\raisebox{-20pt}{\begin{tikzpicture}
    \draw[thick] (-0.6,0.6)--(0,1.2)--(0.6,0.6);
    \node[below right] at (0.5,0.65) {\small $\mu$};
    \node[below left] at (-0.45,0.66) {\small  $\la$};
    \node[above] at (0,1.2) {\small  $\pi$};
    \node[above] at (0.35,0.8){\textcolor{red}{\small  $\b$}};
    \node[above] at (-0.35,0.85){\textcolor{red}{\small  $\a$}};
\end{tikzpicture}}\rb\d\mu( \pi).  
    \end{equation}    
\end{lemma}

\begin{proof}
Using integral formula \eqref{eq:laplace transform},  both sides equal to
$\frac{\abs{\la}^\b \abs{\mu}^\a}{\abs{\la+\mu}^{\a+\b}}\Gamma_d (\a+\b)
$.
\end{proof}

\begin{lemma}[One-layer Littlewood type identity]\label{lem:1-layer Littlewood}
For any $\a,\b\in \RR$ with $\a,\b>\frac{d-1}{2}$ and any $\la\in\pd$, we have
    \begin{equation}\label{eq:1-layer Littlewood}
    \Gamma_d(\a)\int_{\pd}
\wt\lb\raisebox{-20pt}{\begin{tikzpicture}
    \draw[thick] (0,1.2)--(0.6,0.6);
    \node[below right] at (0.5,0.65) {\small $\la$};
    \node[above] at (0,1.2) {\small  $\pi$};
    \node[above] at (0.35,0.85){\textcolor{red}{\small  $\b$}};
\end{tikzpicture}}\rb\d\mu( \pi)
     = \Gamma_d(\b)
\int_{\pd}\wt\lb\raisebox{-20pt}{\begin{tikzpicture}
        \draw[thick] (0,-1.2)--(0.6,-0.6);
        \node[above right] at (0.5,-0.65) {\small  $\la$};
        \node[above] at (0.3,-0.9) {\textcolor{red}{\small  $\a$}};
        \node[below] at (0,-1.2) {\small  $\k$};
\end{tikzpicture}}\rb\d\mu( \k).
    \end{equation}    
\end{lemma}
\begin{proof}
A direct computation shows both sides equal to 
$\Gamma_d(a)\Gamma_d (\b)
$.
\end{proof}

\begin{definition}[Local Markov kernels]\label{def: local Markov kernels equilibrium} For $\a, \b\in\RR$ with $\a,\b>\frac{d-1}{2}$, define 
\begin{subequations}\label{eq:local Markov kernels equilibrium}
    \be\label{eq:local operator eq bulk}
\rmUbw\lb \pi| \la,\mu; \a, \b \rb=
\wt\lb\raisebox{-20pt}{\begin{tikzpicture}
    \draw[thick] (-0.6,0.6)--(0,1.2)--(0.6,0.6);
    \node[below right] at (0.5,0.65) {\small $\mu$};
    \node[below left] at (-0.45,0.66) {\small  $\la$};
    \node[above] at (0,1.2) {\small  $\pi$};
    \node[above] at (0.35,0.85){\textcolor{red}{\small  $\a$}};
    \node[above] at (-0.35,0.8){\textcolor{red}{\small  $\b$}};
\end{tikzpicture}}\rb\bigg/
\int_{\pd}\wt\lb\raisebox{-20pt}{\begin{tikzpicture}
        \draw[thick] (-0.6,-0.6)--(0,-1.2)--(0.6,-0.6);
        \node[above right] at (0.5,-0.65) {\small  $\mu$};
        \node[above left] at (-0.5,-0.65) {\small  $\la$};
        \node[above] at (0.3,-0.95) {\textcolor{red}{\small  $\b$}};
        \node[above] at (-0.3,-0.9) {\textcolor{red}{\small  $\a$}};
        \node[below] at (0,-1.2) {\small  $\k$};
\end{tikzpicture}}\rb\d\mu( \k),
\ee

\be\label{eq:local operator eq left boundary}
\rmUlw\lb \pi| \la; \a, \b \rb=\Gamma_d(\a)
\wt\lb\raisebox{-20pt}{\begin{tikzpicture}
    \draw[thick] (0,1.2)--(0.6,0.6);
    \node[below right] at (0.5,0.65) {\small $\la$};
    \node[above] at (0,1.2) {\small  $\pi$};
    \node[above] at (0.35,0.9){\textcolor{red}{\small  $\b$}};
\end{tikzpicture}}\rb\bigg/ \Gamma_d(\b)
\int_{\pd}\wt\lb\raisebox{-20pt}{\begin{tikzpicture}
        \draw[thick] (0,-1.2)--(0.6,-0.6);
        \node[above right] at (0.5,-0.65) {\small  $\la$};
        \node[above] at (0.3,-0.85) {\textcolor{red}{\small  $\a$}};
        \node[below] at (0,-1.2) {\small  $\k$};
\end{tikzpicture}}\rb\d\mu( \k),
\ee

\be\label{eq:local operator eq right boundary}
\rmUrw\lb \pi| \la; \b, \a \rb=\Gamma_d(\b)
\wt\lb\raisebox{-20pt}{\begin{tikzpicture}
        \draw[thick] (0,-1.2)--(0.6,-0.6);
        \node[above right] at (0.5,-0.65) {\small  $\pi$};
        \node[above] at (0.3,-0.85) {\textcolor{red}{\small  $\a$}};
        \node[below] at (0,-1.2) {\small  $\la$};
\end{tikzpicture}}\rb\bigg/ \Gamma_d(\a)
\int_{\pd}\wt\lb\raisebox{-20pt}{\begin{tikzpicture}
    \draw[thick] (0,1.2)--(0.6,0.6);
    \node[below right] at (0.5,0.65) {\small $\k$};
    \node[above] at (0,1.2) {\small  $\la$};
    \node[above] at (0.35,0.9){\textcolor{red}{\small  $\b$}};
\end{tikzpicture}}\rb\d\mu( \k).
\ee
\end{subequations}
Due to the one-layer Cauchy \eqref{eq:1-layer Cauchy}/Littlewood \eqref{eq:1-layer Littlewood}  identities, they are probability kernels. Note that the bulk kernel~\eqref{eq:local operator eq bulk} coincides with the first-layer bulk kernel ~\eqref{eq:first layer dynamics operator bulk}. The assumption $u+v=0$ guarantees the consistency of the parameters, for all $n,m$:
$$\a_n+\b_m=\theta_n+\theta_m.$$
Consequently, the bulk kernel is consistent with the polymer recurrence \eqref{eq:recurrence inhomogeneous strip} in the following sense: under $\rmUbw\lb \pi| \la,\mu; \a, \b \rb$, 
\be\label{eq: one layer local dynamics encodes polymer recurrence} \pi\stackrel{d}{=}W\star(\la+\mu)\ee
for an independent $W\sim\Wisv(\a_n+\b_m)=\Wisv(\ta_n+\ta_m)$.  Similarly, one can verify that kernels $\rmUlw, \rmUrw$ are consistent with the polymer recurrence at the boundaries.
\end{definition}

In parallel with the two-layer case, we now define the push-block dynamics on the one-layer configurations $\pd^{N+1}$, given by products of delta functions and the local Markov kernels \eqref{eq:local Markov kernels equilibrium}.
\begin{definition}[One-layer inverse-Wishart Markov dynamics] \label{def:Markov kernel strip eq IW}
The one-layer inverse-Wishart Markov dynamics refers to a family of Markov kernels $\{ \rmU^{\hp,\widetilde{\hp}} \}$ on state space $\pd^{N+1}$, where $\hp$ and $\widetilde{\hp}$ differs by a single vertex.
If the update happens at some bulk vertex $\p_j$, the kernel is defined to be:
$$\rmU^{\hp,\widetilde{\hp}}\lb\widetilde{\bl}|\bl\rb = \lb
\prod_{i\neq j} \delta_{\widetilde{\lambda}^{i}=\lambda^{i}} \rb\,\cdot\,\rmUbw(\widetilde{\lambda}^{j}|\lambda^{j-1},\lambda^{j+1};\gamma_j,\gamma_{j+1}).
$$
If the updates happens at some left/right boundary vertex, the kernels are defined respectively as: 
\begin{align*}
    \rmU^{\hp,\widetilde{\hp}}\lb\widetilde{\bl}|\bl\rb &= \lb
\prod_{i\neq 0} \delta_{\widetilde{\lambda}^{i}=\lambda^{i}} \rb\,\cdot\,\rmUlw(\widetilde{\lambda}^{0}|\lambda^{1};\gamma_1,\Tilde{\gamma}_1),\\
\rmU^{\hp,\widetilde{\hp}}\lb\widetilde{\bl}|\bl\rb&= \lb
\prod_{i\neq N} \delta_{\widetilde{\lambda}^{i}=\lambda^{i}} \rb\,\cdot\,\rmUrw(\widetilde{\lambda}^{N}|\lambda^{N-1};\gamma_N,\Tilde{\gamma}_N).
\end{align*}
By compositions of the above kernels, the Markov kernel $\rmU^{\hp,\hq}\lb\widetilde{\bl}|\bl\rb$ is well-defined for arbitrary pair of down-right paths $(\hp, \hq)$ with $\hq$ above $\hp$ and does not depend on the choice of intermediate down-right paths by construction. 
Note that instead of one-layer Gibbs measure, it is the normalized one that is consistent with the dynamics. The difference comes from the fact that the local dynamics at boundary \eqref{eq:local operator eq left boundary} and \eqref{eq:local operator eq right boundary} give some multiplicative contribution to the partition functions. Precisely, for all $\bl'\in \pd^{N+1}$ we have 
\be\label{eq:compatibility of local operator with normalized wt IW full space}   \int_{\pd^{N+1}}\rmU^{\hp,\hq}\lb\bl'|\bl\rb\frac{\wt^{\hp}\lb\bl\rb}{\mathcal{Z}^{\hp}}\d\mu( \bl)
    =\frac{\wt^\hq\lb\bl'\rb}{\mathcal{Z}^{\hq}},
\ee
where for any down-right path $\hp$ with labels $\bg(\hp)=\big(\ga_k(\hp)\big)_{1\leq k\leq N}$ given by \eqref{eq:labels on strip edges equilibrium}, the normalization constant  $\mathcal{Z}^{\hp}$  is defined as
\be\label{eq:partition function equilibrium regime}\mathcal{Z}^{\hp}=\int_{\pd^{N}} \wt^{\hp}\lb \bl\rb \prod_{\substack{i\neq 0}}\d\mu( \lambda^i)= \prod_{k=1}^{N} \Gamma_{d}
\big( \ga_k(\hp)\big). \ee
In particular,  $\mathcal{Z}^{\hp}$ is finite and  does not depend on $\la^0\in\pd$. 
\end{definition}

Taking advantage of the dynamics defined above, we prove that the inverse-Wishart random walk with $\mu$-distributed starting point is consistent with the polymer recurrence. 
\begin{theorem}\label{thm: inhomogeneous equilibrium strip}
Let $\bw\in\{\rightarrow,\downarrow\}^N$ be a word of length $N$, and let parameters $\theta_1,\cdots, \ta_N, u, v\in \RR$ satisfy the equilibrium regime \eqref{eq:parameters inhomogeneous equilibrium regime}.
Consider the inverse-Wishart polymer on a strip with bottom $\hp^{\bw}_{\circ}$ (Definition \ref{def:inhomogeneous strip model})  and initial condition given by
    \be\label{eq:ic theorem inhomogeneous equilibriunm strip} \d\mu\big(B_0\big)\times \d\bR^{\bg(\hp^{\bw}_{\circ}),\bw(\hp^{\bw}_{\circ})}_{ B_0} \Big(\big(B_i\big)_{1\leq i \leq N}\Big).\ee
    Then along any down-right path $\hp=(\p_k)_{1\leq k\leq N}$ above $\hp^{\bw}_{\circ}$, the partition function process is distributed as 
     \be\label{eq: partition function along path inhomogeneous equilibriunm strip} \d\mu\big( Z^{\bta,u,v}(\p_0)\big)\times \d\bR^{\bg(\hp),\bw(\hp)}_{ Z^{\bta,u,v}(\p_0)} \lb\big(Z^{\bta,u,v}(\p_i)\big)_{1\leq i \leq N}\rb.\ee
\end{theorem} 
\begin{proof}
As remarked in \eqref{eq:equilibrium stationary measure is wt hp}, the initial condition \eqref{eq:ic theorem inhomogeneous equilibriunm strip} coincides with the normalized one-layer Gibbs measure 
$$ \frac{1}{\mathcal{Z}^{\hp^{\bw}_{\circ}}}\wt^{\hp^{\bw}_{\circ}}(\bB)\d\mu( \bB)$$ 
which is consistent with the polymer recurrence \eqref{eq:compatibility of local operator with normalized wt IW full space}. 
Therefore, the partition function along any down-right path is distributed as the associated normalized one-layer Gibbs measure, which is in turn a $\Wis^{\pm}$ random walk with $\mu$-distributed starting matrix \eqref{eq: partition function along path inhomogeneous equilibriunm strip}.
\end{proof}

\begin{proof}[Proof of Theorem \ref{thm:homogeneous equilibrium strip}] 
When $\bta=(\ta)^N$, consider the model on a strip with horizontal bottom boundary $\bw(\hp^{\bw}_{\circ})=\lb\rightarrow\rb^N$ and initial condition $\bl_1=(\la^j_1)_{0\leq j \leq N}$ given by  \eqref{eq:stationary measure for the homogneous model on a strip}. By labeling rule \eqref{eq:labels on strip edges equilibrium}, 
 we have  $\bg\lb\tau_k\hp^{\bw}_{\circ}\rb=(\ta-u)^N$ and $\bw\lb\tau_k\hp^{\bw}_{\circ}\rb=(\rightarrow)^N$ for all $k\geq 0$.  By Theorem \ref{thm: inhomogeneous equilibrium strip},  along down-right paths  $\tau_k\hp^{\bw}_{\circ}$, the partition function processes have the same distribution. In other words, 
the $\sigma$-finite measure \eqref{eq:stationary measure for the homogneous model on a strip}
is a stationary measure in the sense precised in Definition \ref{def:homogeneous strip model}.
\end{proof}
Taking advantage of the dynamics introduced above, we now prove that the probability measure given by inverse-Wishart random walk with fixed starting matrix is locally consistent with the polymer recurrence in the following sense.  
 
\begin{proposition}\label{prop:finite stationary measure for equilibrium strip}
Let  $N\in\NN$ and  $\theta_1,\cdots, \ta_N, u, v\in \RR$ be parameters satisfying the equilibrium regime \eqref{eq:parameters inhomogeneous equilibrium regime}. 
    Let   $\bw\in\{\rightarrow,\downarrow\}^N$ be a word of length $N$ and    $S\in\pd$ be some fixed matrix.  
Consider the inverse-Wishart polymer on a strip with bottom boundary $\hp^{\bw}_{\circ}$ (Definition \ref{def:inhomogeneous strip model}) and  initial condition given by 
    \be\label{eq:ic theorem delta inhomogeneous equilibriunm strip}
    \big( B_k \big)_{0\leq k\leq N}\sim
    \delta_{S} \times \d\bR^{\bg(\hp^{\bw}_{\circ}),\bw(\hp^{\bw}_{\circ})}_{ S}  .\ee
    Then along any down-right path $\hq=(\p_k)_{1\leq k\leq N}$ above $\hp^{\bw}_{\circ}$ with starting point $\p_0=(0,0)$, the partition function process is distributed as 
     \be\label{eq: delta partition function along path inhomogeneous equilibriunm strip}  
     \Big( Z^{\bta,u,v}(\p_k)\Big)_{0\leq k\leq N}\sim 
     \delta_{S} \times \d\bR^{\bg(\hp),\bw(\hp)}_{ S}  .\ee
\end{proposition}

\begin{proof}
The initial condition \eqref{eq:ic theorem delta inhomogeneous equilibriunm strip} coincides with the probability measure  given by  the product of delta function at $S$  and the normalized one-layer Gibbs measure:
$$ \delta_{S}( B_0) \times
 \frac{1}{\mathcal{Z}^{\hp^{\bw}_{\circ}} } 
\wt^{\hp^{\bw}_{\circ}}\lb\lb B_k \rb_{1\leq k\leq N}\rb \prod_{k=1}^N \d\mu( B_k).$$ 
Recall that the Markovian kernels defined in Definition~\ref{def: local Markov kernels equilibrium} are given by products of a local Markov kernel and delta functions. Therefore, once the reference point is fixed at the left boundary, the probabilistic component of the  measure above remains consistent with both the bulk kernels~\eqref{eq:local operator eq bulk} and the right boundary kernels~\eqref{eq:local operator eq right boundary}.  That is, \eqref{eq:compatibility of local operator with normalized wt IW full space} holds for $\hp=\hp^{\bw}_{\circ}$ and all down-right paths $\hq$ lying above $\hp^{\bw}_{\circ}$ that share the same starting point. It follows that the partition functions along such down-right paths $\hq$ are distributed according to the corresponding normalized one-layer Gibbs measure, which coincides with the $\Wis^{\pm}$ random walk starting from $S$ \eqref{eq: delta partition function along path inhomogeneous equilibriunm strip}. 
\end{proof}

\subsection{Application to the inhomogeneous model on the quadrant}\label{subsec:Application to inhomogeneous quadrant model}
In this section, we use the stationarity results for the model on a strip in the equilibrium regime (Theorem~\ref{thm: inhomogeneous equilibrium strip} and Proposition~\ref{prop:finite stationary measure for equilibrium strip}) to derive corresponding stationarity properties for the inhomogeneous model on the quadrant with boundary conditions. The main idea is to embed a portion of the quadrant into a suitable strip, then apply the previously constructed Markovian dynamics. In this construction, only the bulk Markov kernels~\eqref{eq:local operator eq bulk} are needed.

\begin{definition}[Inhomogeneous model with boundary conditions]\label{def:inhomogeneous quadrant model with bc}
Let $\ba=(\a_n)_{n\geq 1}$, $\bb=(\b_n)_{n\geq 1}$ be two families of  parameters such that for all $n\geq 1$, $\a_n, \b_n> \frac{d-1}{2}$. Consider a family of independent inverse-Wishart random variables: for all $n,m\geq 1$, 
\begin{align*}
W(n,m)   &\sim \Wisv(\a_n+\b_m).
\end{align*}
 Let $\bB=\lb B_k\rb_{k\in\ZZ}$ be an independent stochastic process on $\pd^{\ZZ}$. The partition functions of inverse-Wishart polymer $\lb Z^{\ba,\bb} (n,m)\rb_{n,m\geq 0}$ are given by the boundary conditions 
$$Z^{\ba,\bb}(\mathbf{b}_k)=B_k, \quad k\in\ZZ, $$ 
and the recurrence: 
\be\label{eq:partition function recurrence inhomogeneous quadrant} Z^{\ba,\bb} (n,m) =W(n,m)\star \lb Z^{\ba,\bb} (n-1,m)+Z^{\ba,\bb} (n,m-1)\rb,\quad n,m\geq 1.\ee
 
\end{definition}

Let us prove an auxiliary lemma which  says that with $\mu$-distributed starting matrix, the Wishart random walk becomes the inverse-Wishart one up to changing the reference point in the following sense.  

\begin{lemma}\label{lem:arbitrary starting point} 
    Let $N\geq 0 $ and $\bg=(\ga_k)_{1\leq k\leq N}$ be parameters such that  $\ga_k>\frac{d-1}{2}$ for each $k$. We have equality between $\sigma$-finite measures on $\pd^{N+1}$:
    $$\d\mu( S_0 )\times\d\bR^{\bg,\rightarrow}_{S_0}
    (S_1,\dots,S_N)= 
    \d\mu( S_N )\times\d\bR^{\bg, \downarrow}_{S_N}
    (S_{N-1},\dots, S_0).$$
\end{lemma}
 
\begin{proof}
    The $\sigma$-finite reference measures on both sides equal  $\prod_{0\leq k\leq N}\d\mu(S_k)$. The density terms match due to the properties of the density of the Wishart and inverse-Wishart distributions: for all $\ta>\frac{d-1}{2}$ and $x,y\in\pd$, 
    $$\P^{+}_{\ta}(y\star x^{-1})=\P^{-}_{\ta}(y^{-1}\star x)=\P^{-}_{\ta}(x\star y^{-1}).$$
\end{proof}
As a corollary, we have the following property for the  stationary boundary conditions $\bB^{\ta,u}_{-M,\mu}$ defined for the homogeneous model on the quadrant (Definition \ref{def:Stationary boundary conditions}).
\begin{corollary}\label{cor:stationary bc reference point}
    The stationary boundary conditions  with $\mu$-distributed reference matrix  does not depend on the reference point, i.e. for all $M\geq 0$, we have equality between $\sigma$-finite measure on $\pd^{\ZZ}$:
    $$\bB^{\ta,u}_{-M,\mu}=\bB^{\ta,u}_{0,\mu}.$$
\end{corollary}
\begin{proof}
By the definition of stationary boundary conditions, the marginal distributions of both measures on $(B_k)_{k < -M}$ and $(B_k)_{k > 0}$ coincide with those of inverse-Wishart random walks starting from $B_{-M}$ and $B_0$, respectively: 
$$(B_{k})_{k<-M}\sim \d\bR^{(\ta+u)^{\NN}, \rightarrow}_{B_{-M}},\qquad
(B_{k})_{k>0}\sim \d\bR^{(\ta-u)^{\NN}, \rightarrow}_{B_{0}}.$$
It therefore suffices to show that their marginals on $(B_k)_{-M \leq k \leq 0}$ are the same. This follows from Lemma~\ref{lem:arbitrary starting point} in the special case $\bg = (\theta+u)^M$. 
\end{proof}

 We are now ready to prove the main goal of this section.
 
\begin{theorem}\label{thm:inhomogeneous quadrant}
Consider the inhomogeneous model on the quadrant with boundary conditions given by  two-sided inverse-Wishart random walk with $\mu$-distributed reference matrix,
\be\label{eq: stationary boundary conditions for the inhomogeneous model on the quadrant}
\d\mu(B_0)\times \d\bR^{\ba,\bb}_{B_0}(\hat{\bB}^{(0)}).\ee
 Then along any down-right path $\hp=(\p_k)_{0\leq k\leq N}$, the joint law of the partition functions $\lb Z^{\ba,\bb}(\p_k)\rb_{0\leq k\leq N}$ is a $\Wis^{\pm}$ random walk with $\mu$-distributed starting matrix, i.e. the $\sigma$-finite measure
 \be\label{eq: desired rw for the inhomogeneous model}
 \d\mu\Big( Z^{\ba,\bb}(\p_0)\Big)\times \d\bR^{\bg(\hp),\bw(\hp)}_{ Z^{\ba,\bb}(\p_0)} \Big((Z^{\ba,\bb}(\p_i))_{1\leq i \leq N}\Big),\ee
where the labels along the path $\bg(\hp)=(\ga_k)_{1\leq k\leq N}$ read   
        \begin{equation}\label{eq:label rules for inhomogeneous quadrant}
            \gamma_i= \left\{  \begin{aligned}
\a_n  \quad & \text{if } \p_{i-1}=(n-1,m), \p_{i}= (n,m) \\
\b_m \quad & \text{if } \p_{i-1}=(n,m), \p_{i}=(n,m-1) 
\end{aligned}
\right., \quad 1\leq i\leq N.
        \end{equation}
\end{theorem}

\begin{proof} 
\begin{figure}
\begin{tikzpicture}[scale=0.87]
 \begin{scope}
    \fill[blue!10] (3,1) rectangle (7,3);
  \end{scope}
\draw[gray] (0,0)--(6,0);
\draw[gray] (1,1)--(3,1);
\draw[gray] (2,2)--(3,2);

\draw[gray] (1,-0.5)--(1,1);
\draw[gray] (2,-0.5)--(2,2);
\draw[gray] (3,-0.5)--(3,3);
\draw[gray] (4,-0.5)--(4,1);
\draw[gray] (5,-0.5)--(5,1);
\draw[gray] (6,-0.5)--(6,1);
\draw (5.5,-0.5)--(10.5,4.5);
\draw (-0.5,-0.5)--(4.5,4.5);

\draw   (6,0)--(10.5,4.5);
\draw   (3,2)--(8,2);
\draw   (3,3)--(9,3);   
\draw   (4,4)--(10,4);

\draw   (4,1)--(4,4);
\draw   (5,1)--(5,4.5);
\draw   (6,1)--(6,4.5);
\draw   (7,1)--(7,4.5); 
\draw   (8,2)--(8,4.5); 
\draw   (9,3)--(9,4.5);
\draw   (10,4)--(10,4.5);
\draw[dotted]  (1.5,-0.5)--(6.5,4.5);
\draw[dotted]  (7.5,-0.5)--(12.5,4.5);

\draw[dotted]  (6,0)--(8,0);
\draw[dotted]  (7,1)--(9,1);
\draw[dotted]  (8,2)--(10,2);
\draw[dotted]  (9,3)--(11,3);   
\draw[dotted]  (10,4)--(12,4); 
\draw[dotted]  (7,-0.5)--(7,4.5);
\draw[dotted]  (8,0)--(8,4.5); 
\draw[dotted]  (9,1)--(9,4.5);
\draw[dotted]  (10,2)--(10,4.5);
\draw[dotted]  (11,3)--(11,4.5);
\draw[dotted]  (12,4)--(12,4.5);
\draw[dotted]  (3,4)--(4,4);
\draw[dotted]  (10,2)--(12,2);
\draw[dotted]  (11,3)--(12,3);
\draw[dotted]  (10,1)--(10,2);
\draw[dotted]  (11,1)--(11,3);
\draw[dotted]  (12,1)--(12,4);

\draw[very thick] (3,3)--(3,1)--(7,1);
\fill (3,3) circle(0.08);
\node[left] at (3,3) {\small $(0,m)=\p_0$};
\fill (3,2) circle(0.08);
\fill (3,1) circle(0.08);
\node[below] at (3,1) {\small $(0,0)$};
\fill (4,1) circle(0.08);
\fill (5,1) circle(0.08);
\fill (6,1) circle(0.08);
\fill (7,1) circle(0.08);
\node[below right] at (7,1) {\small $\p_N=(n,0)$};

\draw[very thick,orange] (3,3)--(5,3)--(5,2)--(7,2)--(7,1);
\fill (3,2) circle(0.08);
\fill (3,1) circle(0.08);
\fill (4,1) circle(0.08);
\fill (5,1) circle(0.08);
\fill (6,1) circle(0.08);
\fill (7,1) circle(0.08);

\fill[blue!65] (4,3) circle(0.08);
\fill[blue!65] (4,2) circle(0.08);
\fill[blue!65] (5,3) circle(0.08);
\fill[blue!65] (5,2) circle(0.08);
\fill[blue!65] (6,3) circle(0.08);
\fill[blue!65] (6,2) circle(0.08);
\fill[blue!65] (7,3) circle(0.08);
\fill[blue!65] (7,2) circle(0.08);
\draw[->,thick] (3,3) -- (3,4.5);
\draw[->,thick] (7,1) -- (12.5,1);
\fill (3,4) circle(0.06);
\fill (8,1) circle(0.06);
\fill (9,1) circle(0.06);
\fill (10,1) circle(0.06);
\fill (11,1) circle(0.06);
\fill (12,1) circle(0.06);
\node at (3.5,1) {\textcolor{red}{\small $\a_1$}};
\node at (4.5,1) {\textcolor{red}{\small $\a_2$}};
\node at (5.5,1) {\textcolor{red}{\small $\a_3$}};
\node at (6.5,1) {\textcolor{red}{\small $\a_4$}};

\node at (3.5,2) {\textcolor{red}{\small $\a_1$}};
\node at (4.5,2) {\textcolor{red}{\small $\a_2$}};
\node at (5.5,2){\textcolor{red}{\small $\a_3$}};
\node at (6.5,2) {\textcolor{red}{\small $\a_4$}};

\node at (3.5,3) {\textcolor{red}{\small $\a_1$}};
\node at (4.5,3) {\textcolor{red}{\small $\a_2$}};
\node at (5.5,3) {\textcolor{red}{\small $\a_3$}};
\node at (6.5,3) {\textcolor{red}{\small $\a_4$}};

\node at (3,1.5) {\textcolor{red}{\small $\b_1$}};
\node at (4,1.5) {\textcolor{red}{\small $\b_1$}};
\node at (5,1.5) {\textcolor{red}{\small $\b_1$}};
\node at (6,1.5) {\textcolor{red}{\small $\b_1$}};
\node at (7,1.5) {\textcolor{red}{\small $\b_1$}};

\node at (3,2.5){\textcolor{red}{\small $\b_2$}};
\node at (4,2.5) {\textcolor{red}{\small $\b_2$}};
\node at (5,2.5) {\textcolor{red}{\small $\b_2$}};
\node at (6,2.5) {\textcolor{red}{\small $\b_2$}};
\node at (7,2.5){\textcolor{red}{\small $\b_2$}};
\end{tikzpicture}
    \centering
    \caption{In the coordinate system provided by  quadrant $\ZZ^2_{\geq 0}$ (dashed lattice), in order to understand the partition function along a down-right path  from $(0,m)$ to $(n,0)$ (in orange), we consider the model on a well-chosen up-translated strip $\widetilde{\strip^{\bw}}$ (solid line). Inside the blue rectangle $R_{nm}$, the partition functions at the blue vertices satisfy the same recurrence in both models.}
    \label{fig:embed quadrant into the strip}
\end{figure}
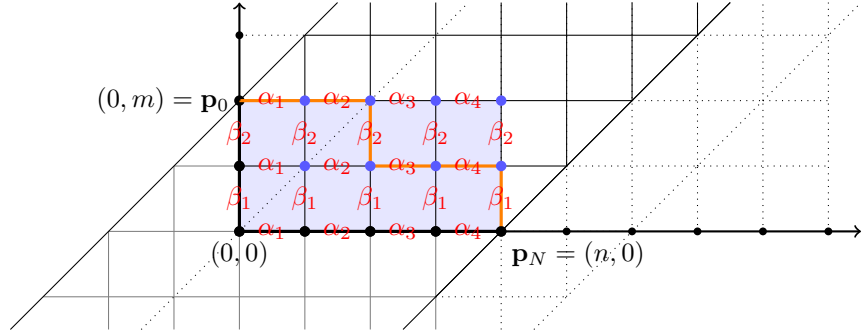

\phantomsection
\label{Step1}
\noindent\textbf{Step 1: Embedding into a strip.}
Fix an arbitrary $u\in\RR$ such that 
$$0\leq2u< \min\lb\min_{1\leq i\leq n} (\a_i),   \min_{1\leq j\leq m}(\b_j) \rb -\frac{d-1}{2}.$$
Define parameters $\bta=(\ta_k)_{1\leq k\leq N}\in\RR^N $ and $  v \in\RR$ by:
$$\ta_i=\a_i+u,\ 1\leq i\leq n; \quad \ta_{n+j}=\b_j-u, \ 1\leq j\leq m; \quad v=-u. $$
One verifies the parameters $\bta, u, v$ satisfy the hypothesis on parameters for the strip model in the equilibrium regime \eqref{eq:parameters inhomogeneous equilibrium regime}. 

Suppose that  the down-right path $\hp=(\p_k)_{0\leq k\leq N}$ has starting point $\p_0=(0,m)$ and ending point $(n,0)$ for some $n,m\geq 0$ with $n+m=N$. 
Consider the inverse-Wishart polymer model $Z^{\bta,u,v}$ on a strip $\strip^{\bw}$ \eqref{eq:strip with bottom}, with bottom shape $\bw=(\downarrow)^{m}\times (\rightarrow)^n$ and labels on edges given by rule \eqref{eq:labels on strip edges equilibrium}. Let us overlap the strip with the coordinate system given by the quadrant $\ZZ^2_{\geq 0}\subset \ZZ^2$. Consider vertical translation 
$v_{m}:\ZZ^2\rightarrow\ZZ^2, \ x\mapsto x+(0,m)$. Denote by
$\widetilde{\strip^{\bw}}$ the image of $\strip^{\bw}$ under $v_m$, together with its bottom $\hp^{\bw}_{\circ}$  and  labels on the edges given by rule  \eqref{eq:labels on strip edges equilibrium}. 
Let
$$R_{nm}=\left\{ (i,j):0\leq i\leq n, 0\leq j\leq m \right\}\subseteq \widetilde{\strip^{\bw}}  \cap \ZZ^2_{\geq 0} $$
be the rectangle  with vertices $(0,0), (0,m), (n,0), (n,m)$.
Note that on the edges inside $R_{nm}$, the labels for the translated strip $\widetilde{\strip^{\bw}}$  equal those for the quadrant $\ZZ^2_{\geq 0}$ given by rule \eqref{eq:label rules for inhomogeneous quadrant}. Therefore, the recurrence satisfied by partition functions indexed by vertices 
$\left\{ (i,j): 1\leq i\leq n, 1\leq j\leq m \right\}$ 
given by \eqref{eq:recurrence inhomogeneous strip} and \eqref{eq:partition function recurrence inhomogeneous quadrant} coincides. As a consequence, 
the partition functions inside $R_{nM}$ are the same for both model:
$$\lb Z^{\bta,u,v}(i,j)\rb_{(i,j)\in R_{nm}}=\lb Z^{\ba,\bb}(i,j)\rb_{(i,j)\in R_{nm}},$$
where the initial condition for the model on a strip are given by  the boundary conditions for the model on a quadrant restricted to the bottom of $\widetilde{\strip^{\bw}}$.
See Figure \ref{fig:embed quadrant into the strip} for illustration. 

\noindent\textbf{Step 2: Identify the boundary conditions.}  
The boundary conditions \eqref{eq: stationary boundary conditions for the inhomogeneous model on the quadrant} restricted to $v_m\hp^{\bw}_{\circ}$ has joint distribution 
$$\d\mu\big(B_0\big)
\times \d\bR^{(\ta+u)^m,\rightarrow}_{B_0}\Big( (B_{-k})_{1\leq k\leq m}\Big)
\times \d\bR^{(\ta-u)^n,\rightarrow}_{B_0}\Big( (B_k)_{1\leq k\leq n}\Big).$$
By  Lemma \ref{lem:arbitrary starting point}, it equals the inverse-Wishart random walk with $\mu$-distributed starting matrix 
$$\d\mu\big(B_{-m}\big)
\times \d\bR^{\bg,\rightarrow}_{B_{-m}}\Big( (B_{k})_{-m< k\leq n}\Big),$$
where 
$\bg=\big( \b_{m}, \cdots,\b_1,\a_1, \cdots,\a_n  \big)= 
\bg(v_m\hp^{\bw}_{\circ})$ given by \eqref{eq:labels on strip edges equilibrium}.
By applying Theorem \ref{thm: inhomogeneous equilibrium strip} for the model on the translated strip $\widetilde{\strip^{\bw}}$, the partition functions along $\hp$ is the inverse-Wishart random walk with $\mu$-distributed starting matrix distributed as \eqref{eq: desired rw for the inhomogeneous model}.
\end{proof}

In the special case $\ba=(\ta-u)^{\NN}, \bb=(\ta+u)^{\NN}$, the inhomogeneous model reduces to the homogeneous one (Definition \ref{def:quadrant model with bc}). Let us prove that the two-sided inverse-Wishart random walk is stationary for the model in the sense precised in Definition \ref{def:quadrant model with bc}.  We note that the stationarity in this case is obtained in \cite[Eq. (26) (27)]{KrajenbrinkLeDoussal} using a different method.  

\begin{proof}[Proof of Theorem \ref{thm:quadrant stationary measure}]
By Corollary \ref{cor:stationary bc reference point}, it suffices to prove that for the homogeneous model with boundary conditions 
$$\bB^{\ta,u}_{0,\mu}=\d\mu(B_0)\times \d\bR^{\ta-u,\ta+u}_{B_{0}}(\hat{\bB}^{(0)}),$$
the partition function process along the translated boundary is distributed as 
\be\label{eq:desired distribution}
\d\mu\Big( Z^{\ba,\bb}(\tau_1\mathbf{b}_0) \Big)
\times \d\bR^{\ta-u,\ta+u}_{Z^{\ba,\bb}(\tau_1\mathbf{b}_0)}
\Big( Z^{\ba,\bb}(\tau_1\mathbf{b}_k)_{k\neq0}\Big).\ee 
By Theorem \ref{thm:inhomogeneous quadrant}, for all $n,m\geq 0$, along the down-right path 
$$\hp=(\p_k)_{0\leq k\leq m+n+2}:(0,m)\rightarrow (1,m)\rightarrow(1,1)\rightarrow(n,1)\rightarrow(n,0),$$
the partition function process is distributed as \eqref{eq: desired rw for the inhomogeneous model}. By Lemma \ref{lem:arbitrary starting point}, its marginal on $(\p_k)_{1\leq k\leq n+m+1}= (\tau_1\mathbf{b}_k)_{-m\leq k\leq n}$ is distributed as 
$$\d\mu\Big( Z^{\ba,\bb}(\tau_1\mathbf{b}_0) \Big)
\times \d\bR^{(\ta+u)^m,\rightarrow}_{Z^{\ba,\bb}(\tau_1\mathbf{b}_0)}
\Big( Z^{\ba,\bb}(\tau_1\mathbf{b}_k)_{-m\leq k\leq -1}\Big)
\times \d\bR^{(\ta-u)^n,\rightarrow}_{Z^{\ba,\bb}(\tau_1\mathbf{b}_0)}
\Big( Z^{\ba,\bb}(\tau_1\mathbf{b}_k)_{1\leq k\leq n}\Big).$$
For every realization of $Z^{\ba,\bb}(\tau_1\mathbf{b}_0)\in\pd$, by Ionescu-Tulcea extension theorem \cite[Theorem 14.32]{Klenke}, $\d\bR^{\ta-u,\ta+u}_{Z^{\ba,\bb}(\tau_1\mathbf{b}_0)} $ is the unique probability measure with marginals 
$ \d\bR^{(\ta+u)^m,\rightarrow}_{Z^{\ba,\bb}(\tau_1\mathbf{b}_0)} \times \d\bR^{(\ta-u)^n,\rightarrow}_{Z^{\ba,\bb}(\tau_1\mathbf{b}_0)}$ for all $n,m\geq 0$. Hence the partition functions along $\tau_1\hp_{b}$ is distributed as \eqref{eq:desired distribution}.
\end{proof}

For general boundary conditions, the product structure of the increment process propagates into the bulk points that are below the reference point. 

\begin{proposition}\label{prop:finite stationary measure for quadrant with boudnary} Fix $M \geq 0$ and $S\in\pd$. Consider the model on the quadrant with boundary conditions such that the marginal distribution satisfies 
   $$\big( B_k \big)_{k\geq -M}\sim
    \delta_{S} \times \d\bR^{\bg ,\bw }_{ S},$$
   where $\bg=(\b_m,\cdots,\b_1,\a_1,\cdots)$, $\bw=(\downarrow)^m\times(\rightarrow)^{\NN}$.  
    Then the partition functions along any down-right path $\hp=(\p_k)_{0\leq k\leq N}$ starting with $\p_0=(0,M)$  is a  $\Wis^{\pm}$ random walk starting from $S$, that is, the probability measure 
        $$\delta_{S}\Big( Z^{\ba,\bb}(\p_0)\Big)\times \d\bR^{\bg(\hp),\bw(\hp)}_{ Z^{\ba,\bb}(\p_0)} \Big((Z^{\ba,\bb}(\p_i))_{1\leq i \leq N}\Big).$$
        where the labels $\bg(\hp)=(\ga_k)_{1\leq k\leq N}$ along the path are given by \eqref{eq:label rules for inhomogeneous quadrant}, 
\end{proposition}

\begin{proof} 
The proof strategy is exactly the same as that in the proof of Theorem \ref{thm:inhomogeneous quadrant}. 
 For any down-right path $\hp=(\p_k)_{0\leq k\leq N}$ starting at $\p_0=(0,M)$ ending at $\p_N=(n,0)$ with $N=M+n$, 
we consider the same embedding as done in the \hyperref[Step1]{Step 1} of proof of Theorem \ref{thm:inhomogeneous quadrant}. 
The partition functions along a down-right path on $\ZZ^2_{\geq0}$ from $(0,M)$ to $(n,0)$ are distributed as an inverse-Wishart random walk starting from $B_{-M}=S$. By applying Proposition \ref{prop:finite stationary measure for equilibrium strip} for the translated strip $\widetilde{\strip^{\bw}}$, the partition functions along $\hp$ is again an inverse-Wishart random walk starting from $S$, as desired.
\end{proof}

In particular, when the environment is homogeneous $\ba=(\ta-u)^{\NN}, \bb=(\ta+u)^{\NN}$, let us prove Proposition \ref{prop:ratio homogeneous quadrant}.

\begin{proof}[Proof of Proposition \ref{prop:ratio homogeneous quadrant}]
Note that since the conclusion in Proposition \ref{prop:finite stationary measure for quadrant with boudnary}  holds for $\Wis^{\pm}$ random walk with arbitrary starting matrix $S \in \pd$, it remains valid for random  $S$ which is independent of the increments.  The first part of Proposition \ref{prop:ratio homogeneous quadrant} is hence the special case $\ba=(\theta+u)^{\mathbb{N}}, \bb=(\theta-u)^{\mathbb{N}}$ of Proposition \ref{prop:finite stationary measure for quadrant with boudnary}.  

For the second part,  consider two down-right paths inside $R_{nM}$: 
$$\hp=(\p_k)_{0\leq k\leq N}: (0,M)\rightarrow (n,M)\rightarrow (n,0);\quad \hq=(\q_k)_{0\leq k\leq N}:(0,M)\rightarrow (0,0)\rightarrow(n,0).$$
The joint probability distribution of partition functions along the paths are given by Proposition \ref{prop:finite stationary measure for quadrant with boudnary}, we have 
\begin{subequations}
    \begin{multline}\label{eq:log differences 1}
    \EE\Big[\log\abs{Z^{\ba,\bb}(n,m)}\Big]-\EE\Big[\log\abs{Z^{\ba,\bb}(0,M)}\Big]=\sum_{k=1}^{n+M-m} \lb\EE\Big[\log\abs{Z^{\ba,\bb}(\p_k)}\Big]-\EE\Big[\log\abs{Z^{\ba,\bb}(\p_{k-1})}\Big]\rb\\
        =  n \EE\Big[\log\abs{\Wisv(\theta-u)}\Big]+(M-m)\EE\Big[\log\abs{\Wis(\theta+u)}\Big], 
\end{multline}
\begin{multline}\label{eq:log differences 2}
    \EE\Big[\log\abs{Z^{\ba,\bb}(0,0)}\Big]-\EE\Big[\log\abs{Z^{\ba,\bb}(0,M)}\Big]=\sum_{k=1}^{M} \lb\EE\Big[\log\abs{Z^{\ba,\bb}(\q_k)}\Big]-\EE\Big[\log\abs{Z^{\ba,\bb}(\q_{k-1})}\Big]\rb\\
        =   M\EE\Big[\log\abs{\Wis(\theta+u)}\Big].
\end{multline}
\end{subequations}
Subtracting \eqref{eq:log differences 2} from \eqref{eq:log differences 1} yields the desired identity \eqref{eq:expectation of partition function}.
\end{proof}

\appendix
\section{Law of large numbers for the inverse-gamma polymer on a strip}\label{appendix:LLN_d1}
In this Appendix, we prove the $d=1$ case of Conjecture \ref{conj:LLN strip}:
\begin{theorem}\label{thm:LLN strip d=1}
    For $d=1$ and parameters $\ta, u,v$ in either the maximal current regime \eqref{eq:parameters homogeneous general regime} or the equilibrium regime \eqref{eq:parameters homogeneous equilibrium regime}, 
    consider the model on a strip of arbitrary width $N\geq 1$, with initial condition given by an arbitrary probability measure on ${(\RR_{>0})}^{N+1}$. Then we have the following convergence:
    \be
    \frac{\log Z^{\ta,u,v}(n,n)}{n}\xrightarrow{n\rightarrow\infty}
f^{\ta,u,v}\qquad \text{a.s.} \ee
for some (explicit) deterministic constant $f^{\ta,u,v}$. Moreover, in the equilibrium regime, we have
    \be
f^{\ta,u,-u}=-\psi(\ta-u)-\psi(\ta+u).  \ee  
\end{theorem} 

Let us first review some known results for the model.
When $d=1$, the inverse-Wishart polymer on a strip (Definition \ref{def:homogeneous strip model}) reduces to the log-gamma polymer on a strip defined in \cite{BarraquandCorwinYang}. 
For $k\geq 0$ and $1\leq l\leq N$, define 
$$\mathbf{H}_{k}(l):=Z^{\ta,u,v}(k+l,k)/Z^{\ta,u,v}(k,k),\qquad \mathbf{V}_{k}(l):=Z^{\ta,u,v}(k,k)/Z^{\ta,u,v}(k,k-l) $$
to be the horizontal and (reversely ordered) vertical increment process, respectively.
Under the recurrence of partition functions, $\big(\mathbf{H}_{k}\big)_{k\geq0}$ and $\big(\mathbf{V}_{k}\big)_{k\geq0}$ are Markov processes on $ \RR_{>0}^{N}$ which does not depend on the value of $Z^{\ta,u,v}(k,k)$'s. 

The two-layer Whittaker process $\PP^{\ga,u,v}_{\la^0_1}\big( \la_2^0, (\la^i)_{1\leq i \leq N}\big)$ starting from $\la^0_1$ ($d=1$ case of Definition \ref{def: homogeneous two layer matrix whittaker process}) is a probability measure on $\RR_{>0}\times \lb \RR_{>0}^{2}\rb^{N}$.  We denote its first-layer marginal by $\bP^{\ga,u,v}_{\la^0_1} \big( \bl_1 \big)$. In the special case $u+v=0$, the law of $\bl_1$ reduce to the inverse-gamma random walk \cite[Eq. ~(7)]{BarraquandCorwinYang}, 
\be\label{eq:stationary measure d=1, u+v=0}\bP^{\ga,u,-u}_{\la^0_1}=\bR^{(\ta-u)^{N},\rightarrow}_{\la^0_1}.\ee 
In general, we have the following.
\begin{proposition}[{\cite[Theorem 1.6]{BarraquandCorwinYang}}]\label{prop:ergodicity measure log gamma strip}
$\bP^{\ga,u,v}_{1} $ is the unique ergodic stationary measure of the process $(\mathbf{H}_{k})_{k\geq0}$.
\end{proposition}
We are now ready to prove Theorem \ref{thm:LLN strip d=1}. The proof is based on Birkhoff's ergodic theorem and, in particular, avoids using the subadditivity of the partition functions of the polymer, which is the standard tool to prove convergence for the model on the quadrant.
\begin{proof}[Proof of Theorem \ref{thm:LLN strip d=1}]
Consider the function $g: \RR_{>0}^{N}\longrightarrow \RR$ defined by $ \mathbf{X}\mapsto \log\big( \mathbf{X}^{(1)}\big)$ where $ \mathbf{X}^{(1)}$ denotes the first coordinate of $ \mathbf{X}$. By Proposition \ref{prop:ergodicity measure log gamma strip} and Birkhoff's ergodic theorem (see, for example, \cite[Theorem 4.50]{BenaimHurth}), we have 
$$\frac{1}{n} \sum_{k=0}^{n-1} g(\mathbf{H}_{k})\xrightarrow{n\rightarrow\infty}\EE_{\bP^{\ta,u,v}_{1}}\Big[ g(\mathbf{H}) \Big] \qquad \text{a.s.}$$
In other words,
\be\label{eq:cv of horizontal process}\frac{1}{n} \sum_{k=0}^{n-1} \log\Big( \mathbf{H}_{k}(1)\Big)\xrightarrow{n\rightarrow\infty}\EE_{\bP^{\ta,u,v}_{1}}\Big[ \log \big(\mathbf{H}^{(1)}\big) \Big] \qquad \text{a.s.}\ee
By symmetry of the model, the Markov process $(\mathbf{V}_{k})_{k\geq0}$ admits the same stationary measure as $(\mathbf{H}_{k})_{k\geq 0}$, provided one exchanges the role of $u$ and $v$. Hence, analogously,
\be\label{eq:cv of vertical process}\frac{1}{n} \sum_{k=0}^{n-1} \Big(\mathbf{V}_{k}(1)\Big)\xrightarrow{n\rightarrow\infty}\EE_{\bP^{\ta,v,u}_{1}}\Big[ \log\big( \mathbf{V}^{(1)} \big)\Big] \qquad \text{a.s.}\ee
Note that the partition function can be decomposed as
\begin{multline*}
Z^{\ta,u,v}(n,n)
= Z^{\ta,u,v}(0,0)\prod_{k=0}^{n-1} \frac{Z^{\ta,u,v}(k+1,k)}{Z^{\ta,u,v}(k,k)}\prod_{k=0}^{n-1} \frac{Z^{\ta,u,v}(k+1,k+1)}{Z^{\ta,u,v}(k+1,k)}\\
=Z^{\ta,u,v}(0,0)\prod_{k=0}^{n-1} \mathbf{H}_{k}(1) \prod_{k=0}^{n-1} \mathbf{V}_{k+1}(1).
\end{multline*}
Taking logarithm of both sides, combining with \eqref{eq:cv of horizontal process} and \eqref{eq:cv of vertical process}, we obtain  
\be\label{eq:strip free energy limit d=1}\frac{\log Z^{\ta,u,v}(n,n)}{n}\xrightarrow{a.s.}
f^{\ta,u,v}
= \EE_{\bP^{\ta,u,v}_{1}}\Big[ \log \big(\mathbf{H}^{(1)}\big) \Big] +\EE_{\bP^{\ta,v,u}_{1}}\Big[ \log\big( \mathbf{V}^{(1)} \big)\Big].  \ee

When $u+v=0$, the stationary measure simplifies significantly (see \eqref{eq:stationary measure d=1, u+v=0}). In particular, the increments are i.i.d. inverse-gamma distributed. As a consequence,
\begin{align*}
f^{\ta,u,-u}
&= \EE_{\bP^{\ta,u,-u}_{1}}\Big[ \log \big(\mathbf{H}^{(1)}\big) \Big] +\EE_{\bP^{\ta,-u,u}_{1}}\Big[ \log\big( \mathbf{V}^{(1)} \big)\Big] \\
&=\EE\Big[ \log \Big(\Gammainv(\a-u)\Big) \Big] +\EE\Big[ \log\Big( \Gammainv(\a+u) \Big) \Big]\\ 
&= -\psi(\ta-u)-\psi(\ta+u), 
\end{align*}
where we recall that $\psi$ denotes the digamma function.
\end{proof}
  
\begin{remark}
Under the stationary measure $\PP^{\ga,u,v}_{1}$ for general $u,v$,  
the multi-points characteristic function can be expressed explicitly in terms of contour integrals involving Gamma functions \cite[Theorem~1.11]{Barraquand}.  
Consequently, the limit $f^{\ta,u,v}$ appearing in \eqref{eq:strip free energy limit d=1} can be further identified through a contour integral representation involving Gamma and digamma functions.
\end{remark}

\renewcommand{\emph}[1]{\textit{#1}}

\newrefcontext[sorting=nyt]
\printbibliography

\end{document}